\documentclass[12pt,reqno]{NumPDEsArticle}

\usepackage[utf8]{inputenc}
\usepackage[T1]{fontenc}
\usepackage[english]{babel}
\usepackage{amsmath, amssymb, amsthm}
\usepackage{enumitem}
\usepackage{moreverb}
\usepackage[activate={true,nocompatibility},
    final,
    kerning=true,
    spacing=true,
    factor=1100,
    stretch=10,
    shrink=10]{microtype}
\usepackage{orcidlink}
\usepackage{mathtools}
\usepackage{subfig}
\title[On optimal complexity of AFEM with inexact solver]{On full linear convergence and optimal complexity of adaptive FEM with inexact solver}

\keywords{adaptive finite element method, optimal convergence rates, cost-optimality, inexact
solver, full linear convergence}

\subjclass[2020]{41A25, 65N15, 65N30, 65N50, 65Y20}

\author{Philipp Bringmann\orcidlink{0000-0002-4546-5165}}
\author{Michael Feischl\orcidlink{0000-0001-7206-1652}}
\author{Ani Miraçi\orcidlink{0000-0003-4962-9662}}
\author{Dirk Praetorius\orcidlink{0000-0002-1977-9830}}
\author{Julian Streitberger\orcidlink{0000-0003-1189-0611}}

\email{philipp.bringmann@asc.tuwien.ac.at}
\email{michael.feischl@asc.tuwien.ac.at}
\email{ani.miraci@asc.tuwien.ac.at}
\email{dirk.praetorius@asc.tuwien.ac.at (corresponding author)}
\email{julian.streitberger@asc.tuwien.ac.at}

\thanks{This research was funded in whole or in part by the Austrian Science Fund (FWF)
[\href{https://www.fwf.ac.at/en/research-radar/10.55776/F65}{10.55776/F65},
\href{https://www.fwf.ac.at/en/research-radar/10.55776/I6802}{10.55776/I6802}, and 
\href{https://www.fwf.ac.at/en/research-radar/10.55776/P33216}{10.55776/P33216}].
For open access purposes, the author has applied a CC BY public copyright license 
to any author accepted manuscript version arising from this submission. Additionally, the Vienna 
	School of Mathematics supports Julian Streitberger.}

\begin{document}
\maketitle

\begin{abstract}
	The ultimate goal of any numerical scheme for partial differential equations (PDEs) is to compute an approximation of user-prescribed accuracy at quasi-minimal
	computation time. To this end, algorithmically, the standard adaptive finite element method (AFEM) integrates an inexact solver and nested iterations with discerning stopping criteria balancing the different error components.
	The analysis ensuring optimal convergence order of AFEM with respect to the overall computational cost critically hinges on the concept of R-linear convergence of a suitable quasi-error quantity.
	This work tackles several shortcomings of previous approaches by introducing a new proof strategy.
	Previously, the analysis of the algorithm required several parameters to be fine-tuned.
	This work leaves the classical reasoning and introduces a summability criterion for R-linear convergence to remove restrictions on those parameters.
	Second, the usual assumption of a (quasi-)Pythagorean identity is replaced by the generalized notion of quasi-orthogonality from [Feischl, Math.~Comp., 91 (2022)]. Importantly, this paves the way towards extending the analysis
	of AFEM with inexact solver to general inf-sup stable problems beyond the energy minimization setting.
	Numerical experiments investigate the choice of the adaptivity parameters.
\end{abstract}

%%%%%%%%%%%%%%%%%%%%%%%%%%%%%%%%%%%%%%%%%%%%%%%%%%%%%%%%%%%%%%%%%%%%%%%%%%%%%%%%%%%
%%%%%%%%%%%%%%%%%%%%%%%%%%%%%%%%%%%%%%%%%%%%%%%%%%%%%%%%%%%%%%%%%%%%%%%%%%%%%%%%%%%
\section{Introduction}
\label{section:introduction}
%%%%%%%%%%%%%%%%%%%%%%%%%%%%%%%%%%%%%%%%%%%%%%%%%%%%%%%%%%%%%%%%%%%%%%%%%%%%%%%%%%%
%%%%%%%%%%%%%%%%%%%%%%%%%%%%%%%%%%%%%%%%%%%%%%%%%%%%%%%%%%%%%%%%%%%%%%%%%%%%%%%%%%%

Over the past three decades, the mathematical understanding of adaptive finite
element methods (AFEMs) has matured; see, e.g.,~\cite{doerfler1996, mns2000,
	bdd2004, stevenson2007, ckns2008, cn2012, ffp2014} for linear elliptic PDEs,
\cite{veeser2002, dk2008, bdk2012, gmz2012} for certain nonlinear PDEs,
and~\cite{cfpp2014} for an axiomatic framework summarizing the
earlier references. In most of the cited works, the focus is on
\textsl{(plain) convergence} in~\cite{doerfler1996, mns2000, veeser2002, dk2008, gmz2012} and
optimal convergence rates with respect to the number of degrees of
freedom, i.e., \textsl{optimal rates},in~\cite{bdd2004, ckns2008, cn2012, bdk2012, gmz2012, ffp2014}.

The adaptive feedback loop strives to approximate the unknown and
possibly
singular exact PDE solution $u^\star$ on the basis of
\textsl{a~posteriori} error estimators and adaptive
mesh refinement strategies. Employing AFEM with \emph{exact solver}, detailed in
Algorithm~\ref{algorithm:exact} below,
generates a sequence $(\mathcal{T}_\ell)_{\ell \in \mathbb{N}_0}$ of successively refined
meshes together with the corresponding finite element solutions
$u_\ell^\star \approx u^\star$ and error estimators
$\eta_\ell(u_\ell^\star)$
by iterating
\begin{equation}\label{eq:semr:exact}
	\boxed{~~\mathtt{solve}~~}
	\quad \longrightarrow \quad
	\boxed{~~\mathtt{estimate}~~}
	\quad \longrightarrow \quad
	\boxed{~~\mathtt{mark}~~}
	\quad \longrightarrow \quad
	\boxed{~~\mathtt{refine}~~}
\end{equation}
A key argument in the analysis of~\eqref{eq:semr:exact} in~\cite{ckns2008} and succeeding works for symmetric PDEs consists in showing \textsl{linear convergence} of the quasi-error
\begin{equation}\label{eq:intro:ckns}
	\Delta_\ell^\star
	\le q_{\textup{lin}} \, \Delta_{\ell-1}^\star
	\quad
	\text{with}
	\quad
	\Delta_\ell^\star \coloneqq \big[ |\mkern-1.5mu|\mkern-1.5mu| u^\star - u_\ell^\star |\mkern-1.5mu|\mkern-1.5mu|^2 +
		\gamma \, \eta_\ell(u_\ell^\star)^2 \big]^{1/2}
	\quad \text{for all } \ell \in \mathbb{N},
\end{equation}
where $0 < q_{\textup{lin}}, \gamma < 1$ depend only on the problem setting and the marking parameter~$\theta$.
Here, $|\mkern-1.5mu|\mkern-1.5mu| \, \cdot \, |\mkern-1.5mu|\mkern-1.5mu|$ is the PDE-induced energy norm providing a Pythagorean identity of the form
\begin{equation}\label{eq:pythagoras}
	|\mkern-1.5mu|\mkern-1.5mu| u^\star - u_{\ell+1}^\star |\mkern-1.5mu|\mkern-1.5mu|^2
	+ |\mkern-1.5mu|\mkern-1.5mu| u_{\ell+1}^\star - u_\ell^\star |\mkern-1.5mu|\mkern-1.5mu|^2
	=
	|\mkern-1.5mu|\mkern-1.5mu| u^\star - u_\ell^\star |\mkern-1.5mu|\mkern-1.5mu|^2
	\quad \text{for all } \ell \in \mathbb{N}_0.
\end{equation}
The work \cite{cfpp2014} showed that a \textsl{tail-summability} of the estimator sequence
\begin{equation*}
	\sum_{\ell' = \ell + 1}^\infty \eta_{\ell'}(u_{\ell'}^\star)
	\le C_{\textup{lin}}' \eta_\ell(u_\ell^\star)
	\quad \text{for all } \ell \in \mathbb{N}_0
\end{equation*}
or, equivalently, \textsl{R-linear convergence}
\begin{equation}\label{eq:intro:axioms}
	\eta_\ell(u_\ell^\star)
	\le C_{\textup{lin}} q_{\textup{lin}}^{\ell-\ell'} \, \eta_{\ell'}(u_{\ell'}^\star)
	\quad \text{for all } \ell \ge \ell' \ge 0,
\end{equation}
with $0 < q_{\textup{lin}} < 1$ and $C_{\textup{lin}}, C_{\textup{lin}}' > 0$, suffices to prove convergence.
An extension of the analysis to nonsymmetric linear PDEs can be done by relaxing the Pythagorean identity to a quasi-Pythagorean estimate in~\cite{cn2012, ffp2014, bhp2017}. However, this comes at the expense that either the initial mesh has to be sufficiently fine as in~\cite{cn2012}, \eqref{eq:intro:ckns} only holds for \(\ell \ge \ell_0 \in \mathbb{N}_0\) \cite{bhp2017}, or~\eqref{eq:intro:ckns} holds in the general form~\eqref{eq:intro:axioms} below, where the constants depend on the adaptively generated meshes in~\cite{ffp2014}.
Additional to R-linear convergence~\eqref{eq:intro:axioms}, a sufficiently small marking parameter $\theta$ leads to optimal rates in the sense of~\cite{stevenson2007, ckns2008}.
This can be stated in terms of approximation classes from~\cite{bdd2004, stevenson2008, ckns2008} by mathe\-ma\-tically guaranteeing the largest possible convergence rate
$s > 0$ with
\begin{equation}
	\sup_{\ell \in \mathbb{N}} (\#\mathcal{T}_\ell)^s \eta_\ell(u_\ell^\star) < \infty.
\end{equation}

However, due to the incremental nature of adaptivity,
the mathematical question on optimal convergence rates should
rather refer to the overall computational cost (resp.\ the
cumulative computation time). This, coined as \textsl{optimal complexity} in the
context of adaptive wavelet methods from~\cite{MR1803124,MR2035007}, was
later adopted for AFEM in~\cite{stevenson2007,cg2012}.
Therein,
optimal complexity is guaranteed
for AFEM with \textsl{inexact} solver,
provided that the computed iterates $u_\ell^{\underline{k}}$
are sufficiently close to the (unavailable) exact discrete solutions
$u_\ell^\star$.
This theoretical result requires that the algebraic error is
controlled by the discretization error multiplied by
a sufficiently small solver-stopping parameter~$\lambda$.
However, numerical experiments in~\cite{cg2012} indicate
that also moderate choices of the stopping parameter suffice
for optimal complexity.
Hence, the interrelated stopping criterion
led to a combined solve-estimate module in
the adaptive algorithm
%~\eqref{eq:semr:inexact}.
\begin{equation}\label{eq:semr:inexact}
	\boxed{~~\mathtt{solve~\&~estimate}~~}
	\quad \longrightarrow \quad
	\boxed{~~\mathtt{mark}~~}
	\quad \longrightarrow \quad
	\boxed{~~\mathtt{refine}~~}
\end{equation}
Driven by the interest in AFEMs for nonlinear problems in~\cite{ev2013,
	cw2017, ghps2018, hw2020a, hw2020b}, recent papers~\cite{ghps2021, hpw2021,
	hpsv2021} aimed to combine linearization and algebraic
iterates into a nested adaptive algorithm. Following the latter, the algorithmic decision for
either mesh refinement or linearization or algebraic solver step is steered by \textsl{a~posteriori}-based stopping criteria with suitable stopping parameters. This allows to balance the error components and compute the inexact approximations $u_\ell^k \approx
	u_\ell^\star$ given by a contractive solver with iteration counter $k = 1,
	\dots, \underline{k}[\ell]$ on the mesh
$\mathcal{T}_\ell$, and $|\ell,k| \in \mathbb{N}_0$ denotes the lexicographic order of the
sequential loop~\eqref{eq:semr:inexact}; see
Algorithm~\ref{algorithm:single} below.

Due to an energy identity (coinciding with~\eqref{eq:pythagoras} for symmetric linear PDEs), the works \cite{ghps2021,hpw2021} prove full R-linear convergence for the quasi-error
$\Delta_\ell^k \coloneqq \big[
		|\mkern-1.5mu|\mkern-1.5mu| u^\star - u_\ell^k |\mkern-1.5mu|\mkern-1.5mu|^2 + \gamma \, \eta_\ell(u_\ell^k)^2 \big]^{1/2}$
with respect to the lexicographic ordering \(|\cdot, \cdot|\), i.e.,
\begin{equation}\label{eq:intro:full_linear_single}
	\Delta_\ell^k
	\le
	C_{\textup{lin}} q_{\textup{lin}}^{|\ell,k| - |\ell',k'|} \, \Delta_{\ell'}^{k'}
	\quad \text{for all } (\ell',k'),(\ell,k) \in \mathcal{Q}
	\text{ with } |\ell',k'| \le |\ell,k|,
\end{equation}
which is guaranteed for arbitrary marking parameter $\theta$ and stopping parameter $\lambda$ (with constants $C_{\textup{lin}} > 0 $ and $0 < q_{\textup{lin}} < 1$ depending on $\theta$ and $\lambda$).
Moreover,~\cite{ghps2021} proves that full R-linear convergence is also the key argument for optimal complexity in the sense that it ensures, for all $s > 0$,
\begin{equation}\label{eq:intro:complexity}
	\vspace{-0.3cm}
	M(s)
	\coloneqq
	\sup_{(\ell,k) \in \mathcal{Q}} (\#\mathcal{T}_\ell)^s \Delta_\ell^k
	\le \sup_{(\ell,k) \in \mathcal{Q}} \Big( \sum_{\substack{(\ell',k') \in \mathcal{Q} \\
			|\ell',k'| \le |\ell,k|}} \#\mathcal{T}_{\ell'} \Big)^{s} \Delta_\ell^k \le C_{\textup{cost}}(s) \,
	M(s),
\end{equation}
where $C_{\textup{cost}}(s) > 1$ depends only on $s$, $C_{\textup{lin}}$, and $q_{\textup{lin}}$.
Since all modules of AFEM with inexact solver as displayed in~\eqref{eq:semr:inexact} can be implemented at linear cost
$\mathcal{O}(\#\mathcal{T}_\ell)$, the equivalence~\eqref{eq:intro:complexity} means that the quasi-error $\Delta_\ell^k$
decays with rate $s$ over the number of elements $\#\mathcal{T}_\ell$ if and only if it
decays with rate $s$ over the related \textsl{overall} computational work (and hence total computation time).

In essence, optimal complexity of AFEM with inexact solver thus follows from a perturbation argument (by taking the stopping parameter $\lambda$ sufficiently small) as soon as full linear convergence~\eqref{eq:intro:full_linear_single} of AFEM with inexact solver and optimal rates of AFEM with exact solver (for sufficiently small $\theta$) have been established; see,
e.g.,~\cite{cfpp2014,ghps2021}.

In this paper, we present a novel proof of full linear
convergence~\eqref{eq:intro:full_linear_single} with contractive
solver that, unlike~\cite{ghps2021, hpw2021}, avoids the Pythagorean
identity~\eqref{eq:pythagoras}, but relies only on the quasi-orthogonality
from~\cite{cfpp2014} (even in its generalized form from~\cite{feischl2022}).
The latter is known to be sufficient and necessary for linear
convergence~\eqref{eq:intro:axioms} in the presence of exact solvers~\cite{cfpp2014}. In
particular, this opens the door to proving optimal complexity for
AFEM beyond
symmetric energy minimization problems.
Moreover, problems exhibiting additional difficulties such as nonsymmetric
linear elliptic PDEs, see~\cite{aisfem}, or nonlinear PDEs,
see~\cite{hpsv2021}, ask for more intricate (nested) solvers that treat
iterative symmetrization/linearization together with solving the arising
linear SPD systems. This leads to computed approximations $u_\ell^{k,j} \approx
	u_\ell^\star$ with symmetrization/linearization iteration counter \(k = k[\ell]\)
and
algebraic solver index \(j = j[\ell, k]\). The new proof of full linear
convergence allows
to improve the analysis of~\cite{aisfem,hpsv2021} by relaxing the choice of
the solver-stopping parameters. Additionally, in the setting of~\cite{aisfem}, we
are able to show that the full linear convergence holds from the arbitrary
\textsl{initial}
mesh onwards
instead of the \textsl{a~priori} unknown and possibly large mesh threshold
level $\ell_0>0$. In particular, unlike the previous works~\cite{cn2012, ffp2014,bhp2017, aisfem} that employ a quasi-Pythagorean identity, the new analysis shows that the constants in the full R-linear convergence are independent of \(\mathcal{T}_0\) and/or the sequence of adaptively generated meshes and therefore fixed \textsl{a~priori}.
Furthermore, the new analysis does not only improve the state-of-the-art theory
of full linear
convergence leading to  optimal complexity,
but also allows the choice of larger solver-stopping parameters
which also leads to a better numerical performance in experiments.

The remainder of this work is structured as follows:
As a model problem, Section~\ref{section:preliminaries} formulates
a general second-order linear elliptic PDE together with the validity of the so-called \emph{axioms of adaptivity} from~\cite{cfpp2014} and the quasi-orthogonality from~\cite{feischl2022}.
In Section~\ref{section:exact}, AFEM with exact solver~\eqref{eq:semr:exact} is presented in Algorithm~\ref{algorithm:exact} and, for completeness and eased readability of the later sections, Theorem~\ref{theorem:exact:convergence}
summarizes the proof of R-linear convergence~\eqref{eq:intro:axioms} from~\cite{cfpp2014, feischl2022}.
Section~\ref{section:single} focuses on AFEM with inexact
contractive solver~\eqref{eq:semr:inexact} detailed in
Algorithm~\ref{algorithm:single}.
The main contribution is
the new and more general
proof of full R-linear convergence of Theorem~\ref{theorem:single:convergence}.
Corollary~\ref{corollary:rates:complexity} proves the important
equivalence~\eqref{eq:intro:complexity}.
The case of AFEM with \textsl{nested} contractive solvers, which are useful for nonlinear or nonsymmetric problems, is treated in Section~\ref{section:double} by presenting Algorithm~\ref{algorithm:double} from~\cite{aisfem} and improving its main result in
Theorem~\ref{theorem:double:convergence}.
In Section~\ref{section:conclusion}, we discuss the impact of the new analysis on AFEM for nonlinear PDEs. We show that Theorem~\ref{theorem:double:convergence} applies also to the setting from~\cite{hpsv2021}, namely strongly monotone PDEs with scalar nonlinearity.
Numerical experiments and remarks are discussed in-depth in Section~\ref{section:numerics}, where the impact of the adaptivity parameters on the overall cost is investigated empirically.

Throughout the proofs, the notation $A \lesssim B$ abbreviates $A \le C B$ for some positive constant $C > 0$ whose dependencies are clearly presented in the respective theorem and $A \simeq B$ abbreviates $A \lesssim B \lesssim A$.

\def\boxed#1{#1}

%%%%%%%%%%%%%%%%%%%%%%%%%%%%%%%%%%%%%%%%%%%%%%%%%%%%%%%%%%%%%%%%%%%%%%%%%%%%%%%%%%%
%%%%%%%%%%%%%%%%%%%%%%%%%%%%%%%%%%%%%%%%%%%%%%%%%%%%%%%%%%%%%%%%%%%%%%%%%%%%%%%%%%%
\section{General second-order linear elliptic PDEs}
\label{section:preliminaries}
%%%%%%%%%%%%%%%%%%%%%%%%%%%%%%%%%%%%%%%%%%%%%%%%%%%%%%%%%%%%%%%%%%%%%%%%%%%%%%%%%%%
%%%%%%%%%%%%%%%%%%%%%%%%%%%%%%%%%%%%%%%%%%%%%%%%%%%%%%%%%%%%%%%%%%%%%%%%%%%%%%%%%%%

On a bounded polyhedral Lipschitz domain $\Omega \subset \mathbb{R}^d$, $d \ge 1$,
we consider the PDE
\begin{equation}\label{eq:strongform}
	-\operatorname{div}(\boldsymbol{A} \nabla u^\star) + \boldsymbol{b} \cdot \nabla u^\star +
	cu^\star = f - \operatorname{div} \boldsymbol{f} \text{ in } \Omega
	\quad \text{subject to}\quad
	u^\star = 0 \text{ on } \partial\Omega,
\end{equation}
where $\boldsymbol{A}, \boldsymbol{b}, c \in L^\infty(\Omega)$ and
$\boldsymbol{f}, f \in L^2(\Omega)$ with, for almost every $x \in \Omega$,
positive definite $\boldsymbol{A}(x) \in \mathbb{R}^{d \times d}_{\textup{sym}}$,
$\boldsymbol{b}(x), \boldsymbol{f}(x) \in \mathbb{R}^d$, and $c(x), f(x) \in \mathbb{R}$. With
$\langle \cdot, \cdot \rangle_{L^2(\Omega)}$ denoting the usual $L^2(\Omega)$-scalar
product, we suppose that the PDE fits into the setting of the Lax--Milgram
lemma, i.e., the bilinear forms
\begin{equation*}
	a(u, v) \coloneqq \langle \boldsymbol{A} \nabla u, \nabla v \rangle_{L^2(\Omega)}
	\quad \text{and} \quad
	b(u, v)
	\coloneqq
	a(u, v) + \langle \boldsymbol{b} \cdot \nabla u + cu, v \rangle_{L^2(\Omega)}
\end{equation*}
are continuous and elliptic on $H^1_0(\Omega)$. Then, indeed, $a(\cdot, \cdot)$ is a scalar product and $|\mkern-1.5mu|\mkern-1.5mu| u |\mkern-1.5mu|\mkern-1.5mu| \coloneqq a(u,u)^{1/2}$ defines an equivalent norm on $H^1_0(\Omega)$. Moreover, the weak formulation
\begin{equation*}
	b(u^\star, v) = F(v) \coloneqq \langle f, v \rangle_{L^2(\Omega)} +
	\langle \boldsymbol{f}, \, \nabla v \rangle_{L^2(\Omega)}
	\quad \text{for all } v \in H^1_0(\Omega)
\end{equation*}
admits a unique solution $u^\star \in H^1_0(\Omega)$. Let \( 0 < C_{\textup{ell}} \le C_{\textup{bnd}} \) denote the continuity and ellipticity constant of $b(\cdot,\cdot)$ with
respect to \(|\mkern-1.5mu|\mkern-1.5mu| \, \cdot \,  |\mkern-1.5mu|\mkern-1.5mu|\), i.e., there holds
\begin{equation*}
	C_{\textup{ell}} \, |\mkern-1.5mu|\mkern-1.5mu| v |\mkern-1.5mu|\mkern-1.5mu|^2 \le b(v,v)
	\quad
	\text{and}
	\quad
	| b(v,w) | \le C_{\textup{bnd}} \, |\mkern-1.5mu|\mkern-1.5mu| v |\mkern-1.5mu|\mkern-1.5mu|
	\, |\mkern-1.5mu|\mkern-1.5mu| w |\mkern-1.5mu|\mkern-1.5mu|
	\quad \text{for all } v,w \in \mathcal{X}.
\end{equation*}

	\begin{remark}
	\label{remark:Garding}
	We stress that the analysis below extends to problems where the associated
	bilinear form $b(\cdot,\cdot)$ is not coercive but satisfies only a
	Gårding-type inequality allowing for well-posed more general second-order linear
	elliptic PDEs. In this context, it is well-known that the well-posedness of the discrete 
	FEM problems and validity of a uniform inf-sup condition require that the triangulations (i.e., the initial triangulation \(\mathcal{T}_0\) below) are sufficiently fine; see, e.g., \cite{bhp2017}.
	However, the question of an optimal algebraic
	solver for the resulting linear systems is beyond the scope of this work
	and is left for future research.
	\end{remark}

Let \(\mathcal{T}_{0}\) be an initial conforming triangulation of
\(\Omega \subset \mathbb{R}^d\) into compact simplices.
The mesh refinement employs newest-vertex bisection (NVB). We refer
to
\cite{stevenson2008} for NVB with admissible $\mathcal{T}_0$ and \(d \ge 2\), to
\cite{Karkulik2013a} for NVB with general \(\mathcal{T}_0\) for \(d=2\), and to the
recent work \cite{dgs2023} for NVB with general $\mathcal{T}_0$ in any dimension
\(d \ge 2\). For \(d = 1\), we refer to \cite{AFFKP13}.
For each
triangulation \(\mathcal{T}_H\) and
\(\mathcal{M}_H \subseteq \mathcal{T}_H\), let
\(
\mathcal{T}_h \coloneqq \texttt{refine}(\mathcal{T}_H, \mathcal{M}_H)
\)
be the coarsest conforming refinement of \(\mathcal{T}_H\) such that
at least all elements \(T \in \mathcal{M}_H\)
have been refined, i.e.,
\(\mathcal{M}_H  \subseteq \mathcal{T}_H \setminus \mathcal{T}_h\).
To abbreviate notation, we write
\(
\mathcal{T}_h \in \mathbb{T}(\mathcal{T}_H)
\)
if \(\mathcal{T}_h\) can be obtained from \(\mathcal{T}_H\) by finitely many
steps of NVB and, in particular,
\(\mathbb{T} \coloneqq \mathbb{T}(\mathcal{T}_0)\).

For each $\mathcal{T}_H \in \mathbb{T}$, we consider conforming
finite element spaces
\begin{equation}\label{eq:discrete_space}
	\mathcal{X}_H \coloneqq \{v_H \in H^1_0(\Omega) \,:\, v_H|_T \text{ is a polynomial of total degree} \le p
	\text{ for all } T \in \mathcal{T}_H\},
\end{equation}
where
$p \in \mathbb{N}$ is
a fixed polynomial degree.
We note that \(\mathcal{T}_h \in \mathbb{T}(\mathcal{T}_H)\) yields nestedness \(\mathcal{X}_H \subseteq \mathcal{X}_h\) of the corresponding discrete spaces.

Given $\mathcal{T}_H \in \mathbb{T}$, there exists a unique Galerkin solution
$u_H^\star \in \mathcal{X}_H$ solving
\begin{equation}\label{eq:discrete}
	b(u^\star_H,v_H) = F(v_H)
	\quad \text{for all } v_H \in \mathcal{X}_H.
\end{equation}
Moreover, there holds
the following C\'ea lemma
\begin{equation}\label{eq:cea}
	|\mkern-1.5mu|\mkern-1.5mu| u^\star - u_H^\star |\mkern-1.5mu|\mkern-1.5mu|
	\le
	C_{\textup{Céa}} \, \min_{v_H \in \mathcal{X}_H} |\mkern-1.5mu|\mkern-1.5mu| u^\star - v_H |\mkern-1.5mu|\mkern-1.5mu|
	\quad \text{with a constant } 1 \le C_{\textup{Céa}} \le C_{\textup{bnd}}/ C_{\textup{ell}},
\end{equation}
where \(C_{\textup{Céa}} \to 1\) as adaptive mesh-refinement progresses as shown in~\cite[Theorem~20]{bhp2017}.

We consider the residual error estimator~\(\eta_H(\cdot)\)
defined, for \(T \in
\mathcal{T}_H\) and \(v_H \in \mathcal{X}_H\), by
\begin{subequations}\label{eq:definition_eta}
	\begin{align}\label{eq:definition_eta:a}
		\begin{split}
			\eta_H(T, v_H)^2
			 & \coloneqq
			|T|^{2/d} \, \Vert -\operatorname{div}(\boldsymbol{A} \nabla v_H - \boldsymbol{f}) + \boldsymbol{b} \cdot \nabla v_H + c \, v_H - f
			\Vert_{L^2(T)}^2
			\\& \qquad
			+
			|T|^{1/d} \, \Vert \lbrack\!\lbrack (\boldsymbol{A} \nabla v_H - \boldsymbol{f}) \cdot n \rbrack\!\rbrack
			\Vert_{L^2(\partial T \cap \Omega)}^2,
		\end{split}
	\end{align}
	where \(\lbrack\!\lbrack \cdot \rbrack\!\rbrack\) denotes the jump over $(d-1)$-dimensional
	faces. Clearly, the well-posedness of~\eqref{eq:definition_eta:a} requires
	more
	regularity of \(\boldsymbol{A}\) and \(\boldsymbol{f}\) than stated above, e.g.,
	\(\boldsymbol{A} \vert_T, \boldsymbol{f} \vert_T \in W^{1, \infty}(T)\) for all \(T \in
	\mathcal{T}_{0}\).
	To abbreviate notation, we define, for all $\mathcal{U}_H \subseteq \mathcal{T}_H$ and all $v_H \in \mathcal{X}_H$,
	\begin{equation}
		\eta_H(v_H)
		\coloneqq
		\eta_H(\mathcal{T}_H, v_H)
		\quad \text{with} \quad
		\eta_H(\mathcal{U}_H, v_H)
		\coloneqq \Big( \sum_{T \in \mathcal{U}_H} \eta_H(T, v_H)^2
		\Big)^{1/2}.
	\end{equation}
\end{subequations}
From~\cite{cfpp2014}, we recall that the error estimator satisfies the following properties.

\begin{proposition}[axioms of adaptivity]\label{prop:axioms}
	There exist constants $C_{\textup{stab}}, C_{\textup{rel}}, C_{\textup{drel}}, C_{\textup{mon}} > 0$, and $0 <
		q_{\textup{red}} < 1$ such that the following properties are satisfied for any
	triangulation $\mathcal{T}_H \in \mathbb{T}$ and any conforming refinement $\mathcal{T}_h \in
		\mathbb{T}(\mathcal{T}_H)$
	with the corresponding Galerkin solutions $u_H^\star \in \mathcal{X}_H$, $u_h^\star \in
		\mathcal{X}_h$ to~\eqref{eq:discrete} and arbitrary $v_H \in \mathcal{X}_H$, $v_h \in \mathcal{X}_h$.
	\begin{enumerate}[leftmargin=2.7em]
		\renewcommand{\theenumi}{A\arabic{enumi}}
		\bf
		\item[(A1)]\refstepcounter{enumi}\label{axiom:stability} \textit{stability}.
		      \quad \rm
		      $|\eta_h(\mathcal{T}_h \cap \mathcal{T}_H, v_h) - \eta_H(\mathcal{T}_h \cap \mathcal{T}_H, v_H)| \le C_{\textup{stab}} \, |\mkern-1.5mu|\mkern-1.5mu| v_h - v_H |\mkern-1.5mu|\mkern-1.5mu|$.
		      \bf
		\item[(A2)]\refstepcounter{enumi}\label{axiom:reduction} \textit{reduction}.
		      \quad \rm
		      $\eta_h(\mathcal{T}_h \backslash \mathcal{T}_H, v_H) \le q_{\textup{red}} \, \eta_H(\mathcal{T}_H \backslash \mathcal{T}_h, v_H)$.
		      \bf
		\item[(A3)]\refstepcounter{enumi}\label{axiom:reliability}
		      \textit{reliability}. \quad \rm
		      $|\mkern-1.5mu|\mkern-1.5mu| u^\star - u_H^\star |\mkern-1.5mu|\mkern-1.5mu| \le C_{\textup{rel}} \,
			      \eta_H(u_H^\star)$.

		      \renewcommand{\theenumi}{A3$^{+}$}
		      \bf
		\item[(A3$^+$)]\refstepcounter{enumi}
		      \label{axiom:discrete_reliability}
		      \it
		      \textbf{discrete reliability.}
		      \quad
		      \rm
		      $
			      |\mkern-1.5mu|\mkern-1.5mu| u_h^\star - u_H^\star |\mkern-1.5mu|\mkern-1.5mu|
			      \le
			      C_{\textup{drel}} \, \eta_H(\mathcal{T}_H \backslash \mathcal{T}_h, u_H^\star).
		      $

		      \renewcommand{\theenumi}{QM}
		      \bf
		\item[(QM)]\refstepcounter{enumi}\label{eq:quasi-monotonicity}
		      \textit{quasi-monotonicity}. \quad \rm
		      $\eta_h(u_h^\star) \le C_{\textup{mon}} \,
			      \eta_H(u_H^\star)$.

	\end{enumerate}
	The constant $C_{\textup{rel}}$ depends only on uniform shape regularity of all meshes
	$\mathcal{T}_H \in \mathbb{T}$ and the dimension $d$, while $C_{\textup{stab}}$ and
	\(C_{\textup{drel}}\) additionally depend on the polynomial degree $p$. The constant
	$q_{\textup{red}}$ reads $q_{\textup{red}} \coloneqq 2^{-1/(2d)}$ for bisection-based refinement
	rules in $\mathbb{R}^d$ and the constant $C_{\textup{mon}}$ can be bounded by $C_{\textup{mon}}
		\le \min\{ 1 +
		C_{\textup{stab}}(1+C_{\textup{Céa}}) \, C_{\textup{rel}} \, , \, 1 + C_{\textup{stab}} \, C_{\textup{drel}}\}$.\qed
\end{proposition}

In addition to the estimator properties in Proposition~\ref{prop:axioms}, we recall the following quasi-orthogonality
result from~\cite{feischl2022} as one cornerstone of the improved analysis in this paper.

\begin{proposition}[validity of quasi-orthogonality]\label{prop:orthogonality}
	There exist
	\(C_{\textup{orth}} > 0\) and \(0 < \delta \le 1\) such
	that
	the following holds: For any sequence $\mathcal{X}_\ell \subseteq \mathcal{X}_{\ell+1} \subset
		H^1_0(\Omega)$ of nested finite-dimensional subspaces, the corresponding Galerkin
	solutions $u_\ell^\star \in \mathcal{X}_\ell$ to \eqref{eq:discrete} satisfy
	\begin{enumerate}[leftmargin=2.7em]
		\renewcommand{\theenumi}{A\arabic{enumi}}
		\setcounter{enumi}{3}
		\bf
		\item[(A4)]\refstepcounter{enumi}\label{axiom:orthogonality} \textit{quasi-orthogonality}. \rm
		      \!\!$\displaystyle
			      \sum_{\ell' = \ell}^{\ell + N} |\mkern-1.5mu|\mkern-1.5mu|u^\star_{\ell'\!+\!1} \!-
			      u^\star_{\ell'} |\mkern-1.5mu|\mkern-1.5mu|^2
			      \!\le\!
			      C_{\textup{orth}} (N+1)^{1-\delta} |\mkern-1.5mu|\mkern-1.5mu| u^\star \!- u^\star_\ell |\mkern-1.5mu|\mkern-1.5mu|^2
			      \, \text{for all }\!\! \ell, N \in \mathbb{N}_0.
		      $
	\end{enumerate}
	Here, $C_{\textup{orth}}$ and \(\delta\) depend only on the
	dimension $d$, the elliptic bilinear form $b(\cdot,\cdot)$, and the chosen norm
	$|\mkern-1.5mu|\mkern-1.5mu| \cdot |\mkern-1.5mu|\mkern-1.5mu|$, but are independent of the spaces $\mathcal{X}_\ell$. \qed
\end{proposition}

\begin{remark}\label{remark:orthogonality}
	Quasi-orthogonality~\eqref{axiom:orthogonality} is a generalization of the Pythagorean identity~\eqref{eq:pythagoras} for symmetric problems.
	Indeed, if $\boldsymbol{b} = 0$ in~\eqref{eq:strongform} and $a(\cdot,\cdot) \coloneqq b(\cdot,\cdot)$ is a scalar product, the Galerkin method for nested subspaces $\mathcal{X}_\ell \subseteq \mathcal{X}_{\ell+1} \subset H^1_0(\Omega)$ guarantees
	~\eqref{eq:pythagoras}.
	Thus, the telescopic series proves~\eqref{axiom:orthogonality} with $C_{\textup{orth}} = 1$ and \(\delta = 1\).
	We highlight that~\cite{feischl2022} proves~\eqref{axiom:orthogonality} even for more general linear problems and Petrov--Galerkin discretizations, where it is only needed that the continuous and discrete inf-sup constants are uniformly bounded from below. In particular, this applies to a wide range of mixed FEM formulations, but also to general second-order linear elliptic PDEs that only satisfy a Gårding-type inequality; see Remark~\ref{remark:Garding} above.
\end{remark}

A closer look at the proofs of R-linear convergence in Section~\ref{section:exact}--\ref{section:double} below reveals that they rely only on the properties~\eqref{axiom:stability}, \eqref{axiom:reduction}, \eqref{axiom:reliability}, \eqref{axiom:orthogonality}, and \eqref{eq:quasi-monotonicity}, but not on~\eqref{axiom:discrete_reliability}, the C\'ea lemma~\eqref{eq:cea}, or linearity of the PDE. Hence, Algorithms~\ref{algorithm:exact}, \ref{algorithm:single}, and~\ref{algorithm:double} and the corresponding Theorems~\ref{theorem:exact:convergence}, \ref{theorem:single:convergence}, and~\ref{theorem:double:convergence} apply beyond the linear problem~\eqref{eq:strongform}; see Section~\ref{section:conclusion} for
a nonlinear PDE.

%%%%%%%%%%%%%%%%%%%%%%%%%%%%%%%%%%%%%%%%%%%%%%%%%%%%%%%%%%%%%%%%%%%%%%%%%%%%%%%%%%%
%%%%%%%%%%%%%%%%%%%%%%%%%%%%%%%%%%%%%%%%%%%%%%%%%%%%%%%%%%%%%%%%%%%%%%%%%%%%%%%%%%%
\section{AFEM with exact solution}
\label{section:exact}
%%%%%%%%%%%%%%%%%%%%%%%%%%%%%%%%%%%%%%%%%%%%%%%%%%%%%%%%%%%%%%%%%%%%%%%%%%%%%%%%%%%
%%%%%%%%%%%%%%%%%%%%%%%%%%%%%%%%%%%%%%%%%%%%%%%%%%%%%%%%%%%%%%%%%%%%%%%%%%%%%%%%%%%

To outline the new proof strategy, we first consider the standard adaptive algorithm
(see, e.g.,~\cite{ckns2008}), where the arising Galerkin
systems~\eqref{eq:discrete} are solved exactly.

\begin{algorithm}[AFEM with exact solver]\label{algorithm:exact}
	Given an initial mesh $\mathcal{T}_0$ and adaptivity parameters $0 < \theta \le 1$ and $C_{\textup{mark}} \ge 1$, iterate the following steps for all $\ell = 0, 1, 2, 3, \dots$:
	\begin{enumerate}[label=(\roman*), font = \upshape]
		\item \textbf{Solve:} Compute the exact solution $u_\ell^\star \in
			      \mathcal{X}_\ell$ to~\eqref{eq:discrete}.
		\item \textbf{Estimate:} Compute the refinement indicators
		      $\eta_\ell(T, u_\ell^\star)$ for all $T \in \mathcal{T}_\ell$.
		\item \textbf{Mark:} Determine a set $\mathcal{M}_\ell \in
			      \mathbb{M}_\ell[\theta,u_\ell^\star]$ satisfying the D\"orfler marking criterion
		      \begin{equation}\label{eq:doerfler}
			      \#\mathcal{M}_\ell \le C_{\textup{mark}} \!\! \min_{\mathcal{U}_\ell^\star \in \mathbb{M}_\ell[\theta,u_\ell^\star]}\!\! \#\mathcal{U}_\ell^\star,
			      \, \text{where} \,\,
			      \mathbb{M}_\ell[\theta,u_\ell^\star] \coloneqq \bigl\{\mathcal{U}_\ell \subseteq \mathcal{T}_\ell
			      \,:\, \theta \eta_\ell(u_\ell^\star)^2 \le \eta_\ell(\mathcal{U}_\ell,
			      u_\ell^\star)^2\bigr\}.
		      \end{equation}
		\item \textbf{Refine:} Generate $\mathcal{T}_{\ell+1} \coloneqq \mathtt{refine}(\mathcal{T}_\ell,\mathcal{M}_\ell)$.
	\end{enumerate}
\end{algorithm}

The following theorem asserts convergence of Algorithm~\ref{algorithm:exact} in the spirit of~\cite{cfpp2014},
and the proof given below essentially summarizes the arguments from~\cite{feischl2022}. It will, however, be the starting point for the later generalizations, i.e., for the adaptive algorithms below with inexact solvers.
\begin{theorem}[R-linear convergence of Algorithm~\ref{algorithm:exact}]\label{theorem:exact:convergence}
	Let $0 < \theta \le 1$ and $C_{\textup{mark}} \ge 1$ be arbitrary. Then,
	Algorithm~\ref{algorithm:exact} guarantees R-linear convergence of the
	estimators $\eta_\ell(u_\ell^\star)$, i.e.,
	there exist constants $0 < q_{\textup{lin}} < 1$ and $C_{\textup{lin}} > 0$ such that
	\begin{equation}\label{eq:exact:convergence}
		\eta_{\ell+n}(u^\star_{\ell+n}) \le C_{\textup{lin}} \, q_{\textup{lin}}^n \, \eta_\ell(u^\star_\ell)
		\quad \text{for all } \ell, n \in \mathbb{N}_0.
	\end{equation}
\end{theorem}

\begin{remark}\label{rem:Rlinear}
	For vanishing convection \(\boldsymbol{b} = 0\) in~\eqref{eq:strongform} and  $a(\cdot,\cdot) \coloneqq b(\cdot,\cdot)$,~\cite{ckns2008} proves linear convergence of the quasi-error~\eqref{eq:intro:ckns}.
	Together with reliability~\eqref{axiom:reliability}, this yields R-linear convergence of the estimator sequence
	\begin{equation*}
		\eta_{\ell+n}(u_{\ell+n}^\star)
		\le \frac{(C_{\textup{rel}}^2+\gamma)^{1/2}}{\gamma^{1/2}} \, q_{\textup{ctr}}^n \,
		\eta_\ell(u_\ell^\star)
		\quad \text{for all } \ell, n \in \mathbb{N}_0.
	\end{equation*}
	In this sense, Theorem~\ref{theorem:exact:convergence} is weaker than linear convergence~\eqref{eq:intro:ckns} from~\cite{ckns2008}, but provides a direct proof of R-linear convergence even if \(b(\cdot, \cdot) \neq a(\cdot, \cdot)\).
	Moreover, while the proof of~\eqref{eq:intro:ckns}
	crucially relies on the Pythagorean identity~\eqref{eq:pythagoras},
	the works~\cite{ffp2014,bhp2017} extend the analysis to
	the general second-order linear elliptic PDE~\eqref{eq:strongform}
	using
	\begin{equation}\label{eq:general-pythagoras}
		\forall 0 < \varepsilon < 1 \, \exists \ell_0 \in \mathbb{N}_0 \, \forall
		\ell \ge \ell_0 \colon \quad
		|\mkern-1.5mu|\mkern-1.5mu| u^\star - u^\star_{\ell+1} |\mkern-1.5mu|\mkern-1.5mu|^2
		\le
		|\mkern-1.5mu|\mkern-1.5mu| u^\star - u^\star_\ell |\mkern-1.5mu|\mkern-1.5mu|^2 - \varepsilon \,
		|\mkern-1.5mu|\mkern-1.5mu| u^\star_{\ell+1} - u^\star_\ell |\mkern-1.5mu|\mkern-1.5mu|^2.
	\end{equation}
	From this, contraction~\eqref{eq:intro:ckns} follows
	for all $\ell \ge
		\ell_0$ and allows to extend the AFEM analysis
	from~\cite{stevenson2007,ckns2008} to
	general second-order linear elliptic PDE.
	However,
	the
	index $\ell_0$ depends on the exact solution $u^\star$ and on the
	sequence of exact discrete solutions
	$(u_\ell^\star)_{\ell \in \mathbb{N}_0}$. Moreover, $\ell_0 = 0$ requires sufficiently fine $\mathcal{T}_0$ in~\cite{cn2012,bhp2017} while the constants in~\eqref{eq:exact:convergence} depend on \(u^\star\) and the sequence \((u_\ell^\star)_{\ell \in \mathbb{N}_0}\) in~\cite{ffp2014}.
	In the present work however, R-linear convergence~\eqref{eq:exact:convergence} from Theorem~\ref{theorem:exact:convergence} holds with $\ell_0 = 0$ and any initial mesh $\mathcal{T}_0$, and the constants are independent of \(u^\star\) and
	\((u_\ell^\star)_{\ell \in \mathbb{N}_0}\), thus clearly improving the previous state of the art.
\end{remark}

The proof of Theorem~\ref{theorem:exact:convergence} relies on the following elementary lemma that extends
arguments implicitly found for the estimator
sequence
in~\cite{feischl2022} but that will be employed for certain
quasi-errors in the present work. Its proof is found in~\ref{appendix}.

\begin{lemma}[tail summability criterion]\label{lemma:summability:criterion}
	Let $(a_\ell)_{\ell \in \mathbb{N}_0}, (b_\ell)_{\ell \in \mathbb{N}_0}$ be scalar sequences in $\mathbb{R}_{\ge 0}$.
	With given constants $0 < q < 1$, \(0 < \delta <
	1\), and $C_1, C_2 > 0$, suppose that
	\begin{equation}\label{eq:summability:criterion}
		a_{\ell+1} \le q a_\ell + b_\ell,
		\quad
		b_{\ell+N} \le C_1 \, a_\ell,
		\,\, \text{and} \,\,
		\sum_{\ell' = \ell}^{\ell+N} b_{\ell'}^2 \le C_2 \,
		(N+1)^{1-\delta} \, a_\ell^2
		\,\,\, \text{for all } \ell, N \in \mathbb{N}_0.
	\end{equation}
	Then, $(a_\ell)_{\ell \in \mathbb{N}_0}$ is R-linearly convergent, i.e.,
	there exist
	\(C_{\textup{lin}} > 0\) and \(0 < q_{\textup{lin}} < 1\)
	with
	\begin{equation}\label{eq:Rlinear:convergence}
		a_{\ell+n} \le C_{\textup{lin}} \, q_{\textup{lin}}^n \, a_\ell
		\quad
		\text{for all \(\ell, n \in \mathbb{N}_0\).}
	\end{equation}
\end{lemma}

\begin{proof}[\textbf{Proof of Theorem~\ref{theorem:exact:convergence}}]
	We employ Lemma~\ref{lemma:summability:criterion} for the sequences defined by $a_\ell =
		\eta_\ell(u_\ell^\star)$ and $b_\ell \coloneqq C_{\textup{stab}} \, |\mkern-1.5mu|\mkern-1.5mu| u_{\ell+1}^\star -
		u_\ell^\star |\mkern-1.5mu|\mkern-1.5mu|$.
	First, we note that
	\begin{equation}
		\label{eq:quasimonotonicity_b}
		|\mkern-1.5mu|\mkern-1.5mu| u_{\ell^{\prime \prime}}^\star - u_{\ell^\prime}^\star |\mkern-1.5mu|\mkern-1.5mu|
		\stackrel{\eqref{axiom:reliability}}\lesssim
		\eta_{\ell''}(u_{\ell''}^\star) +
		\eta_{\ell^\prime}(u_{\ell^\prime}^\star)
		\stackrel{\eqref{eq:quasi-monotonicity}}
		\lesssim
		\eta_{\ell}(u_\ell^\star)
		\quad \text{for all \(\ell, \ell^\prime, \ell^{\prime \prime} \in \mathbb{N}_0\)}
		\text{ with \(\ell \le \ell^\prime \le \ell^{\prime \prime }\)}.
	\end{equation}
	In particular, this proves \(b_{\ell+N} \lesssim a_\ell\) for all
	\(\ell, N \in \mathbb{N}_0\).
	Moreover,
	quasi-orthogonality~\eqref{axiom:orthogonality} and
	reliability~\eqref{axiom:reliability} show
	\begin{equation}\label{eq:summability_b}
		\sum_{\ell' = \ell}^{\ell+N} \, |\mkern-1.5mu|\mkern-1.5mu| u^\star_{\ell'+1} -
		u^\star_{\ell'} |\mkern-1.5mu|\mkern-1.5mu|^2
		\le C_{\textup{orth}} \, C_{\textup{rel}}^2 \, (N+1)^{1-\delta} \,
		\eta_\ell(u^\star_\ell)^2
		\quad \text{for all } \ell, N \in \mathbb{N}_0.
	\end{equation}
	In order to verify~\eqref{eq:summability:criterion}, it thus only remains to prove the perturbed
	contraction of $a_\ell$.
	To this end, let $\ell \in \mathbb{N}_0$. Then, stability~\eqref{axiom:stability} and
	reduction~\eqref{axiom:reduction} show
	\begin{equation*}
		\eta_{\ell+1}(u^\star_\ell)^2
		\le
		\eta_\ell(\mathcal{T}_{\ell+1} \cap \mathcal{T}_\ell, u^\star_\ell)^2
		+
		q_{\textup{red}}^2 \eta_\ell(\mathcal{T}_\ell \backslash \mathcal{T}_{\ell+1}, u^\star_\ell)^2
		=
		\eta_\ell(u_\ell^\star)^2
		\!-\! (1 \!-\! q_{\textup{red}}^2) \, \eta_\ell(\mathcal{T}_\ell \backslash \mathcal{T}_{\ell+1}, u^\star_\ell)^2.
	\end{equation*}
	Moreover, D\"orfler marking~\eqref{eq:doerfler} and refinement of (at least) all marked elements lead to
	\begin{equation*}
		\theta \eta_\ell(u^\star_\ell)^2
		\stackrel{\eqref{eq:doerfler}}\le
		\eta_\ell(\mathcal{M}_\ell, u^\star_\ell)^2
		\le
		\eta_\ell(\mathcal{T}_\ell \backslash \mathcal{T}_{\ell+1}, u^\star_\ell)^2.
	\end{equation*}
	The combination of the two previously displayed formulas results in
	\begin{equation*}
		\eta_{\ell+1}(u^\star_\ell)
		\le
		q_\theta \, \eta_\ell(u_\ell^\star)
		\quad \text{with} \quad
		0 < q_\theta \coloneqq \big[ 1 - (1 - q_{\textup{red}}^2) \theta \big]^{1/2} < 1.
	\end{equation*}
	Finally, stability~\eqref{axiom:stability} thus leads to the desired estimator reduction estimate
	\begin{equation}\label{eq:estimator-reduction}
		\eta_{\ell+1}(u^\star_{\ell+1})
		\stackrel{\eqref{axiom:stability}}\le
		q_\theta \, \eta_\ell(u_\ell^\star) + C_{\textup{stab}} \, |\mkern-1.5mu|\mkern-1.5mu| u_{\ell+1}^\star - u_\ell^\star |\mkern-1.5mu|\mkern-1.5mu|
		\quad \text{for all } \ell \in \mathbb{N}_0.
	\end{equation}
	Altogether, all the assumptions~\eqref{eq:summability:criterion} are satisfied and Lemma~\ref{lemma:summability:criterion} concludes the proof.
\end{proof}

%%%%%%%%%%%%%%%%%%%%%%%%%%%%%%%%%%%%%%%%%%%%%%%%%%%%%%%%%%%%%%%%%%%%%%%%%%%%%%%%%%%
%%%%%%%%%%%%%%%%%%%%%%%%%%%%%%%%%%%%%%%%%%%%%%%%%%%%%%%%%%%%%%%%%%%%%%%%%%%%%%%%%%%
\section{AFEM with contractive solver}
\label{section:single}
%%%%%%%%%%%%%%%%%%%%%%%%%%%%%%%%%%%%%%%%%%%%%%%%%%%%%%%%%%%%%%%%%%%%%%%%%%%%%%%%%%%
%%%%%%%%%%%%%%%%%%%%%%%%%%%%%%%%%%%%%%%%%%%%%%%%%%%%%%%%%%%%%%%%%%%%%%%%%%%%%%%%%%%

Let $\Psi_H \colon \mathcal{X}_H \to \mathcal{X}_H$ be the iteration mapping of
a uniformly contractive solver, i.e.,
\begin{equation}\label{eq:contractive-solver}
	|\mkern-1.5mu|\mkern-1.5mu| u^\star_H - \Psi_H(v_H) |\mkern-1.5mu|\mkern-1.5mu|
	\le
	q_{\textup{alg}} \, |\mkern-1.5mu|\mkern-1.5mu| u^\star_H - v_H |\mkern-1.5mu|\mkern-1.5mu|
	\quad \text{for all } \mathcal{T}_H \in \mathbb{T} \text{ and all } v_H \in \mathcal{X}_H.
\end{equation}
Examples of such iterative solvers include an optimally preconditioned
	conjugate gradient method from~\cite{cnx2012} or an optimal geometric multigrid
	method from~\cite{wz2017,imps2022}.
The following algorithm is thoroughly analyzed in~\cite{ghps2021}
under the
assumption that the problem is symmetric (and hence the Pythagorean
identity~\eqref{eq:pythagoras} holds).

\begin{algorithm}[AFEM with contractive solver]\label{algorithm:single}
	Given an initial mesh $\mathcal{T}_0$, adaptivity parameters $0 < \theta \le 1$ and $C_{\textup{mark}} \ge 1$, a solver-stopping parameter $\lambda > 0$, and an initial guess $u_0^0 \in \mathcal{X}_0$, iterate the following steps~\ref{alg:single_i}--\ref{alg:single_iv} for all $\ell = 0, 1, 2, 3, \dots$:
	\begin{enumerate}[label=(\roman*), ref = {\rm (\roman*)}, font = \upshape]
		\item\label{alg:single_i} \textbf{Solve \& Estimate:} For all $k = 1, 2, 3, \dots$, repeat~\ref{alg:single_a}--\ref{alg:single_b} until
		      \begin{equation}\label{eq:single:termination}
			      |\mkern-1.5mu|\mkern-1.5mu| u_\ell^k - u_\ell^{k-1} |\mkern-1.5mu|\mkern-1.5mu| \le \lambda \, \eta_\ell(u_\ell^k).
		      \end{equation}
		      \begin{enumerate}[label=(\alph*), ref = {\rm (\alph*)}, font = \upshape]
			      \item\label{alg:single_a} Compute $u_\ell^k \coloneqq \Psi_\ell(u_\ell^{k-1})$ with one step of the contractive solver.
			      \item\label{alg:single_b} Compute the refinement indicators $\eta_\ell(T, u_\ell^k)$ for all $T \in \mathcal{T}_\ell$.
		      \end{enumerate}
		\item\label{alg:single:ii} Upon termination of the iterative solver, define the index $\underline{k}[\ell] \coloneqq k \in \mathbb{N}$.
		\item \textbf{Mark:} Determine a set $\mathcal{M}_\ell \in \mathbb{M}_\ell[\theta,
				      u_\ell^{\underline{k}}]$ satisfying~\eqref{eq:doerfler} with $u_\ell^\star$  replaced by
		      $u_\ell^{\underline{k}}$.
		\item\label{alg:single_iv} \textbf{Refine:} Generate $\mathcal{T}_{\ell+1} \coloneqq \mathtt{refine}(\mathcal{T}_\ell,\mathcal{M}_\ell)$ and employ nested iteration $u_{\ell+1}^0 \coloneqq u_\ell^{\underline{k}}$.
	\end{enumerate}
\end{algorithm}

The sequential nature of Algorithm~\ref{algorithm:single} gives rise to the countably infinite index set
\begin{equation*}
	\mathcal{Q} \coloneqq \bigl\{(\ell,k) \in \mathbb{N}_0^2 \,:\, u_\ell^k \in \mathcal{X}_\ell \text{ is defined in Algorithm~\ref{algorithm:single}} \bigr \}
\end{equation*}
together with the lexicographic ordering
\begin{equation*}
	(\ell',k') \le (\ell,k)
	\quad :\Longleftrightarrow \quad
	\text{$u_{\ell'}^{k'}$ is defined not later than $u_\ell^k$ in
		Algorithm~\ref{algorithm:single}}
\end{equation*}
and the total step counter
\begin{equation*}
	| \ell, k | \coloneqq \# \bigl \{(\ell', k') \in \mathcal{Q} \,:\, (\ell', k') \le
	(\ell, k) \bigr \} \in \mathbb{N}_0
	\quad \text{for all \((\ell, k) \in \mathcal{Q}\).}
\end{equation*}
Defining the stopping indices
\begin{subequations}
	\begin{align}
		\underline{\ell}    & \coloneqq \sup \{\ell \in \mathbb{N}_0 \,:\, (\ell,0) \in \mathcal{Q}\} \in \mathbb{N}_0 \cup \{\infty\},
		\\
		\underline{k}[\ell] & \coloneqq \sup \{k \in \mathbb{N}_0 \,:\, (\ell,k) \in \mathcal{Q}\} \in \mathbb{N} \cup \{\infty\},
		\quad \text{whenever } (\ell,0) \in \mathcal{Q},
	\end{align}
\end{subequations}
we note that these definitions are consistent with that of
Algorithm~\ref{algorithm:single}\ref{alg:single:ii}.
We abbreviate \(\underline{k} = \underline{k}[\ell]\), whenever the index
\(\ell\) is clear from the context, e.g., \(u_\ell^{\underline{k}}
\coloneqq u_\ell^{\underline{k}[\ell]}\) or \((\ell, \underline{k}) = (\ell, \underline{k}[\ell])\).

As $\mathcal{Q}$ is an infinite set, the typical case is $\underline{\ell} = \infty$ and $\underline{k}[\ell] < \infty$ for all
$\ell \in \mathbb{N}_0$, whereas $\underline{\ell} < \infty$ implies that $\underline{k}[\underline{\ell}] = \infty$,
i.e., non-termination of the iterative solver on the mesh
\(\mathcal{T}_{\underline{\ell}}\). The following theorem states convergence of
Algorithm~\ref{algorithm:single}.
In particular, it shows that $\underline{\ell} < \infty$ implies
$\eta_{\underline{\ell}}(u_{\underline{\ell}}^\star) = 0$ and consequently $u^\star = u_{\underline{\ell}}^\star$ by
reliability~\eqref{axiom:reliability}.

\begin{theorem}[full R-linear convergence of Algorithm~\ref{algorithm:single}]\label{theorem:single:convergence}
	Let $0 < \theta \le 1$, $C_{\textup{mark}} \ge 1$, $\lambda > 0$, and $u_0^0 \in \mathcal{X}_0$ be arbitrary. Then,
	Algorithm~\ref{algorithm:single} guarantees R-linear convergence of the modified quasi-error
	\begin{equation}\label{eq2:single:quasi-error}
		\mathrm{H}_\ell^k
		\coloneqq
		|\mkern-1.5mu|\mkern-1.5mu| u_\ell^\star - u_\ell^k |\mkern-1.5mu|\mkern-1.5mu| + \eta_\ell(u_\ell^k),
	\end{equation}
	i.e., there exist constants $0 < q_{\textup{lin}} < 1$ and $C_{\textup{lin}} > 0$ such that
	\begin{equation}\label{eq:single:convergence}
		\mathrm{H}_\ell^k
		\le
		C_{\textup{lin}} q_{\textup{lin}}^{|\ell,k| - |\ell',k'|} \, \mathrm{H}_{\ell'}^{k'}
		\quad \text{for all } (\ell',k'),(\ell,k) \in \mathcal{Q}
		\text{ with } |\ell',k'| \le |\ell,k|.
	\end{equation}
\end{theorem}

\begin{remark}\label{rem:equivalence_Delta_Eta}
	Unlike \cite{ghps2021} (and \cite{ckns2008}),
	Theorem~\ref{theorem:single:convergence} and its proof employ the
	quasi-error \(\mathrm{H}_{\ell}^k\) from~\eqref{eq2:single:quasi-error} instead
	of
	\(
	\Delta_\ell^k
	\coloneqq
	\big[
		|\mkern-1.5mu|\mkern-1.5mu| u^\star - u_\ell^k |\mkern-1.5mu|\mkern-1.5mu|^2 + \gamma \, \eta_\ell(u_\ell^k)^2
		\big]^{1/2}
	\)
	analogous to~\eqref{eq:intro:ckns}. We note that stability~\eqref{axiom:stability}
	and reliability~\eqref{axiom:reliability} yield \(\Delta_\ell^k \lesssim
	\mathrm{H}_{\ell}^k\), while the converse estimate follows from the Céa
	lemma~\eqref{eq:cea}.

\end{remark}

\begin{remark}
	The work~\cite{ghps2021} extends the ideas of~\cite{ckns2008} (that proves~\eqref{eq:intro:ckns} for AFEM with exact solver) and of~\cite{fp2018} (that extends~\eqref{eq:intro:ckns} to the final iterates for AFEM with contractive solver).
	For the scalar product $b(\cdot, \cdot) = a(\cdot,\cdot)$ and arbitrary stopping parameters \(\lambda > 0\), it shows that the quasi-error \(\Delta_\ell^k\) from Remark~\ref{rem:equivalence_Delta_Eta} satisfies contraction
	\begin{subequations}\label{eq:single:contraction}
		\begin{alignat}{2}
			\Delta_\ell^{k}   & \le q_{\textup{ctr}} \, \Delta_\ell^{k-1}
			\quad\quad        &                                                       & \text{for all } (\ell,k) \in \mathcal{Q} \text{ with } 0 <
			k <
			\underline{k}[\ell],
			\\
			\label{eq-b:single:contraction}
			\Delta_{\ell+1}^0 & \le q_{\textup{ctr}} \, \Delta_\ell^{\underline{k}-1}
			\quad             &                                                       & \text{for all } (\ell,\underline{k}) \in \mathcal{Q}
		\end{alignat}
	\end{subequations}
	with contraction constant $0 < q_{\textup{ctr}} < 1$, along the approximations $u_\ell^k \in \mathcal{X}_\ell$ generated by Algorithm~\ref{algorithm:single}.
	The proof of~\eqref{eq:single:contraction} can be generalized similarly to Remark~\ref{rem:Rlinear}, see~\cite{aisfem}:
	With the quasi-Pythagorean estimate~\eqref{eq:general-pythagoras}, the
	contraction~\eqref{eq:single:contraction} transfers to general second-order linear elliptic
	PDEs~\eqref{eq:strongform} under the restriction that~\eqref{eq-b:single:contraction}
	holds only for all $\ell \ge \ell_0$, where
	$\ell_0 \in \mathbb{N}_0$
	exists, but is unknown in practice.
	While, as noted in Remark~\ref{rem:Rlinear}, contraction~\eqref{eq:single:contraction} implies
	full R-linear convergence~\eqref{eq:single:convergence}, the %presented 
	proof of
	Theorem~\ref{theorem:single:convergence} works under much weaker assumptions than that of~\cite{ghps2021} and covers the PDE~\eqref{eq:strongform} with $\ell_0=0$.
\end{remark}

The proof of
Theorem~\ref{theorem:single:convergence} relies on
Lemma~\ref{lemma:summability:criterion} and
the following elementary result essentially taken from~\cite[Lemma~4.9]{cfpp2014}.
Its proof is found in~\ref{appendix}.

\begin{lemma}[tail summability vs.\ R-linear
		convergence]\label{lemma:summability}
	Let $(a_\ell)_{\ell \in \mathbb{N}_0}$ be a scalar sequence in $\mathbb{R}_{\ge0}$ and $m > 0$. Then, the
	following statements are equivalent:
	\begin{itemize}\label{eq:summability}
		\item[\textup{(i)}] \textbf{tail summability:} There exists a constant $C_m >
			      0$ such that
		      \begin{equation}\label{eq:tail-summability}
			      \sum_{\ell' = \ell+1}^\infty a_{\ell'}^m \le C_m \, a_\ell^m
			      \quad \text{for all } \ell \in \mathbb{N}_0.
		      \end{equation}
		\item[\textup{(ii)}] \textbf{R-linear convergence:} There
		      holds~\eqref{eq:Rlinear:convergence} with certain
		      $0 < q_{\textup{lin}} < 1$ and $C_{\textup{lin}} > 0$.
	\end{itemize}
\end{lemma}

\begin{proof}[\textbf{Proof of Theorem~\ref{theorem:single:convergence}}]
	We want to briefly summarize the proof strategy. First, we show that the estimator reduction together with the contraction~\eqref{eq:single:contraction} of the algebraic solver leads to tail-summability of a weighted quasi-error on the mesh level \(\ell\). Second, we show that the quasi-error from~\eqref{eq2:single:quasi-error} is contractive in the algebraic solver index \(k\) and is stable in the nested iteration. Finally, we combine these two steps to prove R-linear convergence of the quasi-error~\eqref{eq:single:convergence}.

	The proof is split into two steps.

	\textbf{ Step~1 (tail summability with respect to $\boldsymbol{\ell}$).}
	Let $\ell \in \mathbb{N}$ with $(\ell+1,\underline{k}) \in \mathcal{Q}$. Algorithm~\ref{algorithm:single} guarantees nested iteration $u_{\ell+1}^0 = u_\ell^{\underline{k}}$ and $\underline{k}[\ell+1] \ge 1$. This and contraction of the algebraic solver~\eqref{eq:contractive-solver} show
	\begin{equation}\label{eq1:single:convergence}
		|\mkern-1.5mu|\mkern-1.5mu| u_{\ell+1}^\star - u_{\ell+1}^{\underline{k}} |\mkern-1.5mu|\mkern-1.5mu|
		\stackrel{\eqref{eq:contractive-solver}}\le
		q_{\textup{alg}}^{\underline{k}[\ell+1]} \, |\mkern-1.5mu|\mkern-1.5mu| u_{\ell+1}^\star - u_\ell^{\underline{k}} |\mkern-1.5mu|\mkern-1.5mu|
		\le
		q_{\textup{alg}} \, |\mkern-1.5mu|\mkern-1.5mu| u_{\ell+1}^\star - u_\ell^{\underline{k}} |\mkern-1.5mu|\mkern-1.5mu|
	\end{equation}
	As in the proof of Theorem~\ref{theorem:exact:convergence}, one obtains the estimator reduction
	\begin{equation}\label{eq:single:estimator-reduction}
		\eta_{\ell+1}(u_{\ell+1}^{\underline{k}})
		\stackrel{\mathclap{\eqref{eq:estimator-reduction}}}\le
		q_\theta \, \eta_\ell(u_\ell^{\underline{k}}) + C_{\textup{stab}} \, |\mkern-1.5mu|\mkern-1.5mu| u_{\ell+1}^{\underline{k}} \!-\! u_\ell^{\underline{k}} |\mkern-1.5mu|\mkern-1.5mu|
		\stackrel{\mathclap{\eqref{eq1:single:convergence}}}\le
		q_\theta \, \eta_\ell(u_\ell^{\underline{k}}) + (q_{\textup{alg}} + 1)  C_{\textup{stab}} \,
		|\mkern-1.5mu|\mkern-1.5mu| u_{\ell+1}^\star \!-\! u_\ell^{\underline{k}} |\mkern-1.5mu|\mkern-1.5mu|.
	\end{equation}
	Choosing $0 < \gamma \le 1$ with $0 < q_{\textup{ctr}} \coloneqq \max \{ q_{\textup{alg}} + (q_{\textup{alg}}+1) C_{\textup{stab}} \gamma \,,\, q_\theta\} < 1$,
	the combination of~\eqref{eq1:single:convergence}--\eqref{eq:single:estimator-reduction} reads
	\begin{align}\label{eq:a_ell:perturbed-contraction}
		\begin{split}
			 & a_{\ell+1}
			\coloneqq
			|\mkern-1.5mu|\mkern-1.5mu| u_{\ell+1}^\star - u_{\ell+1}^{\underline{k}} |\mkern-1.5mu|\mkern-1.5mu| + \gamma \,
			\eta_{\ell+1}(u_{\ell+1}^{\underline{k}})
			\le
			q_{\textup{ctr}} \, \big[ |\mkern-1.5mu|\mkern-1.5mu| u_{\ell+1}^\star - u_\ell^{\underline{k}} |\mkern-1.5mu|\mkern-1.5mu| + \gamma \,
				\eta_\ell(u_\ell^{\underline{k}}) \big]
			\\& \quad
			\le
			q_{\textup{ctr}} \, \big[ |\mkern-1.5mu|\mkern-1.5mu| u_{\ell}^\star - u_\ell^{\underline{k}} |\mkern-1.5mu|\mkern-1.5mu| + \gamma \,
				\eta_\ell(u_\ell^{\underline{k}}) \big]
			+ q_{\textup{ctr}} \, |\mkern-1.5mu|\mkern-1.5mu| u_{\ell+1}^\star - u_\ell^\star |\mkern-1.5mu|\mkern-1.5mu|
			\eqqcolon
			q_{\textup{ctr}} \, a_\ell + b_\ell.
		\end{split}
	\end{align}
	Moreover, estimate~\eqref{eq:quasimonotonicity_b} from the proof of
	Theorem~\ref{theorem:exact:convergence} and
	stability~\eqref{axiom:stability} prove that
	\begin{equation}\label{eq:cea+reliability}
		|\mkern-1.5mu|\mkern-1.5mu| u_{\ell''}^\star - u_{\ell'}^\star |\mkern-1.5mu|\mkern-1.5mu|
		\stackrel{\mathclap{\eqref{eq:quasimonotonicity_b}}}
		\lesssim
		\eta_\ell(u_\ell^\star)
		\stackrel{\mathclap{\eqref{axiom:stability}}}\lesssim
		|\mkern-1.5mu|\mkern-1.5mu| u_\ell^\star - u_\ell^{\underline{k}} |\mkern-1.5mu|\mkern-1.5mu|
		+ \eta_\ell(u_\ell^{\underline{k}})
		\simeq a_\ell
		\,\, \text{for} \,\, \ell \le \ell' \le \ell'' \le \underline{\ell}
		\,\, \text{with} \,\, (\ell,\underline{k}) \in \mathcal{Q},
	\end{equation}
	which yields $b_{\ell+N} \lesssim a_\ell$ for all $0 \le \ell \le \ell+N \le \underline{\ell}$ with $(\ell,\underline{k}) \in
		\mathcal{Q}$.
	As in~\eqref{eq:summability_b}, we see
	\begin{align}\label{eq:single:orthogonality}
		\begin{split}
			 & \sum_{\ell' = \ell}^{\ell+N} b_{\ell'}^2
			\simeq
			\sum_{\ell' = \ell}^{\ell+N} |\mkern-1.5mu|\mkern-1.5mu| u_{\ell'+1}^\star - u_{\ell'}^\star |\mkern-1.5mu|\mkern-1.5mu|^2
			\stackrel{\eqref{axiom:orthogonality}}\lesssim
			(N+1)^{1-\delta} \, |\mkern-1.5mu|\mkern-1.5mu| u^\star - u_\ell^\star |\mkern-1.5mu|\mkern-1.5mu|^2
			\stackrel{\eqref{axiom:reliability}}\lesssim
			(N+1)^{1-\delta} \, \eta_\ell(u_\ell^\star)^2
			\\& \,
			\stackrel{\eqref{axiom:stability}}\lesssim
			(N+1)^{1-\delta} \, \big[ \eta_\ell(u_\ell^{\underline{k}}) + |\mkern-1.5mu|\mkern-1.5mu| u_\ell^\star
				- u_\ell^{\underline{k}} |\mkern-1.5mu|\mkern-1.5mu| \big]^2
			\simeq (N+1)^{1-\delta} \, a_\ell^2
			\,\,\, \text{for all } 0 \le \ell \le \ell+N < \underline{\ell}.
		\end{split}
	\end{align}
	Hence, the assumptions~\eqref{eq:summability:criterion} are satisfied and
	Lemma~\ref{lemma:summability:criterion} concludes tail summability (or equivalently R-linear convergence by Lemma~\ref{lemma:summability}) of $\mathrm{H}_\ell^{\underline{k}} \simeq a_\ell$, i.e.,
	\begin{equation}\label{eq:single:summability-ell}
		\boxed{\sum_{\ell' = \ell+1}^{\underline{\ell} - 1} \mathrm{H}_{\ell'}^{\underline{k}}
		\lesssim \mathrm{H}_\ell^{\underline{k}}
		\quad \text{for all } 0 \le \ell < \underline{\ell}.}
	\end{equation}

	\textbf{Step~2 (tail summability with respect to $\boldsymbol{\ell}$ and $\boldsymbol{k}$).}
	First, for $0 \le k < k' < \underline{k}[\ell]$, the failure of the termination criterion~\eqref{eq:single:termination} and contraction of the solver~\eqref{eq:contractive-solver} prove that
	\begin{equation*}
		\mathrm{H}_\ell^{k'}
		\stackrel{\eqref{eq:single:termination}}\lesssim
		|\mkern-1.5mu|\mkern-1.5mu| u_\ell^\star - u_\ell^{k'} |\mkern-1.5mu|\mkern-1.5mu| + |\mkern-1.5mu|\mkern-1.5mu| u_\ell^{k'} - u_\ell^{k'-1} |\mkern-1.5mu|\mkern-1.5mu|
		\stackrel{\eqref{eq:contractive-solver}}\lesssim
		|\mkern-1.5mu|\mkern-1.5mu| u_\ell^\star - u_\ell^{k'-1} |\mkern-1.5mu|\mkern-1.5mu|
		\stackrel{\eqref{eq:contractive-solver}}\lesssim
		q_{\textup{alg}}^{k'-k} \, |\mkern-1.5mu|\mkern-1.5mu| u_\ell^\star - u_\ell^{k} |\mkern-1.5mu|\mkern-1.5mu|
		\stackrel{\eqref{eq2:single:quasi-error}}\le
		q_{\textup{alg}}^{k'-k} \, \mathrm{H}_\ell^k.
	\end{equation*}
	Second, for $(\ell,\underline{k}) \in \mathcal{Q}$, it holds that
	\begin{align*}
		\begin{split}
			\mathrm{H}_\ell^{\underline{k}}
			 & \, \stackrel{\mathclap{\eqref{axiom:stability}}}\lesssim \,
			|\mkern-1.5mu|\mkern-1.5mu|u_\ell^\star - u_\ell^{\underline{k}} |\mkern-1.5mu|\mkern-1.5mu| + \eta_\ell(u_\ell^{\underline{k}-1}) +
			|\mkern-1.5mu|\mkern-1.5mu|u_\ell^{\underline{k}} - u_\ell^{\underline{k}-1} |\mkern-1.5mu|\mkern-1.5mu|
			\\&
			\le
			\mathrm{H}_\ell^{\underline{k}-1} + 2 \, |\mkern-1.5mu|\mkern-1.5mu|u_\ell^\star - u_\ell^{\underline{k}} |\mkern-1.5mu|\mkern-1.5mu|
			\stackrel{\eqref{eq:contractive-solver}}\le (1 + 2 \,q_{\textup{alg}}) \, \mathrm{H}_\ell^{\underline{k}-1}
			\quad \text{for all } (\ell,\underline{k}) \in \mathcal{Q}.
		\end{split}
	\end{align*}
	Hence, we may conclude
	\begin{equation}\label{eq:single:contraction-Lambda}
		\boxed{\mathrm{H}_\ell^{k'}
			\lesssim
			q_{\textup{alg}}^{k'-k} \, \mathrm{H}_\ell^k
			\quad \text{for all } 0 \le k \le k' \le \underline{k}[\ell].}
	\end{equation}
	With
	$|\mkern-1.5mu|\mkern-1.5mu| u_{\ell+1}^\star - u_\ell^\star |\mkern-1.5mu|\mkern-1.5mu|
		\lesssim
		a_\ell \simeq \mathrm{H}_\ell^{\underline{k}}$
	from~\eqref{eq:quasimonotonicity_b}, stability~\eqref{axiom:stability} and reduction~\eqref{axiom:reduction} show
	\begin{equation}\label{eq2:single:contraction-Lambda}
		\boxed{\mathrm{H}_{\ell+1}^0 = |\mkern-1.5mu|\mkern-1.5mu| u_{\ell+1}^\star - u_\ell^{\underline{k}} |\mkern-1.5mu|\mkern-1.5mu| +
			\eta_{\ell+1}(u_\ell^{\underline{k}})
			\le \mathrm{H}_\ell^{\underline{k}} + |\mkern-1.5mu|\mkern-1.5mu| u_{\ell+1}^\star - u_\ell^\star |\mkern-1.5mu|\mkern-1.5mu|
			\lesssim \mathrm{H}_\ell^{\underline{k}}
			\quad \text{for all } (\ell,\underline{k}) \in \mathcal{Q}.}
	\end{equation}
	Overall, the geometric series proves tail summability~\eqref{eq:tail-summability} via
	\begin{align*}
		\sum_{\substack{(\ell',k') \in \mathcal{Q} \\ |\ell',k'|> |\ell,k|}} \mathrm{H}_{\ell'}^{k'}
		 & =
		\sum_{k'=k+1}^{\underline{k}[\ell]} \mathrm{H}_\ell^{k'}
		+ \sum_{\ell' = \ell+1}^{\underline{\ell}} \sum_{k'=0}^{\underline{k}[\ell']} \mathrm{H}_{\ell'}^{k'}
		\\[-2mm]&
		\stackrel{\mathclap{\eqref{eq:single:contraction-Lambda}}}\lesssim
		\mathrm{H}_\ell^k + \sum_{\ell' = \ell+1}^{\underline{\ell}} \mathrm{H}_{\ell'}^0
		\stackrel{\eqref{eq2:single:contraction-Lambda}}\lesssim
		\mathrm{H}_\ell^k + \sum_{\ell' = \ell}^{\underline{\ell}-1} \mathrm{H}_{\ell'}^{\underline{k}}
		\stackrel{\eqref{eq:single:summability-ell}}\lesssim
		\mathrm{H}_\ell^k + \mathrm{H}_\ell^{\underline{k}}
		\stackrel{\eqref{eq:single:contraction-Lambda}}\lesssim
		\mathrm{H}_\ell^k
		\quad \text{for all } (\ell,k) \in \mathcal{Q}.
	\end{align*}
	Since $\mathcal{Q}$ is countable and linearly ordered, Lemma~\ref{lemma:summability} concludes the proof of~\eqref{eq:single:convergence}.
\end{proof}

The following comments on the computational cost of implementations of standard finite element methods
underline the importance of full linear convergence~\eqref{eq:single:convergence}.
\begin{itemize}
	\item \textbf{\textit{Solve} \& \textit{Estimate.}} One solver
	      step of an optimal multigrid
	      method can be performed in $\mathcal{O}(\#\mathcal{T}_\ell)$ operations, if smoothing is done according to the grading of the mesh as in~\cite{wz2017, imps2022}. Instead, one step of a
	      multigrid method on $\mathcal{T}_\ell$, where smoothing is done on all levels and all
	      vertex patches needs $\mathcal{O}(\sum_{\ell'=0}^\ell \#\mathcal{T}_{\ell'})$ operations.
	      The same remark is valid for the preconditioned CG method
	      with optimal additive
	      Schwarz or BPX preconditioner from~\cite{cnx2012}. One solver step
	      can be realized via successive updates in \(\mathcal{O}(\#\mathcal{T}_{\ell})\) operations, while
	      \(\mathcal{O}\bigl(\sum_{\ell' = 0}^\ell \#\mathcal{T}_{\ell'}\bigr)\) is faced if the
	      preconditioner does not respect the grading of the mesh hierarchy.
	\item  \textbf{\textit{Mark.}} The D\"orfler marking strategy~\eqref{eq:doerfler} can be realized in linear complexity $\mathcal{O}(\#\mathcal{T}_\ell)$; see~\cite{stevenson2007} for $C_{\textup{mark}} = 2$ and~\cite{pp2020} for $C_{\textup{mark}} = 1$.
	\item  \textbf{\textit{Refine.}} Local mesh refinement (including mesh closure) of
	      $\mathcal{T}_\ell$ by bisection can be realized in $\mathcal{O}(\#\mathcal{T}_\ell)$ operations;
	      see, e.g., \cite{bdd2004, stevenson2007}.
\end{itemize}
Since the adaptive algorithm depends on the full history of
algorithmic decisions, the overall computational cost until step $(\ell,k) \in
	\mathcal{Q}$, i.e., until (and including) the computation of \(u_\ell^k\),
is thus
proportionally bounded by
\begin{equation*}
	\sum_{\substack{(\ell',k') \in \mathcal{Q} \\ |\ell',k'| \le |\ell,k|}} \#\mathcal{T}_{\ell'}
	\le
	\mathtt{cost}(\ell,k)
	\le
	\sum_{\substack{(\ell',k') \in \mathcal{Q} \\ |\ell',k'| \le |\ell,k|}}  \sum_{\ell'' = 0}^{\ell'} \#\mathcal{T}_{\ell''}.
\end{equation*}
Here, the lower bound corresponds to the case that all steps of
Algorithm~\ref{algorithm:single} are done at linear cost $\mathcal{O}(\#\mathcal{T}_\ell)$. The
upper bound corresponds to the case that \textit{solve} \& \textit{estimate},
\textit{mark}, and
\textit{refine} are performed at linear cost $\mathcal{O}(\#\mathcal{T}_\ell)$, while a suboptimal solver
leads to cost $\mathcal{O}(\sum_{\ell'' = 0}^{\ell}
	\#\mathcal{T}_{\ell''})$ for each mesh $\mathcal{T}_\ell$. In any case, the following
corollary shows that full R-linear convergence guarantees that convergence
rates with respect to the number of degrees of freedom $\dim
	\mathcal{X}_\ell \simeq \#\mathcal{T}_\ell$ and with respect to the
overall computational cost $\mathtt{cost}(\ell,k)$ coincide even for a suboptimal solver. Moreover,
the corollary shows that, under R-linear convergence, the optimal convergence rate with respect to the number of degrees of freedom as well as with respect to the overall computational cost is non-zero.

\begin{corollary}[rates = complexity]\label{corollary:rates:complexity}
	For $s > 0$, full R-linear convergence~\eqref{eq:single:convergence} yields
	\begin{equation*}
		M(s)
		\coloneqq
		\sup_{(\ell,k) \in \mathcal{Q}} (\#\mathcal{T}_\ell)^s \, \mathrm{H}_\ell^k
		\le \sup_{(\ell,k) \in \mathcal{Q}} \Big(\sum_{\substack{(\ell',k') \in \mathcal{Q} \\ |\ell',k'| \le |\ell,k|}}  \sum_{\ell'' =
			0}^{\ell'} \#\mathcal{T}_{\ell''}\Big)^s \mathrm{H}_\ell^k
		\le C_{\textup{cost}}(s) \, M(s),
	\end{equation*}
	where the constant $C_{\textup{cost}}(s) > 0$ depends only on $C_{\textup{lin}}$,
	$q_{\textup{lin}}$, and $s$.
	Moreover, there exists $s_0 > 0$ such that $M(s) < \infty$ for all $0 < s \le s_0$ with
	\(s_0
	= \infty\) if \(\underline{\ell} < \infty\).
\end{corollary}

The last corollary is an immediate consequence of the following elementary lemma for
\(a_{| \ell, k |} \coloneqq \mathrm{H}_\ell^k\) and \(t_{| \ell, k | } \coloneqq \# \mathcal{T}_\ell\).

\begin{lemma}[rates = complexity criterion]\label{lemma:complexity:sequence}
	Let \((a_\ell)_{\ell \in \mathbb{N}_0}\) and \((t_\ell)_{\ell \in \mathbb{N}_0}\) be sequences in \(\mathbb{R}_{\ge 0}\) such that
	\begin{equation}\label{eq:Rlinear:sequence}
		a_{\ell+n}
		\le
		C_1 q^n \,a_\ell
		\quad
		\text{and}
		\quad
		t_{\ell+1} \le C_2 \, t_{\ell}
		\quad
		\text{for all \(\ell, n \in \mathbb{N}_0\)}.
	\end{equation}
	Then, for all \(s > 0\), there holds
	\begin{equation}\label{eq:single:complexity_criterion}
		M(s)
		\coloneqq
		\sup_{\ell \in \mathbb{N}_0} t_\ell^s \, a_\ell
		\le
		\sup_{\ell \in \mathbb{N}_0} \Big(\sum_{\ell' = 0}^\ell  \sum_{\ell'' =
			0}^{\ell'} t_{\ell''}\Big)^s a_\ell
		\le C_{\textup{cost}}(s) \, M(s),
	\end{equation}
	where the constant $C_{\textup{cost}}(s) > 0$ depends only on $C_1$,
	$q$, and $s$.
	Moreover, there exists $s_0 > 0$ depending only on \(C_2\) and \(q\) such that $M(s) <
		\infty$ for all $0 < s \le s_0$.
\end{lemma}

\begin{proof}
	By definition, it holds that
	\begin{equation*}
		t_\ell \le M(s)^{1/s} \, a_\ell^{-1/s}
		\quad \text{for all } \ell \in \mathbb{N}_0.
	\end{equation*}
	This, assumption~\eqref{eq:Rlinear:sequence}, and the geometric series
	prove that
	\begin{align*}
		\sum_{\ell'' = 0}^{\ell'} t_{\ell''}
		\le
		M(s)^{1/s} \, \sum_{\ell'' = 0}^{\ell'} a_{\ell''}^{-1/s}
		 & \stackrel{\mathclap{\eqref{eq:Rlinear:sequence}}}\le
		M(s)^{1/s} \, C_1^{1/s} a_{\ell'}^{-1/s} \, \sum_{\ell'' = 0}^{\ell'} (q^{1/s})^{\ell'- \ell''}
		\\&
		\le
		M(s)^{1/s} \, \frac{C_1^{1/s}}{1-q^{1/s}} \,
		a_{\ell'}^{-1/s}
		\quad \text{for all } \ell' \in \mathbb{N}_0.
	\end{align*}
	A further application of~\eqref{eq:Rlinear:sequence} and the geometric series prove that
	\begin{equation*}
		\sum_{\ell' = 0}^\ell a_{\ell'}^{-1/s}
		\stackrel{\eqref{eq:Rlinear:sequence}}\le
		C_1^{1/s} a_\ell^{-1/s} \, \sum_{\ell' = 0}^\ell (q^{1/s})^{\ell - \ell'}
		\le \frac{C_1^{1/s}}{1-q^{1/s}} \, a_\ell^{-1/s}
		\quad \text{for all } \ell \in \mathbb{N}_0.
	\end{equation*}
	The combination of the two previously displayed formulas results in
	\begin{equation*}
		\sum_{\ell' = 0}^\ell  \sum_{\ell'' = 0}^{\ell'} t_{\ell''}
		\le
		\Bigl(\frac{C_1^{1/s}}{1-q^{1/s}}\Bigr)^2 \,
		M(s)^{1/s} \, a_\ell^{-1/s}
		\quad \text{for all } \ell \in \mathbb{N}_0.
	\end{equation*}
	Rearranging this estimate, we conclude the proof of~\eqref{eq:single:complexity_criterion}.
	It remains to verify $M(s) < \infty$ for some $s > 0$.
	Note that \eqref{eq:Rlinear:sequence} guarantees that
	\begin{equation*}
		0 \le t_{\ell} \le C_2 \, t_{\ell-1} \le C_2^\ell \, t_0
		\quad \text{for all } \ell \in \mathbb{N}.
	\end{equation*}
	Moreover, R-linear convergence~\eqref{eq:Rlinear:sequence} yields that
	\begin{equation*}
		0 \le  a_\ell
		\stackrel{\eqref{eq:Rlinear:sequence}}\le
		C_1 q^{\ell} \, a_0
		\quad \text{for all } \ell \in \mathbb{N}_0.
	\end{equation*}
	Multiplying the two previously displayed formulas, we see that
	\begin{equation*}
		t_\ell^s \, a_\ell
		\le
		(C_2^s q)^\ell C_1 \, t_0^s \,a_0
		\quad \text{for all } \ell \in \mathbb{N}_0.
	\end{equation*}
	Note that the right-hand side is uniformly bounded, provided that $s > 0$ guarantees
	$C_2^s q \le 1$. This concludes the proof with $s_0 \coloneqq \log(1/q) /
		\log(C_2)$.
\end{proof}

With full linear convergence~\eqref{eq:single:convergence}, the following theorem from~\cite[Theorem~8]{ghps2021} can be applied and thus Algorithm~\ref{algorithm:single} guarantees optimal convergence rates
with respect to the overall computational cost in the case of sufficiently
small adaptivity parameters $\theta$ and $\lambda$.
To formalize the notion of \emph{achievable convergence rates} $s>0$, we introduce
nonlinear approximation classes from~\cite{bdd2004, stevenson2007, ckns2008,cfpp2014}
\begin{equation*}
	\| u^\star \|_{\mathbb{A}_s}
	\coloneqq
	\sup_{N \in \mathbb{N}_0}
	\Bigl(
	\bigl( N+1 \bigr)^s
	\min_{\mathcal{T}_{\textup{opt}} \in \mathbb{T}_N } \eta_{\textup{opt}}(u^\star_{\textup{opt}})
	\Bigr),
\end{equation*}
where $\eta_{\textup{opt}}(u^\star_{\textup{opt}})$ is the estimator for
the (unavailable) exact Galerkin solution \(u_{\textup{opt}}^\star\) on
an optimal $\mathcal{T}_{\textup{opt}} \in \mathbb{T}_N \coloneqq \{\mathcal{T}_H \in \mathbb{T} \colon \# \mathcal{T}_H - \# \mathcal{T}_0 \le N\}$.
\begin{theorem}[optimal complexity of Algorithm~\ref{algorithm:single}, {\cite[Theorem~8]{ghps2021}}]\label{theorem:single:optimal}
	Suppose that the estimator satisfies the axioms of
	adaptivity~\eqref{axiom:stability}, \eqref{axiom:reduction}, \eqref{axiom:discrete_reliability},
	and suppose that quasi-orthogonality~\eqref{axiom:orthogonality} holds.
	Suppose that the parameters \( \theta \) and \(\lambda\) are chosen such that
	\begin{equation*}
		0 < \lambda < \lambda^\star = \min \Bigl\{ 1, \frac{1 - q_{\textup{alg}}}{q_{\textup{alg}}} \, C_{\textup{stab}}^{-1} \Bigr\}
		\quad \text{and} \
		0 < \frac{(\theta^{1/2} + \lambda / \lambda^\star)^2}{\bigl(1-\lambda / \lambda^\star\bigr)^2}
		<
		\theta^\star
		\coloneqq
		\bigl( 1 + C_{\textup{stab}}^2 \, C_{\textup{drel}}^2 \bigr)^{-1}.
	\end{equation*}
	Then, Algorithm~\ref{algorithm:single} guarantees for all \( s > 0 \) that
	\begin{equation*}
		c_{\textup{opt}} \| u^\star \|_{\mathbb{A}_s}
		\le
		\sup_{(\ell,k) \in \mathcal{Q}} \Bigl( \sum \limits_{\substack{ (\ell',k') \in \mathcal{Q} \\ | \ell',k' | \le | \ell,k | }} \# \mathcal{T}_{\ell'} \Bigr)^s \, \mathrm{H}_{\ell}^{k}
		\le
		C_{\textup{opt}} \, \max \{ \| u^\star \|_{\mathbb{A}_s}, \mathrm{H}_0^0 \}.
	\end{equation*}
	The constant \(c_{\textup{opt}} > 0\) depends only on~\(C_{\textup{stab}}\), the use of NVB refinement, and \(s\), while \( C_{\textup{opt}} > 0 \) depends only on~\(C_{\textup{stab}}\), \(q_{\textup{red}}\), \(C_{\textup{drel}}\), \(C_{\textup{lin}}\), \(q_{\textup{lin}}\), \(\# \mathcal{T}_0\), \(\lambda\), \(q_{\textup{alg}}\), \(\theta\), \(s\), and the use of NVB refinement. \qed
\end{theorem}

\begin{remark}
	Considering the nonsymmetric model problem~\eqref{eq:strongform}, a natural candidate for the
	solver is the generalized minimal residual method (GMRES) with optimal preconditioner for the
	symmetric part.
	Another alternative would be to consider an optimal preconditioner for the symmetric part and apply a conjugate gradient method to the normal equations~(CGNR).
	However, for both approaches, \textsl{a~posteriori} error estimation and
	contraction in the PDE-related energy norm are still open.
	Instead,~\cite{aisfem} follows the constructive
	proof
	of the Lax--Milgram lemma to
	derive a contractive solver. Its convergence analysis, as given in~\cite{aisfem}, is improved in the
	following Section~\ref{section:double}.
\end{remark}

%%%%%%%%%%%%%%%%%%%%%%%%%%%%%%%%%%%%%%%%%%%%%%%%%%%%%%%%%%%%%%%%%%%%%%%%%%%%%%%%%%%
%%%%%%%%%%%%%%%%%%%%%%%%%%%%%%%%%%%%%%%%%%%%%%%%%%%%%%%%%%%%%%%%%%%%%%%%%%%%%%%%%%%
\section{AFEM with nested contractive solvers}
\label{section:double}
%%%%%%%%%%%%%%%%%%%%%%%%%%%%%%%%%%%%%%%%%%%%%%%%%%%%%%%%%%%%%%%%%%%%%%%%%%%%%%%%%%%
%%%%%%%%%%%%%%%%%%%%%%%%%%%%%%%%%%%%%%%%%%%%%%%%%%%%%%%%%%%%%%%%%%%%%%%%%%%%%%%%%%%

While contractive solvers for SPD systems are well-understood in the literature,
the recent work~\cite{aisfem} presents contractive solvers for the nonsymmetric variational formulation~\eqref{eq:discrete}
that
essentially fit into the framework of Section~\ref{section:single} and allow for the numerical
analysis of AFEM with optimal complexity.
To this end, the proof of the Lax--Milgram lemma as proposed by~\cite{Zarantonello1960} is
exploited algorithmically (while the original proof in~\cite{pl1954} relies on the
Hahn--Banach separation theorem): For $\delta > 0$,
we consider the Zarantonello mapping
$\Phi_H(\delta; \cdot) \colon \mathcal{X}_H \to
	\mathcal{X}_H$ defined by
\begin{equation}\label{eq:double:zarantonello}
	a(\Phi_H(\delta; u_H), v_H)
	=
	a(u_H,v_H) + \delta \big[ F(v_H) - b(u_H, v_H) \big]
	\quad \text{for all } u_H, v_H \in \mathcal{X}_H.
\end{equation}
Since $a(\cdot, \cdot)$ is a scalar product, $\Phi_H(\delta; u_H)
	\in \mathcal{X}_H$ is well-defined. Moreover, for any $0 < \delta <
	2\alpha/L^2$ and $0 < q_{\textup{sym}}^\star \coloneqq [1 - \delta(2\alpha-\delta L^2)]^{1/2} < 1$, this mapping
is contractive, i.e.,
\begin{equation}\label{eq:double:zarantonello:contraction}
	|\mkern-1.5mu|\mkern-1.5mu| u_H^\star - \Phi_H(\delta; u_H) |\mkern-1.5mu|\mkern-1.5mu|
	\le q_{\textup{sym}}^\star \, |\mkern-1.5mu|\mkern-1.5mu| u_H^\star - u_H |\mkern-1.5mu|\mkern-1.5mu|
	\quad \text{for all } u_H \in \mathcal{X}_H;
\end{equation}
see also~\cite{hw2020a,hw2020b}.
Note that~\eqref{eq:double:zarantonello} corresponds to a linear SPD system. For this, we employ a uniformly contractive algebraic solver with
iteration function $\Psi_H(u_H^\sharp;\cdot) \colon \mathcal{X}_H \to \mathcal{X}_H$ to approximate the
solution $u_H^\sharp \coloneqq \Phi_H(\delta; u_H)$ to the SPD system~\eqref{eq:double:zarantonello}, i.e.,
\begin{equation}\label{eq:double:contractive-solver}
	|\mkern-1.5mu|\mkern-1.5mu| u_H^\sharp - \Psi_H(u_H^\sharp; w_H) |\mkern-1.5mu|\mkern-1.5mu|
	\le
	q_{\textup{alg}} \, |\mkern-1.5mu|\mkern-1.5mu| u_H^\sharp - w_H |\mkern-1.5mu|\mkern-1.5mu|
	\quad \text{for all } w_H \in \mathcal{X}_H \text{ and all } \mathcal{T}_H \in \mathbb{T},
\end{equation}
where $0 < q_{\textup{alg}} < 1$ depends only on $a(\cdot,\cdot)$, but is independent of $\mathcal{X}_H$.
Clearly, no knowledge of $u_H^\sharp$ is needed
to compute $\Psi_H(u_H^\sharp; w_H)$ but only that of the corresponding right-hand
side \(a(u_H^\sharp, \cdot) \colon \mathcal{X}_H \to \mathbb{R}\); see, e.g.,~\cite{cnx2012,wz2017,imps2022}.

\begin{algorithm}[AFEM with nested contractive solvers]\label{algorithm:double}
	Given an initial mesh $\mathcal{T}_0$, the Zarantonello parameter $\delta > 0$, adaptivity parameters $0 < \theta \le 1$ and $C_{\textup{mark}} \ge 1$, solver-stopping parameters $\lambda_{\textup{sym}}, \lambda_{\textup{alg}} > 0$, and an initial guess $u_0^{0,0} \coloneqq u_0^{0,\underline{j}} \in \mathcal{X}_0$, iterate the following steps~\ref{alg:double:i}--\ref{alg:double:iv} for all $\ell = 0, 1, 2, 3, \dots$:
	\begin{enumerate}[label=(\roman*), ref = {\rm (\roman*)}, font = \upshape]
		\item\label{alg:double:i} \textbf{Solve \& estimate:} For all $k = 1, 2, 3, \dots$, repeat the following steps~\ref{alg:double:a}--\ref{alg:double:c} until
		      \begin{equation}\label{eq:double:termination:sym}
			      |\mkern-1.5mu|\mkern-1.5mu| u_\ell^{k,\underline{j}} - u_\ell^{k-1,\underline{j}} |\mkern-1.5mu|\mkern-1.5mu|
			      \le
			      \lambda_{\textup{sym}} \, \eta_\ell(u_\ell^{k,\underline{j}})
		      \end{equation}
		      \begin{enumerate}[label=(\alph*), ref = {\rm (\alph*)},
				      font = \upshape]
			      \item\label{alg:double:a} Define \(u_\ell^{k, 0} \coloneqq u_\ell^{k-1,\underline{j}}\) and, only as a theoretical quantity, \(u_\ell^{k,\star} \coloneqq \Phi_\ell(\delta; u_\ell^{k-1,\underline{j}})\).
			      \item\label{alg:double:b} \textbf{Inner solver loop:} For all $j = 1, 2, 3, \dots$, repeat the steps~\ref{alg:double:I}--\ref{alg:double:II} until
			            \begin{equation}\label{eq:double:termination:alg}
				            |\mkern-1.5mu|\mkern-1.5mu| u_\ell^{k,j} - u_\ell^{k,j-1} |\mkern-1.5mu|\mkern-1.5mu|
				            \le
				            \lambda_{\textup{alg}} \, \big[ \lambda_{\textup{sym}} \eta_\ell(u_\ell^{k,j}) + |\mkern-1.5mu|\mkern-1.5mu|
					            u_\ell^{k,j} - u_\ell^{k-1,\underline{j}} |\mkern-1.5mu|\mkern-1.5mu| \big].
			            \end{equation}
			            \begin{enumerate}[label=(\Roman*), ref = {\rm (\Roman*)}, font = \upshape]
				            \item\label{alg:double:I} Compute one step of the contractive SPD solver $u_\ell^{k,j} \coloneqq \Psi_\ell(u_\ell^{k,\star}; u_\ell^{k,j-1})$.
				            \item\label{alg:double:II} Compute the refinement indicators $\eta_\ell(T, u_\ell^{k,j})$ for all $T \in \mathcal{T}_\ell$.
			            \end{enumerate}
			      \item\label{alg:double:c} Upon termination of the inner solver loop, define the index
			            $\underline{j}[\ell,k] \coloneqq j \in \mathbb{N}$.
		      \end{enumerate}
		\item\label{alg:double:ii} Upon termination of the outer solver loop, define the index $\underline{k}[\ell] \coloneqq k \in \mathbb{N}$.
		\item \textbf{Mark:} Determine a set $\mathcal{M}_\ell \in \mathbb{M}_\ell[\theta,
				      u_\ell^{\underline{k},\underline{j}}]$ satisfying~\eqref{eq:doerfler} with $u_\ell^\star$ replaced
		      by $u_\ell^{\underline{k},\underline{j}}$.
		\item\label{alg:double:iv} \textbf{Refine:} Generate $\mathcal{T}_{\ell+1} \coloneqq \mathtt{refine}(\mathcal{T}_\ell,\mathcal{M}_\ell)$ and define $u_{\ell+1}^{0,0} \coloneqq u_{\ell+1}^{0,\underline{j}} \coloneqq u_\ell^{\underline{k},\underline{j}}$.
	\end{enumerate}
\end{algorithm}
Extending the index notation from Section~\ref{section:single}, we define the
triple index set
\begin{equation*}
	\mathcal{Q} \coloneqq \{(\ell, k, j) \in \mathbb{N}_0^3 \,:\, u_\ell^{k, j} \text{ is used in
		Algorithm~\ref{algorithm:double}}\}
\end{equation*}
together with the lexicographic
ordering
\begin{equation*}
	(\ell', k', j') \le (\ell, k, j)
	\quad :\Longleftrightarrow \quad
	u_{\ell'}^{k', j'} \text{ is defined not later than $u_\ell^{k,j}$ in
		Algorithm~\ref{algorithm:double}}
\end{equation*}
and the total step counter
\begin{equation}\label{eq:stepcounter}
	\vert \ell, k, j \vert \coloneqq \# \{(\ell', k', j') \in \mathcal{Q} \,:\,
	(\ell', k', j') \le (\ell, k, j)\}
	\in \mathbb{N}_0 \quad \text{for } (\ell, k, j) \in \mathcal{Q}.
\end{equation}
Moreover, we define the stopping indices
\begin{subequations}
	\begin{align}
		\underline{\ell}      & \coloneqq \sup\{\ell \in \mathbb{N}_0 \,:\, (\ell,0,0) \in \mathcal{Q}\} \in \mathbb{N}_0 \cup \{\infty\},
		\\
		\underline{k}[\ell]   & \coloneqq \sup\{k \in \mathbb{N}_0 \,:\, (\ell,k,0) \in \mathcal{Q}\} \in \mathbb{N} \cup \{\infty\},
		\quad \text{whenever } (\ell,0,0) \in \mathcal{Q},
		\\
		\underline{j}[\ell,k] & \coloneqq \sup\{j \in \mathbb{N}_0 \,:\, (\ell,k,j) \in \mathcal{Q}\} \in \mathbb{N} \cup \{\infty\},
		\quad \text{whenever } (\ell,k,0) \in \mathcal{Q}.
	\end{align}
\end{subequations}
First, these definitions are consistent with
those of
Algorithm~\ref{algorithm:double}\ref{alg:double:b} and
Algorithm~\ref{algorithm:double}\ref{alg:double:ii}.
Second, there holds indeed
\(\underline{j}[\ell, k] < \infty\) for all \((\ell, k, 0) \in \mathcal{Q}\); see~\cite[Lemma~3.2]{aisfem}.
Third,
\(\underline{\ell} < \infty\) yields \(\underline{k}[\underline{\ell}] = \infty\) and
\(\eta_{\underline{\ell}}(u_{\underline{\ell}}^\star) = 0\) with \(u_{\underline{\ell}}^\star = u^\star\);
see~\cite[Lemma~5.2]{aisfem}.

The following theorem improves~\cite[Theorem~4.1]{aisfem} in the following sense. First, we prove R-linear convergence for all $\ell \ge \ell_0 = 0$,
while $\ell_0 \in \mathbb{N}$ is unknown in practice and depends on \(u^\star\) and the non-accessible sequence \((u_\ell^\star)_{\ell \in \mathbb{N}_0}\)
in~\cite{aisfem}. Second,
\cite{aisfem} requires severe restrictions on $\lambda_{\textup{alg}}$ beyond~\eqref{eq:double:assumption:lambda} below.
We note that~\eqref{eq:double:assumption:lambda} is indeed satisfied, if the
algebraic system is solved exactly, i.e., $\lambda_{\textup{alg}} = 0$, so that
Theorem~\ref{theorem:double:convergence} is a consistent
generalization of Theorem~\ref{theorem:single:convergence}.

\begin{theorem}[full R-linear convergence of Algorithm~\ref{algorithm:double}]\label{theorem:double:convergence}
	Let $0 < \theta \le 1$, $C_{\textup{mark}} \ge 1$, $\lambda_{\textup{sym}}, \lambda_{\textup{alg}} > 0$, and $u_0^{0,0} \in \mathcal{X}_0$.
	With $q_\theta \coloneqq [1 - (1-q_{\textup{red}}^2)\theta]^{1/2}$, suppose that
	\begin{equation}\label{eq:double:assumption:lambda}
		0
		< \frac{q_{\textup{sym}}^\star + \frac{2 \, q_{\textup{alg}}}{1-q_{\textup{alg}}} \, \lambda_{\textup{alg}}}{1 - \frac{2 \, q_{\textup{alg}}}{1-q_{\textup{alg}}} \, \lambda_{\textup{alg}}}
		\eqqcolon q_{\textup{sym}} < 1
		\quad \text{and} \quad
		\lambda_{\textup{alg}} \lambda_{\textup{sym}} < \frac{(1-q_{\textup{alg}})(1-q_{\textup{sym}}^\star)(1-q_\theta)}{8 \, q_{\textup{alg}} C_{\textup{stab}}}.
	\end{equation}
	Then, Algorithm~\ref{algorithm:double} guarantees R-linear convergence of the quasi-error
	\begin{equation}\label{eq:double:quasi-error}
		\mathrm{H}_\ell^{k,j}
		\coloneqq
		|\mkern-1.5mu|\mkern-1.5mu| u_\ell^\star - u_\ell^{k,j} |\mkern-1.5mu|\mkern-1.5mu|
		+ |\mkern-1.5mu|\mkern-1.5mu| u_\ell^{k,\star} - u_\ell^{k,j} |\mkern-1.5mu|\mkern-1.5mu|
		+ \eta_\ell(u_\ell^{k,j}),
	\end{equation}
	i.e., there exist constants $0 < q_{\textup{lin}} < 1$ and $C_{\textup{lin}} > 0$ such that
	\begin{equation}\label{eq:double:convergence}
		\mathrm{H}_\ell^{k,j}
		\le
		C_{\textup{lin}} q_{\textup{lin}}^{|\ell,k,j| - |\ell'\!,k'\!,j'|} \, \mathrm{H}_{\ell'}^{k',j'}
		\text{ for all } (\ell'\!,k'\!,j'),(\ell,k,j) \in \mathcal{Q}
		\text{ with } |\ell'\!,k'\!,j'| \le |\ell,k,j|.
	\end{equation}
\end{theorem}

As proven for Corollary~\ref{corollary:rates:complexity} in Section~\ref{section:single}, an immediate consequence of full linear
convergence (and the geometric series) is that convergence rates with respect
to the number of degrees of freedom and with respect to the overall computational cost
coincide.

\begin{corollary}[rates = complexity]
	For $s > 0$, full R-linear convergence~\eqref{eq:double:convergence} yields
	\begin{equation*}
		\medmuskip = -4mu
		M(s)
		\hspace{-0.1cm} \coloneqq
		\hspace{-0.2cm} \sup_{(\ell,k,j) \in \mathcal{Q}} (\#\mathcal{T}_\ell)^s \, \mathrm{H}_\ell^{k,j}
		\le \hspace{-0.2cm}
		\sup_{(\ell,k,j) \in \mathcal{Q}} \Bigl(\hspace{-0.2cm}
		\sum_{\substack{(\ell',k',j') \in \mathcal{Q} \\
				|\ell',k',j'| \le |\ell,k,j|}}
		\sum_{\substack{(\ell'',k'',j'') \in \mathcal{Q} \\ |\ell'',k'',j''| \le |\ell',k',j'|}} \hspace{-0.5cm}\#\mathcal{T}_{\ell''}\Bigr)^s \mathrm{H}_\ell^{k,j}
		\le C_{\textup{cost}}(s) \, M(s),
	\end{equation*}
	where the constant $C_{\textup{cost}}(s) > 0$ depends only on $C_{\textup{lin}}$, $q_{\textup{lin}}$, and $s$.
	Moreover, there exists $s_0 > 0$ such that $M(s) < \infty$ for all $0 < s \le s_0$.\qed
\end{corollary}

The proof of Theorem~\ref{theorem:double:convergence} requires the following
lemma (which is essentially taken from~\cite{aisfem}).
It deduces the contraction of the
inexact Zarantonello iteration with computed iterates $u_\ell^{k,\underline{j}} \approx
	u_\ell^{k, \star}$
from the exact Zarantonello iteration. For the inexact iteration, the linear
SPD system~\eqref{eq:double:zarantonello} is solved with the contractive algebraic solver~\eqref{eq:double:contractive-solver}, i.e.,
$u_\ell^{k,\star} \coloneqq \Phi_\ell(\delta; u_\ell^{k-1,\underline{j}})$
and $u_\ell^{k,j} \coloneqq \Psi_\ell(u_\ell^{k,\star}, u_\ell^{k,j-1})$ guarantee
\begin{equation}\label{eq:zarantonello}
	|\mkern-1.5mu|\mkern-1.5mu| u_\ell^\star - u_\ell^{k,\star} |\mkern-1.5mu|\mkern-1.5mu|
	\le q_{\textup{sym}}^\star \, |\mkern-1.5mu|\mkern-1.5mu| u_\ell^\star - u_\ell^{k-1,\underline{j}} |\mkern-1.5mu|\mkern-1.5mu|
	\quad \text{for all } (\ell,k,j) \in \mathcal{Q} \text{ with } k \ge 1.
\end{equation}
We emphasize that
contraction is only guaranteed for $0 < k < \underline{k}[\ell]$
in~\eqref{eq1:zarantonello:inexact} below, while the final iteration $k =
	\underline{k}[\ell]$ leads to a perturbed contraction~\eqref{eq2:zarantonello:inexact}
thus requiring additional treatment in the later analysis.
The proof of Lemma~\ref{lem:inexact_contraction} is given in~\ref{appendix}.

\begin{lemma}[contraction of inexact Zarantonello iteration]
	\label{lem:inexact_contraction}
	Under the assumptions of Theorem~\ref{theorem:double:convergence}, the inexact
	Zarantonello iteration used in Algorithm~\ref{algorithm:double}
	satisfies
	\begin{equation}\label{eq1:zarantonello:inexact}
		|\mkern-1.5mu|\mkern-1.5mu| u_\ell^\star - u_\ell^{k,\underline{j}} |\mkern-1.5mu|\mkern-1.5mu|
		\le q_{\textup{sym}} \, |\mkern-1.5mu|\mkern-1.5mu| u_\ell^\star - u_\ell^{k-1,\underline{j}} |\mkern-1.5mu|\mkern-1.5mu|
		\quad \text{for all } (\ell,k,\underline{j}) \in \mathcal{Q} \text{ with } 1 \le k < \underline{k}[\ell]
	\end{equation}
	as well as
	\begin{equation}\label{eq2:zarantonello:inexact}
		|\mkern-1.5mu|\mkern-1.5mu| u_\ell^\star - u_\ell^{\underline{k},\underline{j}} |\mkern-1.5mu|\mkern-1.5mu|
		\le
		q_{\textup{sym}}^\star \, |\mkern-1.5mu|\mkern-1.5mu| u_\ell^\star - u_\ell^{\underline{k}-1,\underline{j}} |\mkern-1.5mu|\mkern-1.5mu|
		+ \frac{2 \, q_{\textup{alg}}}{1-q_{\textup{alg}}} \, \lambda_{\textup{alg}} \lambda_{\textup{sym}} \, \eta_\ell(u_\ell^{\underline{k},\underline{j}})
		\quad \text{for all } (\ell,\underline{k},\underline{j}) \in \mathcal{Q}.
	\end{equation}
\end{lemma}

\begin{proof}[\textbf{Proof of Theorem~\ref{theorem:double:convergence}}]
	The building blocks of the proof are the following:
	First, we show that a suitably weighted quasi-error involving the final iterates of the inexact Zarantonello iteration is tail-summable in the mesh-level index \(\ell\). Second, we show that the quasi-errors are tail-summable in the Zarantonello index \(k\) and, third, in the algebraic-solver index \(j\) and that they are stable in the nested iteration. Finally, combining these ideas leads to tail-summability with respect to the total step counter.
	The proof is split into~six steps. The first four steps follow the proof of Theorem~\ref{theorem:single:convergence} using
	\begin{equation}\label{eq1:double:proof}
		\mathrm{H}_\ell^k
		\coloneqq
		|\mkern-1.5mu|\mkern-1.5mu| u_\ell^\star - u_\ell^{k,\underline{j}} |\mkern-1.5mu|\mkern-1.5mu|
		+ \eta_\ell(u_\ell^{k,\underline{j}})
		\quad \text{for all } (\ell,k,\underline{j}) \in \mathcal{Q}.
	\end{equation}
	By contraction of the algebraic solver
	\eqref{eq:double:contractive-solver} as well as the stopping criteria for the
	algebraic
	solver~\eqref{eq:double:termination:alg} and for the symmetrization
	\eqref{eq:double:termination:sym},
	it holds that
	\begin{equation*}
		|\mkern-1.5mu|\mkern-1.5mu| u_\ell^{\underline{k},\star} - u_\ell^{\underline{k},\underline{j}} |\mkern-1.5mu|\mkern-1.5mu|
		\stackrel{\eqref{eq:double:contractive-solver}}
		\lesssim
		|\mkern-1.5mu|\mkern-1.5mu| u_\ell^{\underline{k},\underline{j}} - u_\ell^{\underline{k},\underline{j}-1} |\mkern-1.5mu|\mkern-1.5mu|
		\stackrel{\eqref{eq:double:termination:alg}}
		\lesssim
		\eta_\ell(u_\ell^{\underline{k},\underline{j}}) + |\mkern-1.5mu|\mkern-1.5mu| u_\ell^{\underline{k},\underline{j}} - u_\ell^{\underline{k}-1,\underline{j}} |\mkern-1.5mu|\mkern-1.5mu|
		\stackrel{\eqref{eq:double:termination:sym}}
		\lesssim
		\eta_\ell(u_\ell^{\underline{k},\underline{j}})
		\le
		\mathrm{H}_{\ell}^{\underline{k}}.
	\end{equation*}
	In particular, this proves equivalence
	\begin{equation}\label{eq2:double:proof}
		\mathrm{H}_\ell^{\underline{k}}
		\le
		\mathrm{H}_\ell^{\underline{k}} + |\mkern-1.5mu|\mkern-1.5mu| u_\ell^{\underline{k},\star} - u_\ell^{\underline{k},\underline{j}} |\mkern-1.5mu|\mkern-1.5mu| =
		\mathrm{H}_\ell^{\underline{k},\underline{j}} \lesssim \mathrm{H}_{\ell}^{\underline{k}}
		\quad \text{for all } (\ell,\underline{k},\underline{j}) \in \mathcal{Q}.
	\end{equation}

	\textbf{Step~1 (auxiliary estimates \& estimator reduction).}
	For $(\ell, \underline{k}, \underline{j}) \in \mathcal{Q}$, nested iteration $u_\ell^{\underline{k},0} = u_\ell^{\underline{k}-1,\underline{j}}$ and $\underline{j}[\ell,\underline{k}]
		\ge 1$ yield
	\begin{equation}\label{eq:step1:1*}
		|\mkern-1.5mu|\mkern-1.5mu| u_\ell^{\underline{k},\star} - u_\ell^{\underline{k},\underline{j}} |\mkern-1.5mu|\mkern-1.5mu|
		\stackrel{\eqref{eq:double:contractive-solver}}\le
		q_{\textup{alg}}^{\underline{j}[\ell,\underline{k}]} \, |\mkern-1.5mu|\mkern-1.5mu| u_\ell^{\underline{k},\star} - u_\ell^{\underline{k},0} |\mkern-1.5mu|\mkern-1.5mu|
		\le
		q_{\textup{alg}} \, |\mkern-1.5mu|\mkern-1.5mu| u_\ell^{\underline{k},\star} - u_\ell^{\underline{k}-1,\underline{j}} |\mkern-1.5mu|\mkern-1.5mu|.
	\end{equation}
	From this, we obtain that
	\begin{align}\label{eq:step1:2*}
		\begin{split}
			|\mkern-1.5mu|\mkern-1.5mu| u_\ell^\star - u_\ell^{\underline{k},\underline{j}} |\mkern-1.5mu|\mkern-1.5mu|
			 & \le
			|\mkern-1.5mu|\mkern-1.5mu| u_\ell^\star - u_\ell^{\underline{k},\star} |\mkern-1.5mu|\mkern-1.5mu| + |\mkern-1.5mu|\mkern-1.5mu| u_\ell^{\underline{k},\star} - u_\ell^{\underline{k},\underline{j}} |\mkern-1.5mu|\mkern-1.5mu|
			\\&
			\stackrel{\mathclap{\eqref{eq:step1:1*}}}\le
			(1+q_{\textup{alg}}) \, |\mkern-1.5mu|\mkern-1.5mu| u_\ell^\star - u_\ell^{\underline{k},\star} |\mkern-1.5mu|\mkern-1.5mu| + q_{\textup{alg}} \,
			|\mkern-1.5mu|\mkern-1.5mu| u_\ell^\star - u_\ell^{\underline{k}-1,\underline{j}} |\mkern-1.5mu|\mkern-1.5mu|
			\\&
			\stackrel{\mathclap{\eqref{eq:zarantonello}}}\le
			\bigl[ (1+q_{\textup{alg}}) q_{\textup{sym}}^\star + q_{\textup{alg}} \bigr] \, |\mkern-1.5mu|\mkern-1.5mu| u_\ell^\star -
			u_\ell^{\underline{k}-1,\underline{j}} |\mkern-1.5mu|\mkern-1.5mu|
			\le
			3 \, |\mkern-1.5mu|\mkern-1.5mu| u_\ell^\star - u_\ell^{\underline{k}-1,\underline{j}} |\mkern-1.5mu|\mkern-1.5mu|.
		\end{split}
	\end{align}
	For $(\ell+1, \underline{k}, \underline{j}) \in \mathcal{Q}$, contraction of the inexact
	Zarantonello iteration~\eqref{eq1:zarantonello:inexact}, nested iteration $u_{\ell+1}^{0,\underline{j}} =
		u_\ell^{\underline{k},\underline{j}}$, and $\underline{k}[\ell+1] \ge 1$,  show that
	\begin{equation}\label{eq:step1:3*}
		|\mkern-1.5mu|\mkern-1.5mu| u_{\ell+1}^\star - u_{\ell+1}^{\underline{k}-1,\underline{j}} |\mkern-1.5mu|\mkern-1.5mu|
		\stackrel{\mathclap{\eqref{eq1:zarantonello:inexact}}}\le
		q_{\textup{sym}}^{\underline{k}[\ell+1]-1} \, |\mkern-1.5mu|\mkern-1.5mu| u_{\ell+1}^\star - u_{\ell+1}^{0,\underline{j}} |\mkern-1.5mu|\mkern-1.5mu| \le
		|\mkern-1.5mu|\mkern-1.5mu| u_{\ell+1}^\star - u_\ell^{\underline{k},\underline{j}} |\mkern-1.5mu|\mkern-1.5mu|.
	\end{equation}
	The combination of the previous two displayed formulas shows
	\begin{equation}\label{eq:step1:1}
		|\mkern-1.5mu|\mkern-1.5mu| u_{\ell+1}^\star - u_{\ell+1}^{\underline{k},\underline{j}} |\mkern-1.5mu|\mkern-1.5mu|
		\stackrel{\mathclap{\eqref{eq:step1:2*}}}\le
		3 \, |\mkern-1.5mu|\mkern-1.5mu| u_{\ell+1}^\star - u_{\ell+1}^{\underline{k}-1,\underline{j}} |\mkern-1.5mu|\mkern-1.5mu|
		\stackrel{\mathclap{\eqref{eq:step1:3*}}}\le
		3 \, |\mkern-1.5mu|\mkern-1.5mu| u_{\ell+1}^\star - u_\ell^{\underline{k},\underline{j}} |\mkern-1.5mu|\mkern-1.5mu|.
	\end{equation}
	Analogous arguments to~\eqref{eq:single:estimator-reduction} in the proof of Theorem~\ref{theorem:exact:convergence} establish
	\begin{equation}\label{eq:step2:1}
		\eta_{\ell+1}(u_{\ell+1}^{\underline{k},\underline{j}})
		\stackrel{\eqref{eq:single:estimator-reduction}}\le
		q_\theta \, \eta_\ell(u_\ell^{\underline{k},\underline{j}}) + C_{\textup{stab}} \, |\mkern-1.5mu|\mkern-1.5mu| u_{\ell+1}^{\underline{k},\underline{j}} - u_\ell^{\underline{k},\underline{j}} |\mkern-1.5mu|\mkern-1.5mu|
		\stackrel{\eqref{eq:step1:1}}\le
		q_\theta \, \eta_\ell(u_\ell^{\underline{k}, \underline{j}}) + 4 C_{\textup{stab}} \,
		|\mkern-1.5mu|\mkern-1.5mu| u_{\ell+1}^\star -
		u_\ell^{\underline{k},\underline{j}} |\mkern-1.5mu|\mkern-1.5mu|.
	\end{equation}

	\textbf{Step~2 (tail summability with respect to $\boldsymbol{\ell}$).}
	With $\lambda \coloneqq \lambda_{\textup{alg}}\lambda_{\textup{sym}}$, we define
	\begin{equation*}
		\gamma \coloneqq \frac{q_\theta(1-q_{\textup{sym}}^\star)}{4 \, C_{\textup{stab}}},
		\quad
		C(\gamma,\lambda) \coloneqq 1 + \frac{2 \, q_{\textup{alg}}}{1 - q_{\textup{alg}}} \, \frac{\lambda}{\gamma},
		\quad \text{and} \quad
		\alpha \coloneqq \frac{\lambda}{\gamma} \stackrel{\eqref{eq:double:assumption:lambda}} < \frac{(1-q_{\textup{alg}})(1-q_\theta)}{2 \, q_{\textup{alg}} q_\theta}.
	\end{equation*}
	By definition, it follows that
	\begin{equation*}
		C(\gamma,\lambda)
		=
		1 + \frac{2 \, q_{\textup{alg}}}{1 - q_{\textup{alg}}} \, \alpha
		<
		1 + \frac{1-q_\theta}{q_\theta} = 1 / q_\theta.
	\end{equation*}
	This ensures that
	\begin{equation}\label{eq:double:constants}
		q_\theta \, C(\gamma,\lambda) < 1
		\quad \text{as well as}  \quad
		q_{\textup{sym}}^\star + 4 \, C_{\textup{stab}} C(\gamma,\lambda) \, \gamma
		<
		q_{\textup{sym}}^\star + \frac{4 \, C_{\textup{stab}}}{q_\theta} \, \gamma = 1.
	\end{equation}
	With contraction of the inexact Zarantonello iteration~\eqref{eq2:zarantonello:inexact}, Step~1 proves
	\begin{align}\label{eq:step3:1}
		\begin{split}
			 & |\mkern-1.5mu|\mkern-1.5mu| u_{\ell+1}^\star - u_{\ell+1}^{\underline{k},\underline{j}} |\mkern-1.5mu|\mkern-1.5mu| + \gamma \,
			\eta_{\ell+1}(u_{\ell+1}^{\underline{k},\underline{j}})
			\stackrel{\eqref{eq2:zarantonello:inexact}} \le
			q_{\textup{sym}}^\star \, |\mkern-1.5mu|\mkern-1.5mu| u_{\ell+1}^\star - u_{\ell+1}^{\underline{k}-1,\underline{j}} |\mkern-1.5mu|\mkern-1.5mu|
			+ C(\gamma,\lambda) \, \gamma \, \eta_{\ell+1}(u_{\ell+1}^{\underline{k},\underline{j}})
			\\& \quad
			\stackrel{\mathclap{\eqref{eq:step1:3*}}}\le
			q_{\textup{sym}}^\star \, |\mkern-1.5mu|\mkern-1.5mu| u_{\ell+1}^\star - u_\ell^{\underline{k},\underline{j}} |\mkern-1.5mu|\mkern-1.5mu|
			+C(\gamma,\lambda) \, \gamma \, \eta_{\ell+1}(u_{\ell+1}^{\underline{k},\underline{j}})
			\\& \quad
			\stackrel{\mathclap{\eqref{eq:step2:1}}}\le
			\bigl( q_{\textup{sym}}^\star + 4 \, C_{\textup{stab}} \, C(\gamma,\lambda) \, \gamma \bigr) \,
			|\mkern-1.5mu|\mkern-1.5mu| u_{\ell+1}^\star - u_\ell^{\underline{k},\underline{j}} |\mkern-1.5mu|\mkern-1.5mu|
			+ q_\theta \, C(\gamma,\lambda) \, \gamma \, \eta_\ell(u_\ell^{\underline{k},\underline{j}})
			\\& \quad
			\le
			q_{\textup{ctr}} \,
			\bigl[ |\mkern-1.5mu|\mkern-1.5mu| u_{\ell+1}^\star - u_\ell^{\underline{k},\underline{j}} |\mkern-1.5mu|\mkern-1.5mu| + \gamma \,
				\eta_\ell(u_\ell^{\underline{k},\underline{j}}) \bigr]
			\quad \text{for all } (\ell+1,\underline{k},\underline{j}) \in \mathcal{Q},
		\end{split}
	\end{align}
	where~\eqref{eq:double:constants} ensures the bound
	\begin{equation*}
		0 < q_{\textup{ctr}}
		\coloneqq
		\max \bigl\{ q_{\textup{sym}}^\star + 4 \, C_{\textup{stab}} \, C(\gamma,\lambda) \, \gamma \,,\, q_\theta \, C(\gamma,\lambda) \bigr\} < 1.
	\end{equation*}
	Altogether, we obtain
	\begin{align*}
		a_{\ell+1}
		\coloneqq
		|\mkern-1.5mu|\mkern-1.5mu| u_{\ell+1}^\star - u_{\ell+1}^{\underline{k},\underline{j}} |\mkern-1.5mu|\mkern-1.5mu| + \gamma \,
		\eta_{\ell+1}(u_{\ell+1}^{\underline{k},\underline{j}})
		 & \, \stackrel{\mathclap{\eqref{eq:step3:1}}}\le \,
		q_{\textup{ctr}} \,
		\bigl[ |\mkern-1.5mu|\mkern-1.5mu| u_\ell^\star - u_\ell^{\underline{k},\underline{j}} |\mkern-1.5mu|\mkern-1.5mu| + \gamma \,
			\eta_\ell(u_\ell^{\underline{k},\underline{j}}) \bigr]
		+ q_{\textup{ctr}} \, |\mkern-1.5mu|\mkern-1.5mu| u_{\ell+1}^\star - u_\ell^\star |\mkern-1.5mu|\mkern-1.5mu|
		\\&
		\, \eqqcolon \,
		q_{\textup{ctr}} \, a_\ell + b_\ell
		\quad \text{for all } (\ell,\underline{k},\underline{j}) \in \mathcal{Q},
	\end{align*}
	which corresponds to~\eqref{eq:a_ell:perturbed-contraction} in the case of a
	single contractive solver (with $u_\ell^{\underline{k},\underline{j}}$ replacing $u_\ell^{\underline{k}}$ in~\eqref{eq:a_ell:perturbed-contraction}).
	Together with~\eqref{eq:cea+reliability}--\eqref{eq:single:orthogonality} (with $u_\ell^{\underline{k},\underline{j}}$ replacing $u_\ell^{\underline{k}}$), the
	assumptions~\eqref{eq:summability:criterion} of
	Lemma~\ref{lemma:summability:criterion} are satisfied. Therefore,
	Lemma~\ref{lemma:summability:criterion} proves tail summability
	\begin{align*}
		 & \sum_{\ell' = \ell+1}^{\underline{\ell}-1} \mathrm{H}_{\ell'}^{\underline{k}}
		\stackrel{\eqref{eq1:double:proof}}\simeq
		\sum_{\ell' = \ell+1}^{\underline{\ell}-1} \bigl[ |\mkern-1.5mu|\mkern-1.5mu| u_{\ell'}^\star -
			u_{\ell'}^{\underline{k},\underline{j}} |\mkern-1.5mu|\mkern-1.5mu| + \gamma \, \eta_{\ell'}(u_{\ell'}^{\underline{k},\underline{j}}) \bigr]
		\\& \qquad\qquad\qquad
		\lesssim \,
		|\mkern-1.5mu|\mkern-1.5mu| u_\ell^\star - u_\ell^{\underline{k},\underline{j}} |\mkern-1.5mu|\mkern-1.5mu| + \gamma \, \eta_\ell(u_\ell^{\underline{k},\underline{j}})
		\stackrel{\eqref{eq1:double:proof}}\simeq
		\mathrm{H}_\ell^{\underline{k}}
		\quad \text{for all }  (\ell,\underline{k},\underline{j}) \in \mathcal{Q}.
	\end{align*}

	\textbf{Step~3 (auxiliary estimates).}
	First, we employ~\eqref{eq:step1:2*} to deduce
	\begin{align}\label{eq1:step7}
		\begin{split}
			\mathrm{H}_\ell^{\underline{k}}
			\  & \stackrel{\mathclap{\eqref{axiom:stability}}}\lesssim \
			|\mkern-1.5mu|\mkern-1.5mu| u_\ell^\star - u_\ell^{\underline{k},\underline{j}} |\mkern-1.5mu|\mkern-1.5mu|
			+ |\mkern-1.5mu|\mkern-1.5mu| u_\ell^{\underline{k},\underline{j}} - u_\ell^{\underline{k}-1,\underline{j}} |\mkern-1.5mu|\mkern-1.5mu|
			+ \eta_\ell(u_\ell^{\underline{k}-1,\underline{j}})
			\stackrel{\eqref{eq1:double:proof}}\le
			\mathrm{H}_\ell^{\underline{k}-1} + 2 \, |\mkern-1.5mu|\mkern-1.5mu| u_\ell^{\underline{k},\underline{j}} - u_\ell^{\underline{k}-1,\underline{j}} |\mkern-1.5mu|\mkern-1.5mu|
			\\&
			\stackrel{\mathclap{\eqref{eq:step1:2*}}}\le \
			\mathrm{H}_\ell^{\underline{k}-1} + 8 \, |\mkern-1.5mu|\mkern-1.5mu| u_\ell^\star - u_\ell^{\underline{k}-1,\underline{j}} |\mkern-1.5mu|\mkern-1.5mu|
			\le 9 \, \mathrm{H}_\ell^{\underline{k}-1}
			\quad \text{for all } (\ell,\underline{k},\underline{j}) \in \mathcal{Q}.
		\end{split}
	\end{align}
	Second, for $0 \le k < k' < \underline{k}[\ell]$, the failure of the stopping criterion for the inexact Zarantonello symmetrization~\eqref{eq:double:termination:sym} and contraction~\eqref{eq1:zarantonello:inexact} prove that
	\begin{equation}\label{eq1a:step7}
		\mathrm{H}_\ell^{k'}
		\stackrel{\eqref{eq:double:termination:sym}}\lesssim
		|\mkern-1.5mu|\mkern-1.5mu| u_\ell^\star - u_\ell^{k',\underline{j}} |\mkern-1.5mu|\mkern-1.5mu| +
		|\mkern-1.5mu|\mkern-1.5mu| u_\ell^{k',\underline{j}} - u_\ell^{k'-1,\underline{j}} |\mkern-1.5mu|\mkern-1.5mu|
		\stackrel{\eqref{eq1:zarantonello:inexact}}\lesssim
		|\mkern-1.5mu|\mkern-1.5mu| u_\ell^\star - u_\ell^{k'-1,\underline{j}} |\mkern-1.5mu|\mkern-1.5mu|
		\stackrel{\eqref{eq1:zarantonello:inexact}}\lesssim
		q_{\textup{sym}}^{k'-k} \, |\mkern-1.5mu|\mkern-1.5mu| u_\ell^\star - u_\ell^{k,\underline{j}} |\mkern-1.5mu|\mkern-1.5mu|.
	\end{equation}
	Moreover, for $k < k' = \underline{k}[\ell]$, we
	combine~\eqref{eq1:step7} with~\eqref{eq1a:step7} to get
	\begin{equation}\label{eq1b:step7}
		\mathrm{H}_\ell^{\underline{k}}
		\stackrel{\eqref{eq1:step7}}\lesssim \mathrm{H}_\ell^{\underline{k}[\ell]-1}
		\stackrel{\eqref{eq1a:step7}}\lesssim q_{\textup{sym}}^{(\underline{k}[\ell]-1)-k} \, |\mkern-1.5mu|\mkern-1.5mu| u_\ell^\star - u_\ell^{k,\underline{j}} |\mkern-1.5mu|\mkern-1.5mu|
		\simeq q_{\textup{sym}}^{\underline{k}[\ell]-k} \,
		|\mkern-1.5mu|\mkern-1.5mu| u_\ell^\star - u_\ell^{k,\underline{j}} |\mkern-1.5mu|\mkern-1.5mu|.
	\end{equation}
	The combination of~\eqref{eq1a:step7}--\eqref{eq1b:step7} proves
	that
	\begin{equation}\label{eq2:step7}
		\boxed{
			\mathrm{H}_\ell^{k'}
			\lesssim q_{\textup{sym}}^{\,k' - k} \,|\mkern-1.5mu|\mkern-1.5mu| u_\ell^\star - u_\ell^{k,\underline{j}} |\mkern-1.5mu|\mkern-1.5mu|
			\lesssim q_{\textup{sym}}^{\,k' - k} \, \mathrm{H}_\ell^{k}
			\quad \text{for all } (\ell,0, 0) \in \mathcal{Q} \text{ with } 0 \le k \le
			k' \le \underline{k}[\ell],
		}
	\end{equation}
	where the hidden constant depends only on $C_{\textup{stab}}$, $\lambda_{\textup{sym}}$, and $q_{\textup{sym}}$.
	Third, we recall
	\begin{equation*}
		|\mkern-1.5mu|\mkern-1.5mu| u_\ell^\star - u_{\ell-1}^\star |\mkern-1.5mu|\mkern-1.5mu|
		\stackrel{\eqref{eq:quasimonotonicity_b}}
		\lesssim
		\eta_{\ell-1}(u_{\ell-1}^\star)
		\stackrel{\eqref{axiom:stability}}
		\lesssim
		\eta_{\ell-1}(u_{\ell-1}^{\underline{k},\underline{j}}) + |\mkern-1.5mu|\mkern-1.5mu| u_{\ell-1}^\star -
		u_{\ell-1}^{\underline{k},\underline{j}} |\mkern-1.5mu|\mkern-1.5mu|
		=
		\mathrm{H}_{\ell-1}^{\underline{k}}.
	\end{equation*}
	Together with nested iteration $u_{\ell-1}^{\underline{k},\underline{j}} = u_\ell^{0,\underline{j}}$, this yields that
	\begin{equation}\label{eq3:step7}
		\boxed{
		\mathrm{H}_\ell^0
		= |\mkern-1.5mu|\mkern-1.5mu| u_\ell^\star - u_{\ell-1}^{\underline{k},\underline{j}} |\mkern-1.5mu|\mkern-1.5mu| + \eta_\ell(u_{\ell-1}^{\underline{k},\underline{j}})
		\le |\mkern-1.5mu|\mkern-1.5mu| u_\ell^\star - u_{\ell-1}^\star |\mkern-1.5mu|\mkern-1.5mu| + \mathrm{H}_{\ell-1}^{\underline{k}}
		\lesssim
		\mathrm{H}_{\ell-1}^{\underline{k}}
		\quad \text{for all } (\ell,0,0) \in \mathcal{Q}.
		}
	\end{equation}

	\textbf{Step~4 (tail summability with respect to $\boldsymbol{\ell}$ and $\boldsymbol{k}$).}
	The auxiliary estimates from Step~3 and the geometric series prove that
	\begin{align}\label{eq:step8}
		\begin{split}
			 & \sum_{\substack{(\ell',k',\underline{j}) \in \mathcal{Q} \\ |\ell',k',\underline{j}| > |\ell,k,\underline{j}|}} \mathrm{H}_{\ell'}^{k'}
			=
			\sum_{k' = k+1}^{\underline{k}[\ell]} \mathrm{H}_\ell^{k'}
			+ \sum_{\ell' = \ell+1}^{\underline{\ell}} \sum_{k'=0}^{\underline{k}[\ell]} \mathrm{H}_{\ell'}^{k'}
			\stackrel{\eqref{eq2:step7}}\lesssim
			\mathrm{H}_\ell^k
			+ \sum_{\ell' = \ell+1}^{\underline{\ell}} \mathrm{H}_{\ell'}^0
			\\& \qquad
			\stackrel{\eqref{eq3:step7}}\lesssim
			\mathrm{H}_\ell^k
			+ \sum_{\ell' = \ell}^{\underline{\ell}-1} \mathrm{H}_{\ell'}^{\underline{k}}
			\lesssim
			\mathrm{H}_\ell^k
			+ \mathrm{H}_\ell^{\underline{k}}
			\stackrel{\eqref{eq2:step7}}\lesssim
			\mathrm{H}_\ell^k
			\quad \text{for all } (\ell,k,\underline{j}) \in \mathcal{Q}.
		\end{split}
	\end{align}

	\textbf{Step~5 (auxiliary estimates).}
	Recall $\mathrm{H}_\ell^{\underline{k}} \le \mathrm{H}_\ell^{\underline{k},\underline{j}}$ from~\eqref{eq2:double:proof}.
	For $j=0$ and $k=0$, the definition $u_\ell^{0,0} \coloneqq u_\ell^{0,\underline{j}} \coloneqq u_\ell^{0,\star}$ leads to $\mathrm{H}_\ell^{0,0} = \mathrm{H}_\ell^0$. For $k \ge 1$, nested iteration $u_\ell^{k,0} = u_\ell^{k-1,\underline{j}}$ and contraction of the Zarantonello iteration~\eqref{eq:zarantonello} imply
	\begin{equation*}
		|\mkern-1.5mu|\mkern-1.5mu| u_\ell^{k,\star} - u_\ell^{k,0} |\mkern-1.5mu|\mkern-1.5mu|
		\le
		|\mkern-1.5mu|\mkern-1.5mu| u_\ell^\star - u_\ell^{k,\star} |\mkern-1.5mu|\mkern-1.5mu|
		+ |\mkern-1.5mu|\mkern-1.5mu| u_\ell^\star - u_\ell^{k-1,\underline{j}} |\mkern-1.5mu|\mkern-1.5mu|
		\stackrel{\eqref{eq:zarantonello}}\le
		(q_{\textup{sym}}^\star + 1) \, |\mkern-1.5mu|\mkern-1.5mu| u_\ell^\star - u_\ell^{k-1,\underline{j}} |\mkern-1.5mu|\mkern-1.5mu|
		\le 2 \, \mathrm{H}_\ell^{k-1}.
	\end{equation*}
	Therefore, we derive that
	\begin{equation}\label{eq1:step9}
		\boxed{
			\mathrm{H}_\ell^{k,0} \le 3 \, \mathrm{H}_\ell^{(k-1)_+}
			\quad \text{for all } (\ell,k,0) \in \mathcal{Q},
			\quad \text{where } (k-1)_+ \coloneqq \max\{0, k-1 \}.
		}
	\end{equation}
	For any $0 \le j < j' < \underline{j}[\ell,k]$,
	the contraction of the Zarantonello
	iteration~\eqref{eq:zarantonello}, the contraction of the algebraic solver~\eqref{eq:double:contractive-solver}, and the failure of the stopping criterion for the algebraic solver~\eqref{eq:double:termination:alg} prove
	\begin{align*}
		\begin{split}
			\mathrm{H}_\ell^{k, j'}\
			 & \le \
			|\mkern-1.5mu|\mkern-1.5mu| u_\ell^\star - u_\ell^{k,\star} |\mkern-1.5mu|\mkern-1.5mu|
			+ 2 \, |\mkern-1.5mu|\mkern-1.5mu|u_\ell^{k,\star} - u_\ell^{k, j'} |\mkern-1.5mu|\mkern-1.5mu|
			+ \eta_\ell(u_\ell^{k,j'})
			\\&
			\stackrel{\mathclap{\eqref{eq:zarantonello}}}\lesssim \
			|\mkern-1.5mu|\mkern-1.5mu| u_\ell^{k,j'} - u_\ell^{k-1, \underline{j}} |\mkern-1.5mu|\mkern-1.5mu|
			+
			|\mkern-1.5mu|\mkern-1.5mu| u_\ell^{k,\star} - u_\ell^{k,j'} |\mkern-1.5mu|\mkern-1.5mu|
			+ \eta_\ell(u_\ell^{k,j'})
			\\&
			\stackrel{\mathclap{\eqref{eq:double:contractive-solver}}}\lesssim \
			|\mkern-1.5mu|\mkern-1.5mu| u_\ell^{k,j'} - u_\ell^{k-1, \underline{j}} |\mkern-1.5mu|\mkern-1.5mu| +
			|\mkern-1.5mu|\mkern-1.5mu| u_\ell^{k,j'} - u_\ell^{k, j'-1} |\mkern-1.5mu|\mkern-1.5mu| +
			\eta_\ell(u_\ell^{k,j'})
			\\&
			\stackrel{\mathclap{\eqref{eq:double:termination:alg}}}\lesssim \
			|\mkern-1.5mu|\mkern-1.5mu| u_\ell^{k,j'} - u_\ell^{k, j'-1} |\mkern-1.5mu|\mkern-1.5mu|
			\stackrel{\eqref{eq:double:contractive-solver}}\lesssim
			|\mkern-1.5mu|\mkern-1.5mu| u_\ell^{k,\star} - u_\ell^{k, j'-1} |\mkern-1.5mu|\mkern-1.5mu|
			\stackrel{\eqref{eq:double:contractive-solver}} \lesssim \
			q_{\textup{alg}}^{j'-j} \, |\mkern-1.5mu|\mkern-1.5mu| u_\ell^{k,\star} - u_\ell^{k, j} |\mkern-1.5mu|\mkern-1.5mu|
			\lesssim
			q_{\textup{alg}}^{j'-j} \, \mathrm{H}_\ell^{k,j}.
		\end{split}
	\end{align*}
	For $j' = \underline{j}[\ell,k]$, it follows that
	\begin{equation*}
		\mathrm{H}_\ell^{k,\underline{j}}
		\stackrel{\eqref{axiom:stability}}\lesssim
		\mathrm{H}_\ell^{k,\underline{j}-1} +
		|\mkern-1.5mu|\mkern-1.5mu| u_\ell^{k,\underline{j}} - u_\ell^{k,\underline{j}-1} |\mkern-1.5mu|\mkern-1.5mu|
		\stackrel{\mathclap{\eqref{eq:double:contractive-solver}}}\lesssim
		\mathrm{H}_\ell^{k,\underline{j}-1} +
		|\mkern-1.5mu|\mkern-1.5mu| u_\ell^{k,\star} - u_\ell^{k,\underline{j}-1} |\mkern-1.5mu|\mkern-1.5mu|
		\stackrel{\mathclap{\eqref{eq:double:quasi-error}}}\le
		2 \, \mathrm{H}_\ell^{k,\underline{j}-1}
		\lesssim
		q_{\textup{alg}}^{\underline{j}[\ell, k]-j} \, \mathrm{H}_\ell^{k,j}.
	\end{equation*}
	The combination of the previous two displayed formulas results in
	\begin{equation}\label{eq3:step10}
		\boxed{
			\mathrm{H}_\ell^{k, j'}
			\lesssim
			q_{\textup{alg}}^{j'- j} \, \mathrm{H}_\ell^{k,j}
			\quad \text{for all } (\ell, k, 0) \in \mathcal{Q} \quad \text{with} \quad
			0 \le
			j \le j' \le \underline{j}[\ell,k],
		}
	\end{equation}
	where the hidden constant depends only on $q_{\textup{sym}}^\star$, $\lambda_{\textup{sym}}$, $q_{\textup{alg}}$, $\lambda_{\textup{alg}}$, and $C_{\textup{stab}}$.

	\medskip
	\textbf{Step~6 (tail summability with respect to $\boldsymbol{\ell}$, $\boldsymbol{k}$, and $\boldsymbol{j}$).}
	Finally, we observe that
	\begin{align*}
		 & \sum_{\substack{(\ell',k',j') \in \mathcal{Q}           \\ |\ell',k',j'| > |\ell,k,j|}}  \mathrm{H}_{\ell'}^{k',j'}
		=
		\sum_{j'=j+1}^{\underline{j}[\ell,k]} \mathrm{H}_{\ell}^{k,j'}
		+ \sum_{k'=k+1}^{\underline{k}[\ell]} \sum_{j'=0}^{\underline{j}[\ell,k']} \mathrm{H}_\ell^{k',j'}
		+ \sum_{\ell' = \ell+1}^{\underline{\ell}} \sum_{k'=0}^{\underline{k}[\ell']} \sum_{j'=0}^{\underline{j}[\ell',k']} \mathrm{H}_{\ell'}^{k',j'}
		\\& \qquad
		\stackrel{\eqref{eq3:step10}}\lesssim
		\mathrm{H}_{\ell}^{k,j}
		+ \sum_{k'=k+1}^{\underline{k}[\ell]} \mathrm{H}_\ell^{k',0}
		+ \sum_{\ell' = \ell+1}^{\underline{\ell}} \sum_{k'=0}^{\underline{k}[\ell]} \mathrm{H}_{\ell'}^{k',0}
		\stackrel{\eqref{eq1:step9}}\lesssim
		\mathrm{H}_{\ell}^{k,j}
		+ \sum_{\substack{(\ell',k',\underline{j}) \in \mathcal{Q} \\ |\ell',k',\underline{j}| > |\ell,k,\underline{j}|}} \mathrm{H}_{\ell'}^{k'}
		\\& \qquad
		\stackrel{\eqref{eq:step8}}\lesssim
		\mathrm{H}_\ell^{k,j}
		+ \mathrm{H}_\ell^{k}
		\stackrel{\eqref{eq2:double:proof}}\lesssim
		\mathrm{H}_\ell^{k,j}
		+ \mathrm{H}_\ell^{k,\underline{j}}
		\stackrel{\eqref{eq3:step10}}\lesssim
		\mathrm{H}_\ell^{k,j}
		\quad \text{for all } (\ell,k,j) \in \mathcal{Q}.
	\end{align*}
	Since $\mathcal{Q}$ is countable and linearly ordered, Lemma~\ref{lemma:summability} concludes the proof of~\eqref{eq:double:convergence}.
\end{proof}

The final theorem, following from~\cite[Theorem~4.3]{bhimps2023b}, states that for sufficiently small adaptivity parameters \(\theta\),
\(\lambda_{\textup{sym}}\), and \(\lambda_{\textup{alg}}\), Algorithm~\ref{algorithm:double} achieves optimal complexity.

\begin{theorem}[Optimal complexity of Algorithm~\ref{algorithm:double}, {\cite[Theorem~4.3]{bhimps2023b}}]\label{theorem:double:complexity}
	Suppose that the estimator satisfies the axioms of adaptivity~\eqref{axiom:stability}--\eqref{axiom:discrete_reliability}
	and suppose that quasi-orthogonality~\eqref{axiom:orthogonality} holds. Assume
	full R-linear convergence from Theorem~\ref{theorem:double:convergence}.
	Define the constants \(\theta^\star\), \(\lambda_{\textup{sym}}^\star\) by
	\begin{equation*}
		\begin{aligned}
			\theta^\star                 & \coloneqq \bigl( 1 + C_{\textup{stab}}^2 \, C_{\textup{drel}}^2 \bigr)^{-1},
			\\
			\lambda_{\textup{sym}}^\star & \coloneqq \min\{ 1, C_{\textup{stab}}^{-1} \, C_{\textup{alg}}^{-1}\}
			\quad \text{with} \quad
			C_{\textup{alg}} \coloneqq \frac{1}{1-q_{\textup{sym}}^\star} \Bigl( \frac{2 \, q_{\textup{alg}}}{1 - q_{\textup{alg}}} \, \lambda_{\textup{alg}}^\star + q_{\textup{sym}}^\star \Bigr).
		\end{aligned}
	\end{equation*}
	Suppose that the constants \(\theta\), \(\lambda_{\textup{sym}}\), and \(\lambda_{\textup{alg}}\) are
	sufficiently small in the sense that, additionally to~\eqref{eq:double:assumption:lambda},
	there holds
	\begin{equation*}
		0 < \lambda_{\textup{sym}} < \lambda_{\textup{sym}}^\star
		\quad \text{and} \quad
		0 < \frac{\bigl(\theta^{1/2} + \lambda_{\textup{sym}} / \lambda_{\textup{sym}}^\star\bigr)^2}{\bigl(1 - \lambda_{\textup{sym}} / \lambda_{\textup{sym}}^\star\bigr)^2}
		< \theta^\star < 1.
	\end{equation*}
	Then, Algorithm~\ref{algorithm:double} guarantees for all \( s > 0 \)
	\begin{equation*}
		c_{\textup{opt}} \| u^\star \|_{\mathbb{A}_s}
		\le
		\sup_{(\ell,k,j) \in \mathcal{Q}}
		\Bigl(
		\sum \limits_{\substack{(\ell',k',j') \in \mathcal{Q} \\ |\ell',k',j'| \le |\ell,k,j|}}
		\# \mathcal{T}_{\ell'}
		\Bigr)^{s}
		\, \mathrm{H}_\ell^{k,j}
		\le
		C_{\textup{opt}} \max \{ \| u^\star \|_{\mathbb{A}_s}, \mathrm{H}_0^{0,0} \}.
	\end{equation*}
	The constant \(c_{\textup{opt}} > 0\) depends only on~\(C_{\textup{stab}}\), the use of NVB refinement, and \(s\), while \( C_{\textup{opt}} > 0 \) depends only on~\(C_{\textup{stab}}\), \(q_{\textup{red}}\), \(C_{\textup{drel}}\), \(C_{\textup{lin}}\), \(q_{\textup{lin}}\), \(\# \mathcal{T}_0\), \(\lambda_{\textup{sym}}\), \(q_{\textup{sym}}^\star\), \(\lambda_{\textup{alg}}\), \(q_{\textup{alg}}\), \(\theta\), \(s\), and the use of NVB refinement. \qed
\end{theorem}

%%%%%%%%%%%%%%%%%%%%%%%%%%%%%%%%%%%%%%%%%%%%%%%%%%%%%%%%%%%%%%%%%%%%%%%
%%%%%%%%%%%%%%%%%%%%%%%%%%%%%%%%%%%%%%%%%%%%%%%%%%%%%%%%%%%%%%%%%%%%%%%
\section{Application to strongly monotone nonlinear PDEs}
\label{section:conclusion}
%%%%%%%%%%%%%%%%%%%%%%%%%%%%%%%%%%%%%%%%%%%%%%%%%%%%%%%%%%%%%%%%%%%%%%%
%%%%%%%%%%%%%%%%%%%%%%%%%%%%%%%%%%%%%%%%%%%%%%%%%%%%%%%%%%%%%%%%%%%%%%%

In the previous sections, the particular focus was on general second-order linear elliptic PDEs~\eqref{eq:strongform}. However, the results also apply to nonlinear PDEs with strongly monotone and Lipschitz-continuous nonlinearity as considered, e.g., in \cite{gmz2011, gmz2012, cw2017, ghps2018, hw2020a, hw2020b, ghps2021, hpsv2021, hpw2021, hw2022, hmrv2023, mv2023} to mention only some recent works.

Given a nonlinearity $\boldsymbol{A} \colon \mathbb{R}^d \to \mathbb{R}^d$, we consider the nonlinear elliptic PDE
\begin{equation}\label{eq:nonlinear:strongform}
	-\operatorname{div}\big(\boldsymbol{A}(\nabla u^\star)\big) = f - \operatorname{div} \boldsymbol{f} \text{
		in } \Omega
	\quad \text{subject to}\quad
	u^\star = 0 \text{ on } \partial\Omega.
\end{equation}
We define the associated nonlinear operator $\mathcal{A} \colon
	H^1_0(\Omega) \to H^{-1}(\Omega) \coloneqq H^1_0(\Omega)^*$ via $\mathcal{A} u \coloneqq
	\langle \boldsymbol{A}(\nabla u), \nabla(\cdot) \rangle_{L^2(\Omega)}$, where we suppose
that the $L^2(\Omega)$ scalar product on the right-hand side is well-defined.
Then, the weak formulation of~\eqref{eq:nonlinear:strongform} reads
\begin{equation}\label{eq:nonlinear:weakform}
	\langle \mathcal{A} u^\star, v \rangle = F(v) \coloneqq \langle f, v \rangle_{L^2(\Omega)} +
	\langle \boldsymbol{f}, \nabla v \rangle_{L^2(\Omega)}
	\quad \text{for all } v \in H^1_0(\Omega),
\end{equation}
where $\langle \cdot, \cdot \rangle$ on the left-hand side denotes the duality brackets
on $H^{-1}(\Omega) \times H^1_0(\Omega)$.

Let $a(\cdot,\cdot)$ be an equivalent scalar product on $H^1_0(\Omega)$ with
induced norm $|\mkern-1.5mu|\mkern-1.5mu| \, \cdot \, |\mkern-1.5mu|\mkern-1.5mu|$. Suppose that $\mathcal{A}$ is strongly monotone and
Lipschitz continuous, i.e., there exist $0 < \alpha \le L$ such that, for all
$u, v, w \in H^1_0(\Omega)$,
\begin{equation}\label{eq:nonlinear:assumptions}
	\alpha \, |\mkern-1.5mu|\mkern-1.5mu| u-v |\mkern-1.5mu|\mkern-1.5mu|^2 \le \langle \mathcal{A} u - \mathcal{A} v, u-v \rangle
	\quad \text{and} \quad
	\langle \mathcal{A} u - \mathcal{A} v, w \rangle \le L \, |\mkern-1.5mu|\mkern-1.5mu| u-v |\mkern-1.5mu|\mkern-1.5mu|
	\, |\mkern-1.5mu|\mkern-1.5mu| w |\mkern-1.5mu|\mkern-1.5mu|.
\end{equation}
Under these assumptions, the Zarantonello theorem from~\cite{Zarantonello1960} (or
main theorem on strongly monotone operators~\cite[Section 25.4]{Zeidler1990})
yields existence and uniqueness of the solution $u^\star \in H^1_0(\Omega)$
to~\eqref{eq:nonlinear:weakform}.
For $\mathcal{T}_H \in \mathbb{T}$ and $\mathcal{X}_H \subseteq H^1_0(\Omega)$ from~\eqref{eq:discrete_space},
it also applies to the discrete setting and yields existence and
uniqueness of the discrete solution $u_H^\star \in \mathcal{X}_H$ to
\begin{equation}\label{eq:nonlinear:discrete}
	\langle \mathcal{A} u_H^\star, v_H \rangle = F(v_H)
	\quad \text{for all } v_H \in \mathcal{X}_H,
\end{equation}
which is quasi-optimal in the sense of the C\'ea lemma~\eqref{eq:cea}.

As already discussed in Section~\ref{section:double}, the proof of the
Zarantonello theorem relies on the Banach fixed-point theorem:
For $0 < \delta < 2\alpha/L^2$,
define $\Phi_H(\delta; \cdot) \colon \mathcal{X}_H \to
	\mathcal{X}_H$ via
\begin{equation}\label{eq:nonlinear:zarantonello}
	a(\Phi_H(\delta; u_H), v_H)
	=
	a(u_H, v_H) + \delta \big[ F(v_H) - \langle \mathcal{A}(u_H), v_H \rangle \big]
	\quad \text{for all } u_H, v_H \in \mathcal{X}_H.
\end{equation}
Since $a(\cdot, \cdot)$ is a scalar product, $\Phi_H(\delta; u_H) \in \mathcal{X}_H$ is well-defined. Moreover, for $0 < \delta < 2\alpha/L^2$ and $0 < q_{\textup{sym}}^\star \coloneqq [1 - \delta(2\alpha-\delta L^2)]^{1/2} < 1$, this mapping is a contraction, i.e.,
\begin{equation*}
	|\mkern-1.5mu|\mkern-1.5mu| u_H^\star - \Phi_H(\delta; u_H) |\mkern-1.5mu|\mkern-1.5mu|
	\le q_{\textup{sym}}^\star \, |\mkern-1.5mu|\mkern-1.5mu| u_H^\star - u_H |\mkern-1.5mu|\mkern-1.5mu|
	\quad \text{for all } u_H \in \mathcal{X}_H;
\end{equation*}
see also~\cite{hw2020a,hw2020b}. Analogously to Section~\ref{section:double}, the variational formulation~\eqref{eq:nonlinear:zarantonello} leads to a linear SPD system for
which we employ a uniformly contractive
solver~\eqref{eq:double:contractive-solver}.
Overall, we note that for the nonlinear PDE~\eqref{eq:nonlinear:strongform}, the natural AFEM loop consists of
\begin{itemize}
	\item discretization via a conforming triangulation~\(\mathcal{T}_{\ell}\) (leading to the non-computable solution \(u_\ell^\star\) to the discrete nonlinear system~\eqref{eq:nonlinear:discrete}),
	\item iterative linearization (giving rise to the solution \(u_\ell^{k, \star} = \Phi_\ell(\delta; u_\ell^{k-1, \underline{j}})\) of the large-scale discrete SPD system~\eqref{eq:nonlinear:zarantonello} obtained by linearizing~\eqref{eq:nonlinear:discrete} in \(u_\ell^{k-1, \underline{j}}\)),
	\item and an algebraic solver (leading to computable approximations \(u_\ell^{k, j} \approx u_\ell^{k, \star}\)).
\end{itemize}
Thus, the natural AFEM algorithm takes the form of
Algorithm~\ref{algorithm:double} in Section~\ref{section:double}.

So far, the only work analyzing convergence of such a full adaptive loop for the numerical solution of~\eqref{eq:nonlinear:strongform} is
\cite{hpsv2021}, which uses the Zarantonello approach~\eqref{eq:nonlinear:zarantonello} for linearization and a preconditioned CG method with optimal additive Schwarz preconditioner for solving the arising SPD systems. Importantly and contrary to the present work, the
adaptivity parameters \(\theta, \lambda_{\textup{sym}}\), and \(\lambda_{\textup{alg}}\) in~\cite{hpsv2021} must be sufficiently
small to guarantee full linear convergence and optimal complexity,
while even plain convergence for arbitrary \(\theta\), \(\lambda_{\textup{sym}}\), and \(\lambda_{\textup{alg}}\) is left open. Instead, the present work proves full R-linear convergence at least for arbitrary $\theta$ and $\lambda_{\textup{sym}}$ and the milder constraint~\eqref{eq:double:assumption:lambda} on $\lambda_{\textup{alg}}$.

To apply the analysis from Section~\ref{section:double}, it only remains to
check the validity of Proposition~\ref{prop:axioms} and
Proposition~\ref{prop:orthogonality}. The following result provides the
analogue of Proposition~\ref{prop:axioms} for scalar nonlinearities. Note that, first, the same assumptions are made in~\cite{hpsv2021} and,
second, only the proof of stability~\eqref{axiom:stability} (going back
to~\cite{gmz2012}) is restricted to scalar nonlinearities and lowest-order
discretizations, i.e., $p = 1$ in~\eqref{eq:discrete_space}, while reduction~\eqref{axiom:reduction}, reliability~\eqref{axiom:reliability},
	and discrete reliability~\eqref{axiom:discrete_reliability} follow as for linear PDEs and thus hold for all \(p \ge 1\).

\begin{proposition}[{see, e.g.,~\cite[Section~3.2]{gmz2012} or~\cite[Section~10.1]{cfpp2014}}]\label{prop:nonlinear:axioms}
	Suppose that $\boldsymbol{A}(\nabla u) = a(|\nabla u|^2) \nabla u$, where $a \in C^1(\mathbb{R}_{\ge0})$ satisfies
	\begin{equation*}
		\alpha (t-s) \le a(t^2) t - a(s^2) s \le \frac{L}{3} \, (t-s)
		\quad \text{for all } t \ge s \ge 0.
	\end{equation*}
	Then, there holds~\eqref{eq:nonlinear:assumptions} for $|\mkern-1.5mu|\mkern-1.5mu| v |\mkern-1.5mu|\mkern-1.5mu|
		\coloneqq \| \nabla v \|_{L^2(\Omega)}$ and the standard residual error estimator~\eqref{eq:definition_eta} for lowest-order elements $p = 1$ (with $\boldsymbol{A}\nabla v_H$ understood as $\boldsymbol{A}(\nabla v_H)$ and $\boldsymbol{b} = 0 = c$) satisfies stability~\eqref{axiom:stability}, reduction~\eqref{axiom:reduction}, reliability~\eqref{axiom:reliability}, discrete reliability~\eqref{axiom:discrete_reliability}, and quasi-monotonicity~\eqref{eq:quasi-monotonicity} from Proposition~\ref{prop:axioms}.
	\qed
\end{proposition}

Under the same assumptions as in Proposition~\ref{prop:nonlinear:axioms}, quasi-orthogonality~\eqref{axiom:orthogonality} is satisfied. For the convenience of the reader, we include a sketch of the proof, which also shows that the quasi-orthogonality holds for any \(p \ge 1\) with \(C_{\textup{orth}} = L/\alpha\) and \(\delta = 1\).

\begin{proposition}\label{prop:nonlinear:orthogonality}
	Under the assumptions of Proposition~\ref{prop:nonlinear:axioms} and for any
	sequence of nested finite-dimensional subspaces $\mathcal{X}_\ell \subseteq \mathcal{X}_{\ell+1} \subset
		H^1_0(\Omega)$, the corresponding Galerkin solutions $u_\ell^\star \in
		\mathcal{X}_\ell$ to~\eqref{eq:nonlinear:discrete} satisfy
	quasi-orthogonality~\eqref{axiom:orthogonality} with $\delta = 1$ and $C_{\textup{orth}} =
		L/\alpha$, i.e.,
	\begin{equation*}
		\sum_{\ell' = \ell}^\infty \, |\mkern-1.5mu|\mkern-1.5mu| u_{\ell'+1}^\star - u_{\ell'}^\star |\mkern-1.5mu|\mkern-1.5mu|^2
		\le \frac{L}{\alpha} \, |\mkern-1.5mu|\mkern-1.5mu| u^\star - u_\ell |\mkern-1.5mu|\mkern-1.5mu|^2
		\quad \text{for all } \ell \in \mathbb{N}_0.
	\end{equation*}
\end{proposition}

\begin{proof}[Sketch of proof]
	One can prove that the energy
	\begin{equation*}
		E(v) \coloneqq \frac{1}{2}\int_\Omega \int_0^{|\nabla v(x)|^2} a(t) \, \textup{d}t
		\, \textup{d}x - F(v)
		\quad
		\text{for all $v \in H^1_0(\Omega)$}
	\end{equation*}
	is G\^ateaux-differentiable with $\textup{d} E(v) = \mathcal{A} v - F$. Then,
	elementary calculus (see, e.g., \cite[Lemma~5.1]{ghps2018}
	or~\cite[Lemma~2]{hw2020b}) yields the equivalence
	\begin{equation}\label{eq:nonlinear:equivalence}
		\frac{\alpha}{2} \, |\mkern-1.5mu|\mkern-1.5mu| u_H^\star \!-\! v_H |\mkern-1.5mu|\mkern-1.5mu|^2
		\le E(v_H) \!-\! E(u_H^\star)
		\le \frac{L}{2} \, |\mkern-1.5mu|\mkern-1.5mu|u_H^\star \!-\! v_H |\mkern-1.5mu|\mkern-1.5mu|^2
		\,\, \text{for all } \mathcal{T}_H \in \mathbb{T} \text{ and all } v_H \in \mathcal{X}_H.
	\end{equation}
	In particular, we see that $u_H^\star$ is the unique minimizer to
	\begin{equation}\label{eq:nonlinear:minimization}
		E(u_H^\star) = \min_{v_H \in \mathcal{X}_H} E(v_H),
	\end{equation}
	and~\eqref{eq:nonlinear:equivalence}--\eqref{eq:nonlinear:minimization} also
	hold for $u^\star$ and $H^1_0(\Omega)$ replacing $u_H^\star$ and
	$\mathcal{X}_H$, respectively.

	From this and the telescopic sum, we infer that
	\begin{align*}
		 & \frac{\alpha}{2} \sum_{\ell' = \ell}^{\ell+N} \, |\mkern-1.5mu|\mkern-1.5mu|u_{\ell'+1}^\star -
		u_{\ell'}^\star |\mkern-1.5mu|\mkern-1.5mu|^2
		\stackrel{\eqref{eq:nonlinear:equivalence}}\le  \sum_{\ell' = \ell}^{\ell+N}  \big[
			E(u_{\ell'}^\star) - E(u_{\ell'+1}^\star) \big]
		= E(u_{\ell}^\star) - E(u_{\ell+N+1}^\star)
		\\& \qquad
		\stackrel{\eqref{eq:nonlinear:minimization}}\le E(u_{\ell}^\star) - E(u^\star)
		\stackrel{\eqref{eq:nonlinear:equivalence}}\le \frac{L}{2} \, |\mkern-1.5mu|\mkern-1.5mu| u^\star -
		u_{\ell}^\star |\mkern-1.5mu|\mkern-1.5mu|^2
		\quad \text{for all } \ell, N \in \mathbb{N}_0.
	\end{align*}
	Since the right-hand side is independent of $N$, we conclude the proof for $N \to \infty$.
\end{proof}

Thus, full R-linear convergence from Theorem~\ref{theorem:double:convergence} and optimal complexity from Theorem~\ref{theorem:double:complexity} apply also to the nonlinear PDE~\eqref{eq:nonlinear:strongform} under the assumptions on the nonlinearity from Proposition~\ref{prop:nonlinear:axioms}. Unlike~\cite{hpsv2021}, we can guarantee full R-linear convergence~\eqref{eq:double:convergence} for arbitrary \(\theta\), arbitrary \(\lambda_{\textup{sym}}\), and a weaker constraint~\eqref{eq:double:assumption:lambda} on \(\lambda_{\textup{alg}}\).
As in
\cite[Theorem~5]{hpsv2021}, optimal complexity follows if the adaptivity parameters are sufficiently small.

\begin{remark}
	The cost-optimal numerical solution of nonlinear PDEs is widely open beyond the case of strongly monotone and Lipschitz continuous nonlinearities considered here. We stress that this problem class even excludes the \(p\)-Laplacian, for which linear convergence in ~\cite{dk2008} and optimal convergence rates in~\cite{bdk2012} are known under the constraint of the exact solution of the arising nonlinear discrete systems. Convergent linearization strategies for the \(p\)-Laplacian are the topic of recent research, e.g.,~\cite{dftw2020, bds2023, heid2023}. However, optimal complexity appears to be still out of reach. Nevertheless, the present work could outline potential strategies also in this respect.
\end{remark}

%%%%%%%%%%%%%%%%%%%%%%%%%%%%%%%%%%%%%%%%%%%%%%%%%%%%%%%%%%%%%%%%%%%%%%%%%%%%%%%
%%%%%%%%%%%%%%%%%%%%%%%%%%%%%%%%%%%%%%%%%%%%%%%%%%%%%%%%%%%%%%%%%%%%%%%%%%%%%%%
\section{Numerical experiments}\label{section:numerics}
%%%%%%%%%%%%%%%%%%%%%%%%%%%%%%%%%%%%%%%%%%%%%%%%%%%%%%%%%%%%%%%%%%%%%%%%%%%%%%%
%%%%%%%%%%%%%%%%%%%%%%%%%%%%%%%%%%%%%%%%%%%%%%%%%%%%%%%%%%%%%%%%%%%%%%%%%%%%%%%
%
The following numerical experiments employ the \textsc{Matlab}
software package MooAFEM from~\cite{MooAFEM}.\footnote{The experiments
	accompanying this paper will be provided under
	\url{https://www.tuwien.at/mg/asc/praetorius/software/mooafem}.}
The first numerical example illustrates the performance of
	Algorithm~\ref{algorithm:single} for a symmetric linear elliptic PDE with a
	strong jump in the diffusion coefficient and compares Algorithm~\ref{algorithm:single}
	to the Algorithm~\ref{algorithm:exact} with exact solution.
	The second numerical example demonstrates the efficiency of
	Algorithm~\ref{algorithm:double} for a nonsymmetric general second-order linear
	elliptic PDE with a
	moderate convection. The numerical behavior of the adaptive algorithm proposed
	in Section~\ref{section:conclusion} for a strongly monotone and Lipschitz continuous
	nonlinear PDE concludes the numerical experiments. All numerical experiments
	showcase the performance of the adaptive algorithms for different selections of the
	involved adaptivity parameters.

%%%%%%%%%%%%%%%%%%%%%%%%%%%%%%%%%%%%%%%%%%%%%%%%%%%%%%%%%%%%%%%%%%%%%%%%%%%%%%%
\subsection{AFEM for a symmetric linear elliptic PDE with strong jump in the diffusion coefficient}
\label{section:numerics:symmetric}
%%%%%%%%%%%%%%%%%%%%%%%%%%%%%%%%%%%%%%%%%%%%%%%%%%%%%%%%%%%%%%%%%%%%%%%%%%%%%%%

	Following~\cite{kellogg1974}, we consider
	the square domain \(\Omega \coloneqq (-1, 1)^2\) and the jumping diffusion coefficient
	\(\boldsymbol{A} (x) \coloneqq
	a(x)  \, I_{2 \times 2} \in L^\infty(\Omega)\)
	for
	\(a(x) \coloneqq 161.4476387975881\) if \(x_1 x_2 > 0\)
	and \(a(x) \coloneqq 1\) if \(x_1 x_2 < 0\). In order to measure the performance
	of Algorithm~\ref{algorithm:single}, we consider the interface problem
	\begin{equation}\label{eq:experiment:symmetric}
		-\operatorname{div}\bigl( \boldsymbol{A} \, \nabla  u^\star \bigr) = 0 \text{ in } \Omega.
	\end{equation}
	The exact weak solution in polar coordinates reads
	\(u^\star(r, \phi) \coloneqq r^\alpha \mu(\phi)\)
	where the constants are set to be \(\alpha = 0.1\),
	\(\beta = -14.92256510455152\), \(\delta = \pi/4\),
	and \(\mu (\phi)\) is defined as
	\[
		\mu(\phi)
		\coloneqq
		\begin{cases}
			\cos((\pi/2 - \beta)\alpha) \,
			\cos((\phi - \pi/2 + \delta) \alpha)
			 & \text{if } 0 \leq \phi < \pi / 2,
			\\
			\cos(\delta \alpha)\, \cos((\phi - \pi + \beta) \alpha)
			 & \text{if } \pi / 2 \leq \phi < \pi,
			\\
			\cos(\beta \alpha) \, \cos((\phi - \pi - \delta) \alpha)
			 & \text{if } \pi \leq \phi < 3 \pi / 2,
			\\
			\cos((\pi/2 - \delta)\alpha) \,
			\cos((\phi - 3\pi/2 - \beta) \alpha)
			 & \text{if } 3 \pi / 2 \leq \phi < 2 \pi.
		\end{cases}
	\]
	The exact solution determines the inhomogeneous Dirichlet boundary conditions
	\(u_{\textup{D}}(x) \coloneqq u^\star(x)\)
	for \(x \in \partial\Omega\). The parameter \(\alpha\) gives the regularity
	of the solution \(u \in H^{1+\alpha-\varepsilon}(\Omega)\)
	for all \(\varepsilon > 0\), having a strong point singularity
	at the origin where the interfaces intersect.
	Figure~\ref{fig:meshes:solutions:symmetric} illustrates the initial triangulation
	\(\mathcal{T}_{0}\), the adaptively generated mesh \(\mathcal{T}_{15}\)
	with \(518\) triangles, and the exact solution \(u^\star\), and
	the computed solutions \(u_{15}^{\underline{k}}\). We see that the adaptive
	algorithm captures the singularity induced by the strong jump in the diffusion
	coefficient and refines around the origin.
	Let \(\Pi_{E}^{p-1}\) be the \(L^2(E)\)-orthogonal projection
	onto the space of polynomials of degree at most \(p-1\) on the boundary face
	\(E \subset \partial \Omega\) and
	\(\partial u_{\textup{D}} / \partial s\) denote the arc-length derivative
	of \(u_{\textup{D}}\).
	We approximate the boundary data
	\(u_{\textup{D}}\) by
	nodal interpolation from~\cite{mns2003, fpp2014}, leading to
	an additional boundary-data oscillation term (for sufficiently smooth
	\(u_{\textup{D}}\), see, e.g., \cite{afkpp2013}),
	in the error estimator
	\begin{align*}
		\eta_H(T,v_H)^2
		 & \coloneqq
		\vert T \vert\,
		\Vert \operatorname{div}(\boldsymbol{A} \nabla v_H) \Vert_{L^2(T)}^2
		+
		\vert T \vert^{1/2}\,
		\Vert [\boldsymbol{A} \nabla v_H \cdot n] \Vert_{L^2(\partial T \setminus \partial \Omega)}^2
		\\
		 & \hphantom{{}\coloneqq{}}
		+
		\sum_{E \subset \partial T \cap \partial \Omega}
		\vert T \vert^{1/2}
		\Vert
		(1 - \Pi_{E}^{p-1}) \, \partial u_{\textup{D}} / \partial s
		\Vert_{L^2(E)}.
	\end{align*}
	Figure~\ref{fig:convergence:symmetric} shows that Algorithm~\ref{algorithm:single}
	leads to optimal convergence rates \(-p/2\) with respect to the number of degrees
	and the overall computation time for arbitrary polynomial degrees
	\(p \in \{1, 2, 3, 4\}\) and a moderate marking parameter \(\theta = 0.5\) and fixed
	algebraic solver parameter
	\(\lambda = 0.01\). Furthermore, the reduced elliptic regularity
	leads to convergence rates \(-1/10\) for uniform mesh refinement and
	any polynomial degree \(p\). In Figure~\ref{fig:lamalg:symmetric}, we observe
	that even moderate values of the algebraic solver parameter \(\lambda\)
	lead to optimal convergence rates \(-1\) for polynomial degree \(p = 2\) with
	respect to the number of degrees and the overall computation time.
	Figure~\ref{fig:solver:symmetric} verifies that
	Algorithm~\ref{algorithm:single} in combination with the optimal \(hp\)-robust
	geometric multigrid solver from~\cite{imps2022} outperforms the \textsc{Matlab}
	built-in \texttt{mldivide} in terms of the cumulative computation time.
	Table~\ref{tab:cost:symmetric} summarizes the optimal selection of the adaptivity
	parameters for the interface problem
	in~\eqref{eq:experiment:symmetric} with polynomial degree \(p = 2\). The best
	performance is observed for the marking parameter \(\theta \in \{0.5, 0.7\}\)
	and the algebraic solver parameter \(\lambda = 0.9\).

\begin{figure}
	\centering
	\subfloat{\includegraphics[width = 0.35\textwidth]{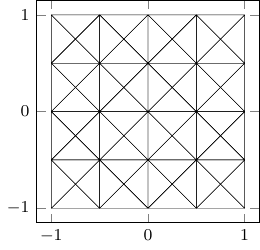}}
	\hfil
	\subfloat{\includegraphics[width = 0.35\textwidth]{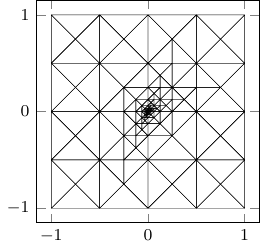}}
	\\
	\subfloat{\includegraphics[width = 0.45\textwidth]{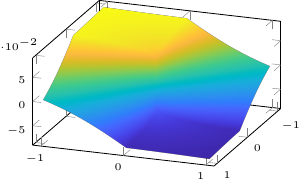}}
	\hfil
	\subfloat{\includegraphics[width = 0.45\textwidth]{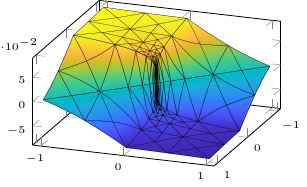}}
	\caption{Illustration of the initial triangulation \(\mathcal{T}_{0}\), the
		adaptively generated mesh \(\mathcal{T}_{15}\) with \(518\) triangles,
		the exact solution \(u^\star\) and the computed solution
		\(u_{15}^{\underline{k}}\) for the Kellogg benchmark
		problem~\eqref{eq:experiment:symmetric} with polynomial degree \(p = 2\),
		marking parameter \(\theta = 0.5\), and algebraic solver parameter \(\lambda = 0.01\).}
	\label{fig:meshes:solutions:symmetric}
\end{figure}

\begin{figure}
	\centering
	\includegraphics[width = 0.45\textwidth]{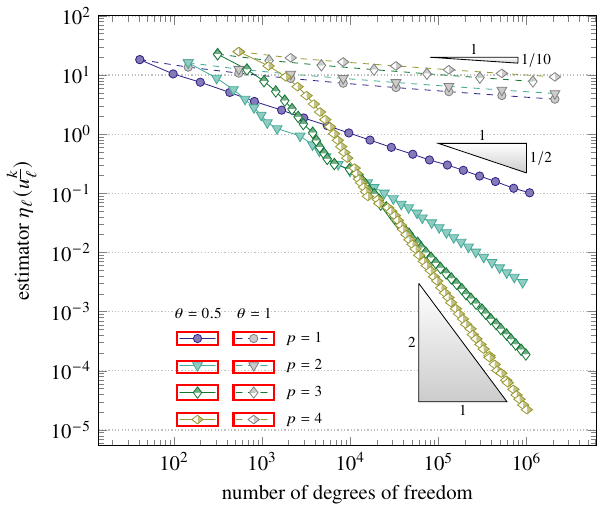}
	\hfil
	\includegraphics[width = 0.45\textwidth]{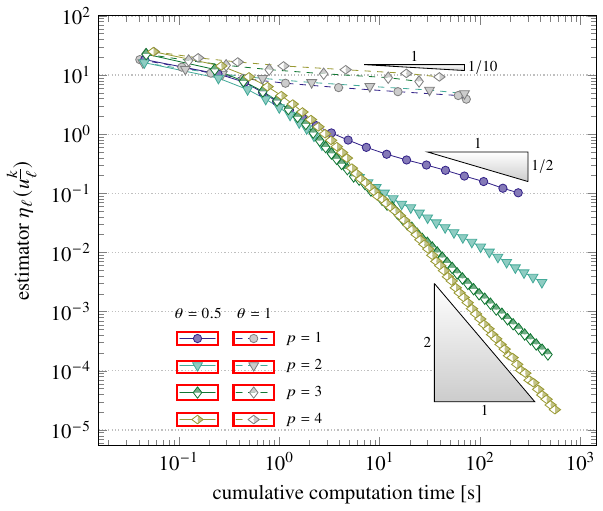}
	\caption{Convergence history plot of the error estimator \(\eta_{\ell}(u_{\ell}^{\underline{k}})\)
		with respect to the number of degrees of freedom (left) and the cumulative computation time (right) for the Kellogg benchmark problem~\eqref{eq:experiment:symmetric} for different polynomial degrees \(p \in \{1, 2, 3, 4\}\)
		with fixed marking parameter \(\theta = 0.5\) and
		algebraic solver parameter \(\lambda = 0.01\).}
	\label{fig:convergence:symmetric}
\end{figure}

\begin{figure}
	\centering
	\includegraphics[width = 0.45\textwidth]{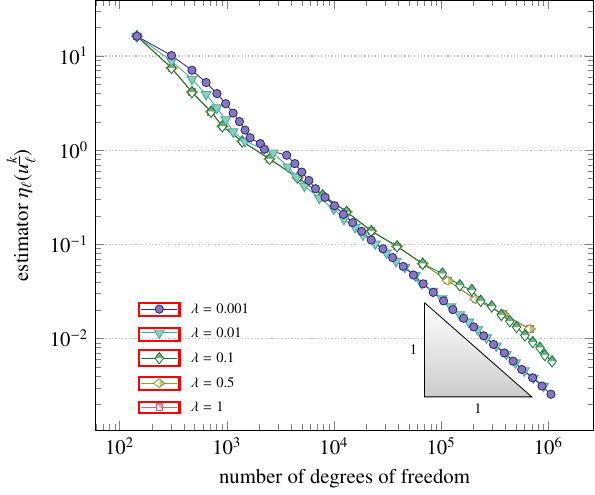}
	\hfil
	\includegraphics[width = 0.45\textwidth]{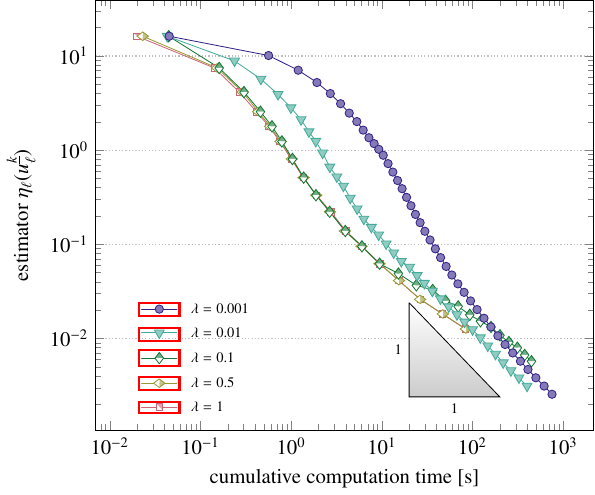}
	\caption{Convergence history plot of the error estimator \(\eta_{\ell}(u_{\ell}^{\underline{k}})\) for different algebraic solver parameters \(\lambda \in \{0.001, 0.01, 0.1, 0.5, 1\}\) and fixed polynomial degree \(p=2\) and marking parameter \(\theta = 0.5\) with respect to the number of degrees of freedom (left) and the cumulative computation time (right) for the Kellogg benchmark problem~\eqref{eq:experiment:symmetric}.}
	\label{fig:lamalg:symmetric}
\end{figure}

\begin{figure}
	\centering
	\includegraphics[width = 0.45\textwidth]{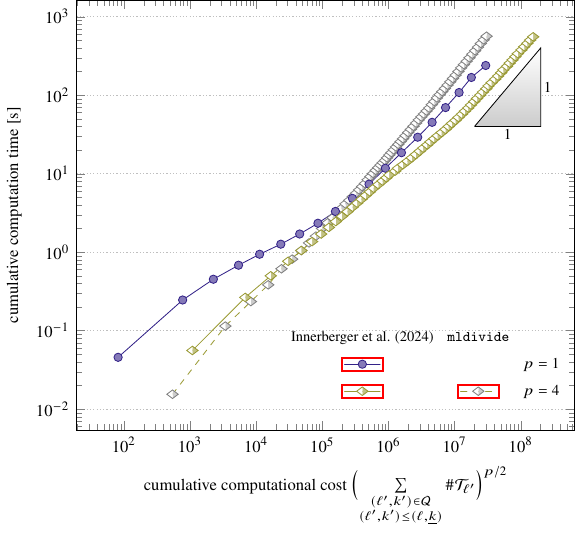}
	\caption{Comparison of the cumulative computation time for the algebraic solver (Algorithm~\ref{algorithm:single})
		from~\cite{imps2022} with the \textsc{Matlab} built-in \texttt{mldivide} (Algorithm~\ref{algorithm:exact})
		over the cumulative number of degrees of freedom to solve the
		Kellogg benchmark problem
		\eqref{eq:experiment:symmetric} with polynomial degree \(p \in \{1, 4\}\),
		marking parameter \(\theta = 0.5\), and algebraic solver parameter \(\lambda = 0.01\).}
	\label{fig:solver:symmetric}
\end{figure}

\begin{table}
	\centering
	\includegraphics{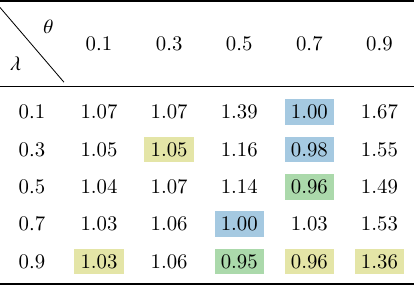}
	\caption{Optimal selection of the parameters for the Kellogg benchmark problem~\eqref{eq:experiment:symmetric} with polynomial degree \(p = 2\). For the comparison, we consider the weighted cumulative time \(\big[ \eta_{\ell}(u_{\ell}^{\underline{k}}) \, \sum_{| \ell', k' | \le | \ell, \underline{k} |} \mathrm{time}(\ell') \big]\) with stopping criterion \(\eta_{\ell}(u_{\ell}^{\underline{k}}) < 10^{-2} \, \eta_0(u_0^0)\) for various choices of marking parameter \(\theta\)
		and algebraic solver parameter \(\lambda\). The best choice per column is marked in yellow, per row in blue, and for both in green. For all fixed marking parameter \(\theta\), the best performance is observed for \(\lambda = 0.9\), while overall best results are achieved for \(\theta \in \{0.5, 0.7\}\).}
	\label{tab:cost:symmetric}
\end{table}

%%%%%%%%%%%%%%%%%%%%%%%%%%%%%%%%%%%%%%%%%%%%%%%%%%%%%%%%%%%%%%%%%%%%%%%%%%%%%%%
\subsection{AFEM for a general second-order linear elliptic PDE}
%%%%%%%%%%%%%%%%%%%%%%%%%%%%%%%%%%%%%%%%%%%%%%%%%%%%%%%%%%%%%%%%%%%%%%%%%%%%%%%
On the L-shaped domain \(\Omega = (-1,1)^2 \setminus [0, 1) \times
[-1, 0)\), we consider
\begin{equation}\label{eq:experiment}
	-\Delta u^\star + \boldsymbol{b} \cdot \nabla u^\star + u^\star = 1
	\text{ in } \Omega \quad \text{and} \quad u^\star = 0 \text{ on } \partial \Omega
	\quad \text{with} \quad \boldsymbol{b}(x) = x;
\end{equation}
see Figure~\ref{fig:meshes} for the geometry and some adaptively
generated meshes.

\begin{figure}[htbp!]
	\resizebox{\textwidth}{!}{
		\includegraphics[width=0.19\textwidth]{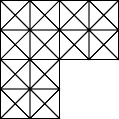}
		\hfil
		\includegraphics[width=0.19\textwidth]{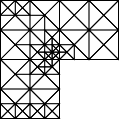}
		\hfil
		\includegraphics[width=0.19\textwidth]{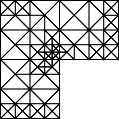}
		\hfil
		\includegraphics[width=0.19\textwidth]{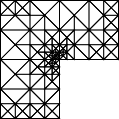}
		\hfil
		\includegraphics[width=0.19\textwidth]{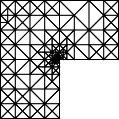}
	}
	\caption{\label{fig:meshes} Illustration of the initial triangulation
		\(\mathcal{T}_{0}\) and the sequence of adaptively generated meshes \(\mathcal{T}_{0},
		\ldots, \mathcal{T}_{4}\) for the experiment~\eqref{eq:experiment}.}
\end{figure}
\emph{Optimality of Algorithm~\ref{algorithm:double} with respect to
	large solver-stopping parameters \(\lambda_{\textup{sym}}\) and \(\lambda_{\textup{alg}}\).}
We choose \(\delta = 0.5\), $\theta = 0.3$, and the polynomial
degree
\(p=2\). Figure~\ref{fig:optimality_lamsym} presents the convergence rates for fixed $\lambda_{\textup{alg}} = 0.7$ and
several symmetrization parameters \(\lambda_{\textup{sym}} \in \{0.1, 0.3, 0.5, 0.7, 0.9\}\).
We observe that Algorithm~\ref{algorithm:double} obtains the optimal
convergence rate \(-1\) with respect to the number of degrees of freedom
and the cumulative computation time for any
selection of \(\lambda_{\textup{sym}}\).
Moreover, the same holds true for fixed $\lambda_{\textup{sym}} = 0.7$ and any choice of the algebraic solver parameter \(\lambda_{\textup{alg}} \in \{0.1, 0.3, 0.5, 0.7, 0.9\}\) as depicted in Figure~\ref{fig:optimality_lamalg}.
Table~\ref{tab:cost} illustrates the weighted cumulative
computation time of
Algorithm~\ref{algorithm:double} and shows that a smaller marking parameter
\(\theta = 0.3\) in combination with larger solver-stopping
parameters \(\lambda_{\textup{sym}}\) and
\(\lambda_{\textup{alg}}\) is favorable. Furthermore, Figure~\ref{fig:optimality_p}
shows that Algorithm~\ref{algorithm:double} guarantees optimal convergence
rates \(-p/2\) for several polynomial degrees \(p\) with fixed
\(\delta = 0.5\), marking parameter \(\theta = 0.3\), and moderate %solver-stopping parameters 
\(\lambda_{\textup{sym}} = \lambda_{\textup{alg}} = 0.7\).

\begin{table}[t]
	\centering
	\includegraphics[width=\textwidth]{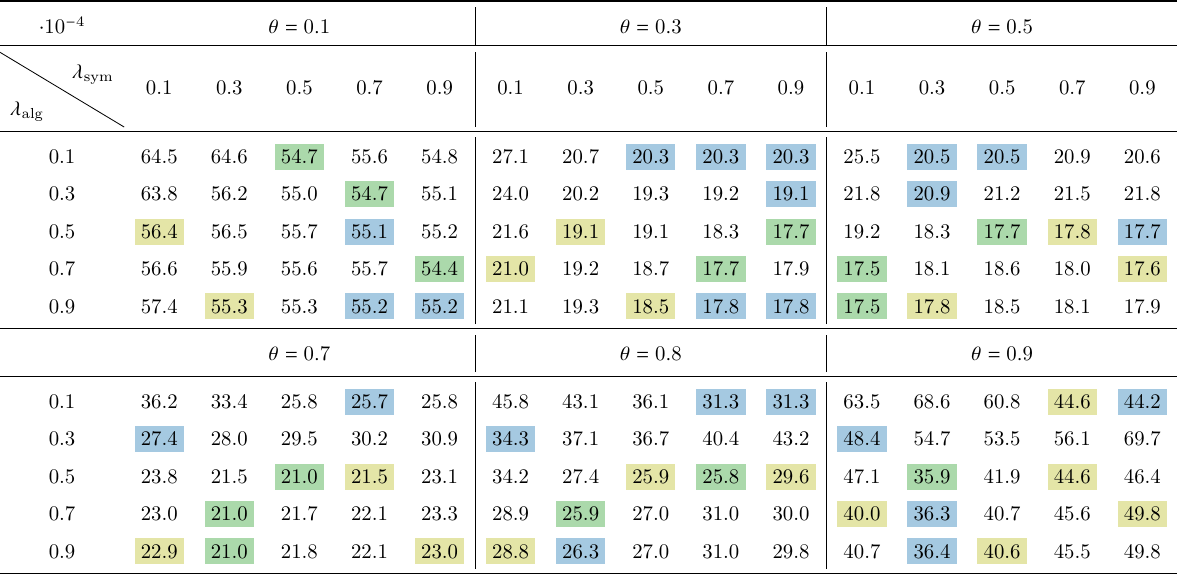}
	\caption{\label{tab:cost} Optimal selection of parameters with
		respect to the computational costs for the nonsymmetric experiment~\eqref{eq:experiment} with $p=2$ and $\delta = 0.5$. For the comparison, we consider
		the
		weighted
		cumulative time $ \big[ \eta_{\ell}(u_\ell^{\underline{k},
					\underline{j}}) \,
				\sum_{|
					\ell', k', j' | \le
					| \ell, \underline{k}, \underline{j} |}
				\mathrm{time}(\ell')
				\big]$ (values in \(10^{-4}\)) with stopping
		criterion $\eta_\ell(u_\ell^{\underline{k}, \underline{j}}) < 5 \cdot
			10^{-5}$ for
		various choices of
		$\lambda_{\textup{sym}}$, $\lambda_{\textup{alg}}$, and $\theta$.
		In each $\theta$-block, we mark in yellow the best choice per column, in blue the best choice per row, and in green when both choices coincide. The best choices for $\lambda_{\textup{alg}}$ and $\lambda_{\textup{sym}}$ are observed for $\theta = 0.3$ and $\theta=0.5$.}
\end{table}
\begin{figure}[t]
	\includegraphics[width=0.45\textwidth]{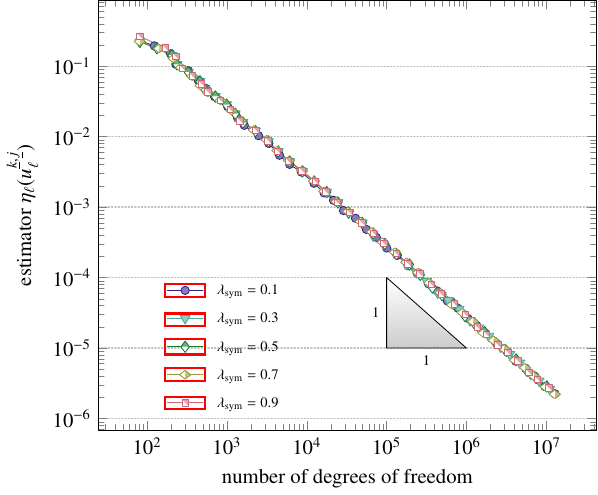}
	\hfil
	\includegraphics[width=0.45\textwidth]{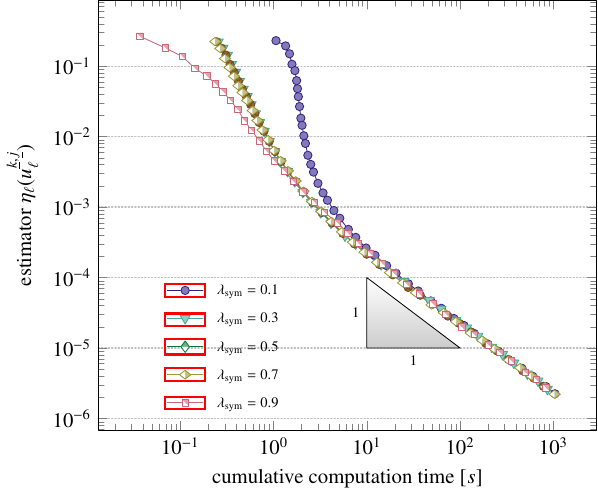}
	\caption{\label{fig:optimality_lamsym} Convergence history plot of the
		error estimator with respect to the
		number of degrees of freedom (left) and the computation time (right) for the nonsymmetric experiment~\eqref{eq:experiment} with $p=2$ and $\delta = 0.5$ for
		several symmetrization parameters \(\lambda_{\mathrm{sym}}
		\in
		\{0.1, 0.3, 0.5, 0.7, 0.9\}\) and fixed
		algebraic solver parameter \(\lambda_{\mathrm{alg}} = 0.7\) and marking parameter \(\theta = 0.3\).
	}
\end{figure}
\begin{figure}[t]
	\includegraphics[width=0.45\textwidth]{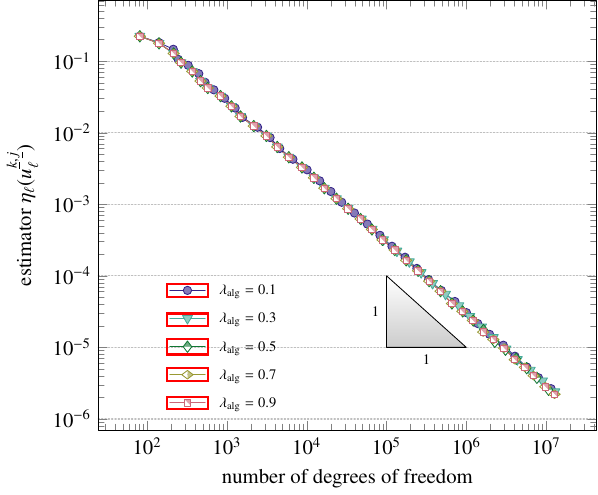}
	\hfil
	\includegraphics[width=0.45\textwidth]{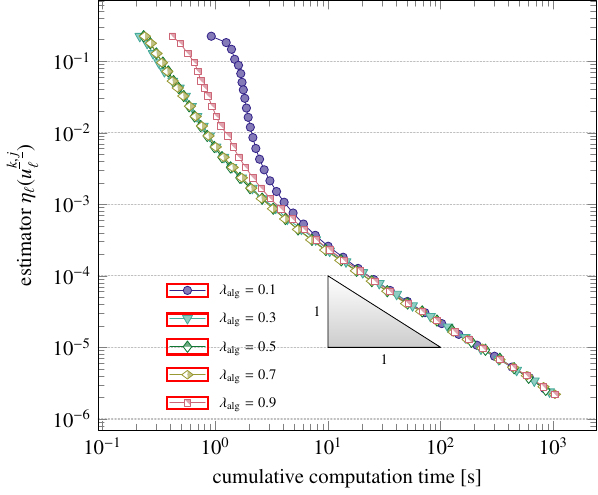}
	\caption{\label{fig:optimality_lamalg} Convergence history plot of the
		error estimator with respect to the number of degrees of freedom (left) and the computation time (right) for the nonsymmetric experiment~\eqref{eq:experiment} with $p=2$ and $\delta = 0.5$ for
		several algebraic solver parameters \(\lambda_{\mathrm{alg}} \in
		\{0.1, 0.3, 0.5, 0.7, 0.9\}\) and fixed symmetrization parameter \(\lambda_{\mathrm{sym}} = 0.7\) and marking parameter \(\theta = 0.3\).
	}
\end{figure}
\emph{Optimality of Algorithm~\ref{algorithm:double} with respect to
	large
	marking parameter \(\theta\).}
We choose the polynomial
degree \(p = 2\), \(\delta = 0.5\), and solver-stopping parameters \(\lambda_{\textup{alg}} = \lambda_{\textup{sym}}
= 0.7\). Figure~\ref{fig:optimality_theta} shows that also
for moderate marking parameters \(\theta\),
Algorithm~\ref{algorithm:double} guarantees optimal convergence rates with
respect to the number of degrees of freedom
and the cumulative computation time. Moreover, we observe that a
very small as well as a large choice of \(\theta\) lead to a worse performance.
\begin{figure}

	\includegraphics[width=0.45\textwidth]{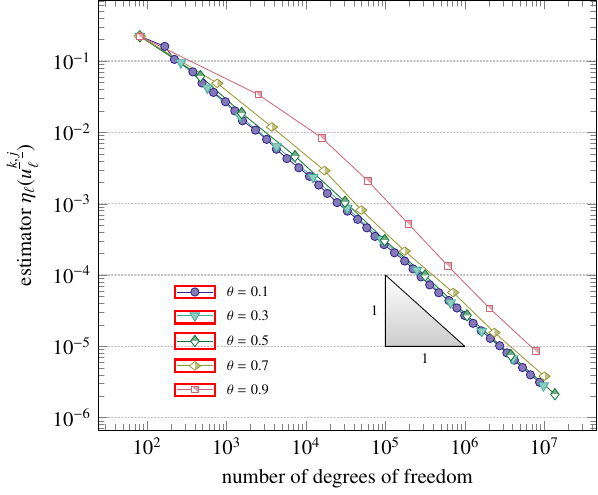}
	\hfil
	\includegraphics[width=0.45\textwidth]{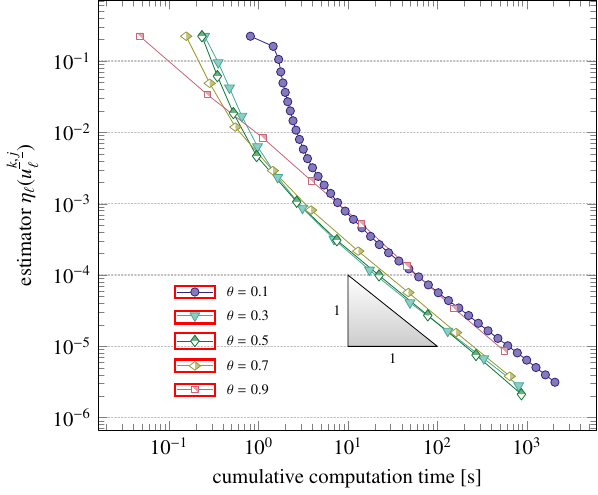}
	\caption{\label{fig:optimality_theta} Convergence history plot of the error
		estimator with respect to the number of degrees of freedom (left) and the computation time (right) for the nonsymmetric experiment~\eqref{eq:experiment} with $p=2$ and $\delta = 0.5$ for several
		Dörfler marking parameters \(\theta \in
		\{0.1, 0.3, 0.5, 0.7, 0.9\}\) and fixed solver-stopping parameters
		\(\lambda_{\mathrm{sym}} = \lambda_{\mathrm{alg}} = 0.7\).}
\end{figure}
\begin{figure}
	\includegraphics[width=0.45\textwidth]{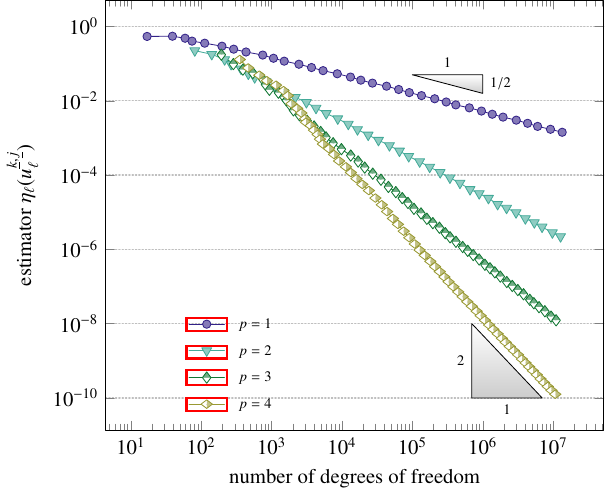}
	\hfil
	\includegraphics[width=0.45\textwidth]{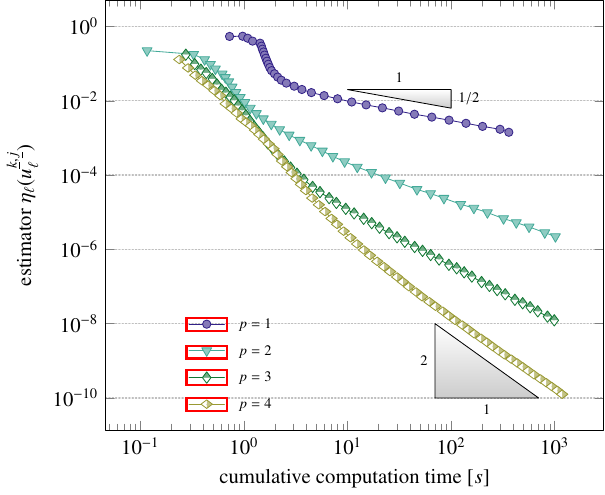}
	\caption{\label{fig:optimality_p} Convergence history plot of the error
		estimator with respect to the
		number of degrees of freedom
		(left) and with
		respect to the overall computation time (right) for the nonsymmetric experiment~\eqref{eq:experiment} with $\delta = 0.5$
		for several polynomial degrees \(p = 1,2,3,4\),
		and fixed marking parameter \(\theta = 0.3\) and
		solver-stopping parameters \(\lambda_{\mathrm{sym}} =
		\lambda_{\mathrm{alg}} = 0.7\).}
\end{figure}

%%%%%%%%%%%%%%%%%%%%%%%%%%%%%%%%%%%%%%%%%%%%%%%%%%%%%%%%%%%%%%%%%%%%%%%%%%%%%%%
\subsection{AFEM for a strongly monotone and Lipschitz continuous nonlinearity}
%%%%%%%%%%%%%%%%%%%%%%%%%%%%%%%%%%%%%%%%%%%%%%%%%%%%%%%%%%%%%%%%%%%%%%%%%%%%%%%
On the Z-shaped domain
	\(\Omega \coloneqq (-1, 1)^2 \setminus \operatorname{conv}\{ (-1, 0), (0, 0), (-1, -1) \}\)
	and with the nonlinearity
	\(a(x, t) = 1 + \log(1+t) / (1+t)\) for all \(x \in \Omega\) and \(t \ge 0\),
	we consider the quasi-linear elliptic PDE with homogeneous Dirichlet boundary conditions
	\begin{equation}\label{eq:experiment:nonlinear}
		-\operatorname{div} \bigl(a(\cdot, | \nabla u^\star |^2) \, \nabla u^\star \bigr) + u^\star
		= 1 \text{ in } \Omega
		\quad \text{and} \quad u^\star = 0 \text{ on } \partial \Omega
	\end{equation}
	Hence, the nonlinearity \(a(\cdot)\) satisfies the growth condition in Proposition~\ref{prop:nonlinear:axioms}
	with constants \(\alpha \approx 0.9582898017\) and \(L \approx 1.542343818\).
	In the experiments, we use the optimal damping parameter \(\delta = 1/L\) and
	the fixed polynomial degree \(p = 1\). Figure~\ref{fig:nonlinear:meshes} illustrates
	the initial triangulation \(\mathcal{T}_{0}\), the adaptively generated mesh \(\mathcal{T}_{7}\)
	with \(1483\) triangles, and the computed solution \(u_{7}^{\underline{k}, \underline{j}}\).
	In the Figures~\ref{fig:nonlinear:lamlin}--\ref{fig:nonlinear:theta}, we observe that
	the strategy from Algorithm~\ref{algorithm:double} applied to this nonlinear problem
	guarantees optimal convergence
	rates \(-1/2\) with respect to the number of degrees of freedom and the cumulative
	computation time for arbitrary marking parameters \(\theta\), linearization
	parameters \(\lambda_{\textup{lin}}\), and algebraic solver parameters \(\lambda_{\textup{alg}}\).
	In particular, even large values of \(\theta\), \(\lambda_{\textup{lin}}\), and
	\(\lambda_{\textup{alg}}\)
	lead to optimal convergence rates. Table~\ref{tab:cost:nonlinear} summarizes the optimal
	selection of the adaptivity parameters for the nonlinear problem~\eqref{eq:experiment:nonlinear}
	and indicates that moderate values of \(\theta\) in combination with large values of
	\(\lambda_{\textup{lin}}\) and \(\lambda_{\textup{alg}}\) are beneficial in terms
	of computational cost.

\begin{figure}
	\centering
	\includegraphics[width = 0.35\textwidth]{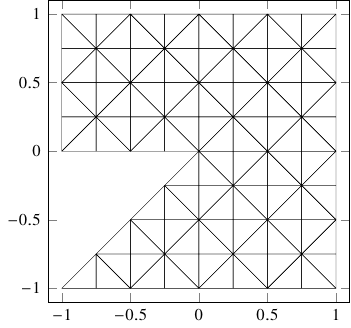}
	\hfil
	\includegraphics[width = 0.35\textwidth]{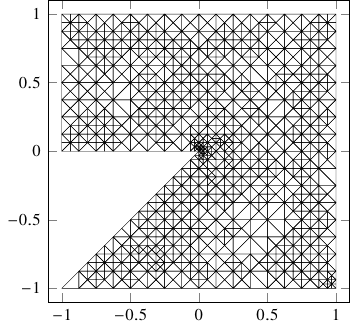}
	\\
	\includegraphics[width = 0.45\textwidth]{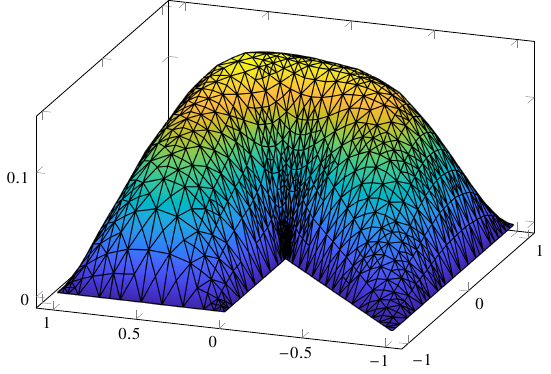}
	\caption{Initial triangulation \(\mathcal{T}_{0}\), adaptively generated mesh \(\mathcal{T}_{7}\) with \(1483\) triangles,
		and the computed solution \(u_{7}^{\underline{k}, \underline{j}}\) for the nonlinear experiment~\eqref{eq:experiment:nonlinear}
		with polynomial degree \(p=1\), optimal damping parameter \(\delta = 1/L\),
		marking parameter \(\theta = 0.3\), linearization parameter \(\lambda_{\textup{lin}} = 0.7\),
		and algebraic solver parameter \(\lambda_{\textup{alg}} = 0.7\).}
	\label{fig:nonlinear:meshes}
\end{figure}

\begin{figure}
	\centering
	% \begin{tikzpicture}
	% 	\input{figure_Ailfem_lamlin_dofs.tex}
	% \end{tikzpicture}
	% \hfil	
	% \begin{tikzpicture}
	% 	\input{figure_Ailfem_lamlin_time.tex}
	% \end{tikzpicture}
	\includegraphics[width = 0.45\textwidth]{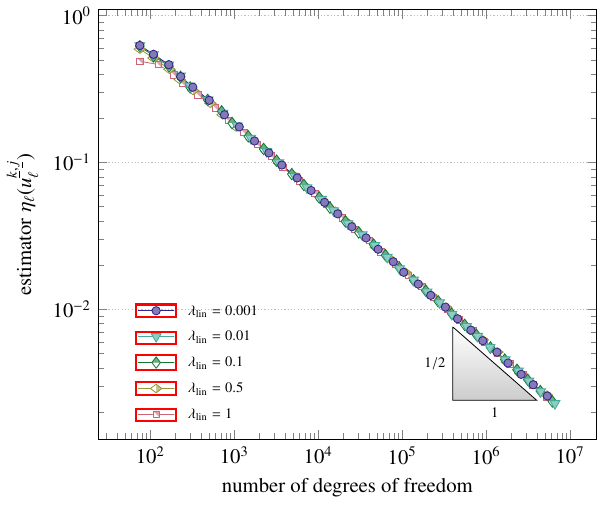}
	\hfil
	\includegraphics[width = 0.45\textwidth]{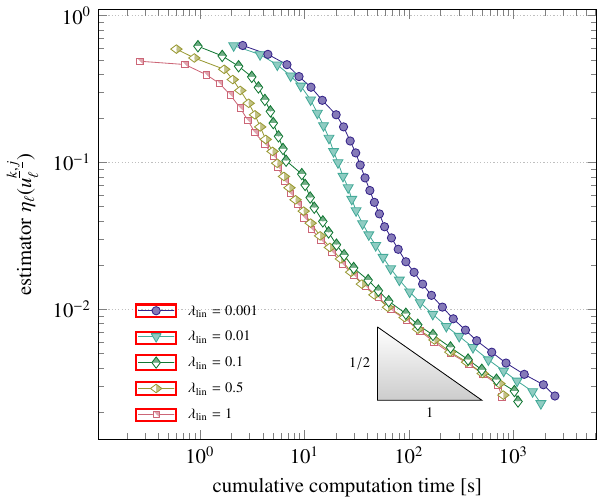}
	\caption{Convergence history plot of the error estimator \(\eta_{\ell}(u_{\ell}^{\underline{k}, \underline{j}})\) with respect to the number of degrees of freedom (left) and the cumulative computation time (right) for the nonlinear experiment~\eqref{eq:experiment:nonlinear}
		with polynomial degree \(p=1\) and
		optimal damping parameter \(\delta = 1/L\) for several linearization
		parameters \(\lambda_{\textup{lin}} \in \{0.001, 0.01, 0.1, 0.5, 1\}\) and
		fixed marking parameter \(\theta = 0.3\) and algebraic solver parameter
		\(\lambda_{\textup{alg}} = 0.7\).}
	\label{fig:nonlinear:lamlin}
\end{figure}

\begin{figure}
	\includegraphics[width = 0.45\textwidth]{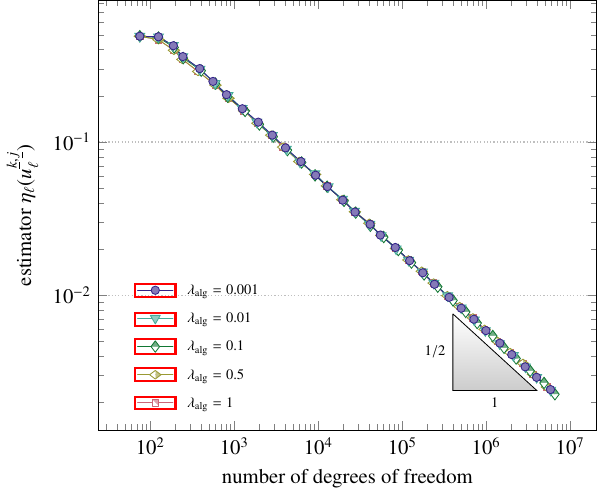}
	\hfil
	\includegraphics[width = 0.45\textwidth]{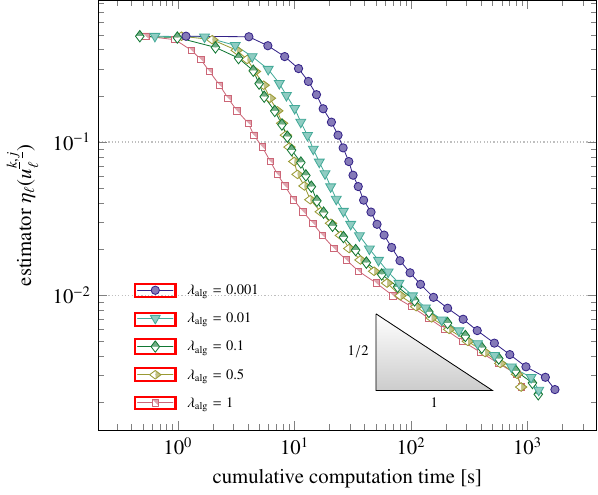}
	\caption{Convergence history plot of the error estimator \(\eta_{\ell}(u_{\ell}^{\underline{k}, \underline{j}})\) with respect to the number of degrees of freedom (left) and the cumulative computation time (right) for the nonlinear experiment~\eqref{eq:experiment:nonlinear}
		with polynomial degree \(p=1\) and
		optimal damping parameter \(\delta = 1/L\)
		for several algebraic solver parameters \(\lambda_{\textup{alg}} \in \{0.001, 0.01, 0.1, 0.5, 1\}\)
		and fixed marking parameter \(\theta = 0.3\), and linearization parameter
		\(\lambda_{\textup{lin}} = 0.7\).}
	\label{fig:nonlinear:lamalg}
\end{figure}

\begin{figure}
	\includegraphics[width = 0.45\textwidth]{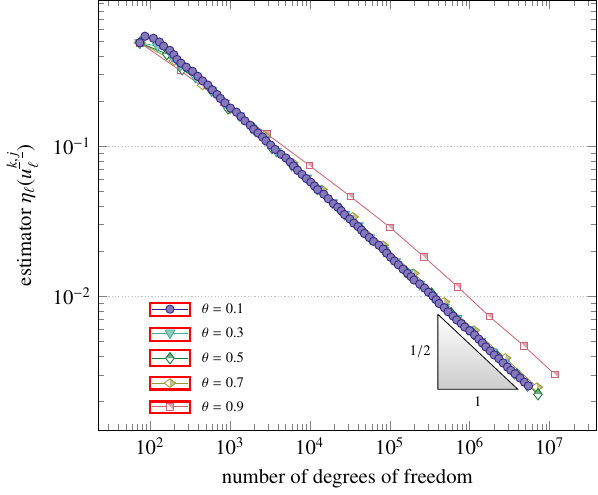}
	\hfil
	\includegraphics[width = 0.45\textwidth]{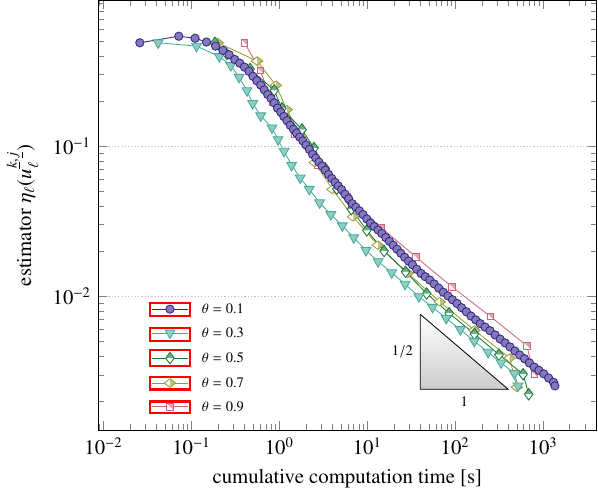}
	\caption{Convergence history plot of the error estimator
		\(\eta_{\ell}(u_{\ell}^{\underline{k}, \underline{j}})\) with respect to
		the number of degrees of freedom (left) and the cumulative computation time
		(right) for the nonlinear experiment~\eqref{eq:experiment:nonlinear} with
		polynomial degree \(p=1\) and
		optimal damping parameter \(\delta = 1/L\) for
		several marking parameters \(\theta \in \{0.1, 0.3, 0.5, 0.7, 0.9\}\)
		and solver parameters \(\lambda_{\textup{lin}} = \lambda_{\textup{alg}} = 0.7\).}
	\label{fig:nonlinear:theta}
\end{figure}

\begin{table}
	\centering
	\includegraphics[width = \textwidth]{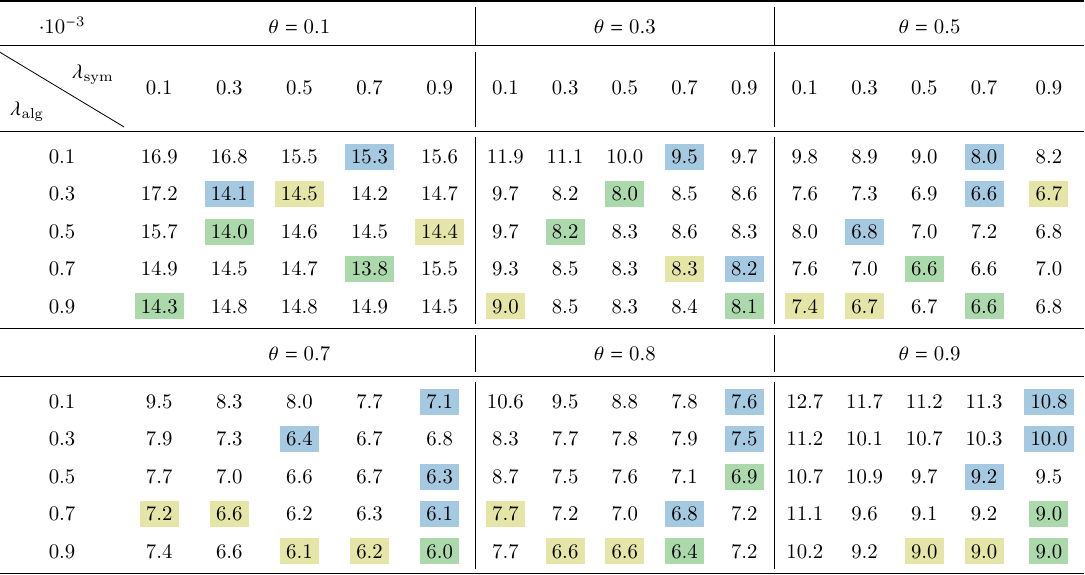}
	\caption{Optimal selection of parameters with
		respect to the computational costs for the nonlinear experiment~\eqref{eq:experiment:nonlinear}
		with $p=1$ and $\delta = 1/L$. For the comparison, we consider
		the
		weighted
		cumulative time $ \big[ \eta_{\ell}(u_\ell^{\underline{k},
					\underline{j}}) \,
				\sum_{|
					\ell', k', j' | \le
					| \ell, \underline{k}, \underline{j} |}
				\mathrm{time}(\ell')
				\big]$ (values in \(10^{-3}\)) with stopping
		criterion $\eta_\ell(u_\ell^{\underline{k}, \underline{j}}) < 5 \cdot
			10^{-2} \, \eta_0(u_0^{0,0})$ for
		various choices of
		$\lambda_{\textup{lin}}$, $\lambda_{\textup{alg}}$, and $\theta$.
		In each $\theta$-block, we mark in yellow the best choice per column, in blue the best choice per row, and in green when both choices coincide. The best choices for $\lambda_{\textup{lin}}$ and $\lambda_{\textup{alg}}$ are observed for $\theta = 0.5$ and $\theta=0.7$.}
	\label{tab:cost:nonlinear}
\end{table}

\section*{Funding}
This research was funded in whole or in part by the Austrian Science Fund (FWF)
[\href{https://www.fwf.ac.at/en/research-radar/10.55776/F65}{10.55776/F65},
		\href{https://www.fwf.ac.at/en/research-radar/10.55776/I6802}{10.55776/I6802}, and
		\href{https://www.fwf.ac.at/en/research-radar/10.55776/P33216}{10.55776/P33216}].
For open access purposes, the author has applied a CC BY public copyright license
to any author accepted manuscript version arising from this submission.

\sloppy
\printbibliography

@Article{fpp2014,
  author   = {Feischl, M. and Page, M. and Praetorius, D.},
  journal  = {J. Comput. Appl. Math.},
  title    = {Convergence and quasi-optimality of adaptive {FEM} with
              inhomogeneous {D}irichlet data},
  year     = {2014},
  issn     = {0377-0427,1879-1778},
  pages    = {481--501},
  volume   = {255},
  doi      = {10.1016/j.cam.2013.06.009},
  fjournal = {Journal of Computational and Applied Mathematics},
  mrclass  = {65N30 (65N12 65N15 65N50)},
  mrnumber = {3093437},
}

@Article{afkpp2013,
  author     = {Aurada, M. and Feischl, M. and Kemetm\"{u}ller, J. and Page,
              M. and Praetorius, D.},
  journal    = {ESAIM Math. Model. Numer. Anal.},
  title      = {Each {$H^{1/2}$}-stable projection yields convergence and
              quasi-optimality of adaptive {FEM} with inhomogeneous
              {D}irichlet data in {$\mathbb{R}^d$}},
  year       = {2013},
  issn       = {2822-7840,2804-7214},
  number     = {4},
  pages      = {1207--1235},
  volume     = {47},
  doi        = {10.1051/m2an/2013069},
  fjournal   = {ESAIM. Mathematical Modelling and Numerical Analysis},
  mrclass    = {65N30 (65N50)},
  mrnumber   = {3082295},
  mrreviewer = {Steve\ Wright},
}

@Article{mns2003,
  author     = {Morin, Pedro and Nochetto, Ricardo H. and Siebert, Kunibert
              G.},
  journal    = {Math. Comp.},
  title      = {Local problems on stars: a posteriori error estimators,
              convergence, and performance},
  year       = {2003},
  issn       = {0025-5718,1088-6842},
  number     = {243},
  pages      = {1067--1097},
  volume     = {72},
  doi        = {10.1090/S0025-5718-02-01463-1},
  fjournal   = {Mathematics of Computation},
  mrclass    = {65N15 (65N50)},
  mrnumber   = {1972728},
  mrreviewer = {Piotr\ P.\ Matus},
}

@Article{kellogg1974,
  author     = {Kellogg, R. Bruce},
  journal    = {Applicable Anal.},
  title      = {On the {P}oisson equation with intersecting interfaces},
  year       = {1974},
  issn       = {0003-6811},
  pages      = {101--129},
  volume     = {4},
  doi        = {10.1080/00036817408839086},
  fjournal   = {Applicable Analysis. An International Journal},
  mrclass    = {35J25},
  mrnumber   = {393815},
  mrreviewer = {Yu.\ V.\ Kostarchuk},
}

@Article{heid2023,
  author   = {Heid, Pascal},
  journal  = {Partial Differ. Equ. Appl.},
  title    = {A damped {K}a\v{c}anov scheme for the numerical solution of a relaxed {$p(x)$}-{P}oisson equation},
  year     = {2023},
  issn     = {2662-2963,2662-2971},
  number   = {5},
  pages    = {Paper No. 40, 20},
  volume   = {4},
  doi      = {10.1007/s42985-023-00259-7},
  fjournal = {Partial Differential Equations and Applications},
  mrclass  = {65N30 (35J05 47J25)},
  mrnumber = {4635915},
}

@Article{bds2023,
  author   = {Balci, Anna Kh. and Diening, Lars and Storn, Johannes},
  journal  = {SIAM J. Numer. Anal.},
  title    = {Relaxed {K}a\v{c}anov scheme for the {$p$}-{L}aplacian with
              large exponent},
  year     = {2023},
  issn     = {0036-1429,1095-7170},
  number   = {6},
  pages    = {2775--2794},
  volume   = {61},
  doi      = {10.1137/22M1528550},
  fjournal = {SIAM Journal on Numerical Analysis},
  mrclass  = {65N30 (35J70 35J92 65N22)},
  mrnumber = {4668389},
}

@Article{dftw2020,
  author     = {Diening, L. and Fornasier, M. and Tomasi, R. and Wank, M.},
  journal    = {Numer. Math.},
  title      = {A relaxed {K}a\v{c}anov iteration for the {$p$}-{P}oisson problem},
  year       = {2020},
  issn       = {0029-599X,0945-3245},
  number     = {1},
  pages      = {1--34},
  volume     = {145},
  doi        = {10.1007/s00211-020-01107-1},
  fjournal   = {Numerische Mathematik},
  mrclass    = {65N30 (35J70 35J92 65N12 65N22)},
  mrnumber   = {4091593},
  mrreviewer = {Michael\ Neilan},
}

@Article{ev2013,
  author     = {Ern, A. and Vohral\'{\i}k, M.},
  journal    = {SIAM J. Sci. Comput.},
  title      = {Adaptive inexact {N}ewton methods with a posteriori stopping criteria for nonlinear diffusion {PDE}s},
  year       = {2013},
  issn       = {1064-8275,1095-7197},
  number     = {4},
  pages      = {A1761--A1791},
  volume     = {35},
  doi        = {10.1137/120896918},
  fjournal   = {SIAM Journal on Scientific Computing},
  mrclass    = {65H10 (65M08 65M22 65M50 65M60)},
  mrnumber   = {3072765},
  mrreviewer = {B\"{u}lent\ Karas\"{o}zen},
}

@Article{veeser2002,
  author   = {Veeser, A.},
  journal  = {Numer. Math.},
  title    = {Convergent adaptive finite elements for the nonlinear {L}aplacian},
  year     = {2002},
  issn     = {0029-599X,0945-3245},
  number   = {4},
  pages    = {743--770},
  volume   = {92},
  doi      = {10.1007/s002110100377},
  fjournal = {Numerische Mathematik},
  mrclass  = {65N30 (65N12)},
  mrnumber = {1935808},
}

@Article{bdk2012,
  author     = {Belenki, L. and Diening, L. and Kreuzer, C.},
  journal    = {IMA J. Numer. Anal.},
  title      = {Optimality of an adaptive finite element method for the
                {$p$}-{L}aplacian equation},
  year       = {2012},
  issn       = {0272-4979,1464-3642},
  number     = {2},
  pages      = {484--510},
  volume     = {32},
  doi        = {10.1093/imanum/drr016},
  fjournal   = {IMA Journal of Numerical Analysis},
  mrclass    = {65N30 (65N15)},
  mrnumber   = {2911397},
  mrreviewer = {Snorre\ H.\ Christiansen},
}

@Article{dk2008,
  author     = {Diening, L. and Kreuzer, C.},
  journal    = {SIAM J. Numer. Anal.},
  title      = {Linear convergence of an adaptive finite element method for
                the {$p$}-{L}aplacian equation},
  year       = {2008},
  issn       = {0036-1429,1095-7170},
  number     = {2},
  pages      = {614--638},
  volume     = {46},
  fjournal   = {SIAM Journal on Numerical Analysis},
  mrclass    = {65N30 (35J60 35J70)},
  mrnumber   = {2383205},
  mrreviewer = {Olivier\ Besson},
  DOI = {10.1137/070681508}
}

@Article{doerfler1996,
  author     = {D{ö}rfler, W.},
  journal    = {SIAM J. Numer. Anal.},
  title      = {A convergent adaptive algorithm for {P}oisson's equation},
  year       = {1996},
  issn       = {0036-1429},
  number     = {3},
  pages      = {1106--1124},
  volume     = {33},
  doi        = {10.1137/0733054},
  fjournal   = {SIAM Journal on Numerical Analysis},
  mrclass    = {65N50 (65N55)},
  mrnumber   = {1393904},
  mrreviewer = {S. F. McCormick},
}

@Article{mns2000,
  author     = {Morin, P. and Nochetto, R.H. and Siebert, K. G.},
  journal    = {SIAM J. Numer. Anal.},
  title      = {Data oscillation and convergence of adaptive {FEM}},
  year       = {2000},
  issn       = {0036-1429},
  number     = {2},
  pages      = {466--488},
  volume     = {38},
  doi        = {10.1137/s0036142999360044},
  fjournal   = {SIAM Journal on Numerical Analysis},
  mrclass    = {65N30 (65N55 65Y20)},
  mrnumber   = {1770058},
  mrreviewer = {Petr N. Vabishchevich},
}

@Article{cg2012,
  author     = {Carstensen, C. and Gedicke, J.},
  journal    = {SIAM J. Numer. Anal.},
  title      = {An adaptive finite element eigenvalue solver of asymptotic 
                quasi-optimal computational complexity},
  year       = {2012},
  issn       = {0036-1429},
  number     = {3},
  pages      = {1029--1057},
  volume     = {50},
  coden      = {SJNAAM},
  doi        = {10.1137/090769430},
  mrclass    = {65N25 (65N15 65N30)},
  mrnumber   = {2970733},
  mrreviewer = {V. L. Makarov},
  DOI = {10.1137/090769430}
}

@Article{mv2023,
  author        = {Mitra, K. and Vohral{\'i}k, M.},
  title         = {{Guaranteed, locally efficient, and robust a posteriori estimates 
                   for nonlinear elliptic problems in iteration-dependent norms. An orthogonal 
                   decomposition result based on iterative linearization}},
  year          = {2023},
  month         = Jul,
  archiveprefix = {HAL preprint},
  eprint        = {hal-04156711},
  hal_id        = {hal-04156711},
  hal_version   = {v1},
}

@Article{hmrv2023,
  author        = {Harnist, A. and Mitra, K. and Rappaport, A. and 
                   Vohral{\'i}k, M.},
  title         = {{Robust energy a posteriori estimates for nonlinear elliptic 
                   problems}},
  year          = {2023},
  month         = May,
  archiveprefix = {HAL preprint},
  eprint        = {hal-04033438},
  hal_id        = {hal-04033438},
  hal_version   = {v2},
}

@Article{dgs2023,
  author        = {L. Diening and L. Gehring and J. Storn},
  title         = {Adaptive Mesh Refinement for arbitrary initial Triangulations},
  year          = {2023},
  archiveprefix = {arXiv},
  eprint        = {2306.02674},
}

@Article{hw2022,
  author  = {Heid, P. and Wihler, T.P.},
  journal = {ESAIM: Math. Model. Numer. Anal},
  title   = {A modified Kacanov iteration scheme with application to quasilinear diffusion models},
  year    = {2022},
  number  = {2},
  pages   = {433-450},
  volume  = {56},
  doi     = {10.1051/m2an/2022008},
}

@Article{hw2020a,
  author     = {Heid, P. and Wihler, T.P.},
  journal    = {Math. Comp.},
  title      = {Adaptive iterative linearization {G}alerkin methods for nonlinear problems},
  year       = {2020},
  number     = {326},
  pages      = {2707--2734},
  volume     = {89},
  doi        = {10.1090/mcom/3545},
  fjournal   = {Mathematics of Computation},
  mrclass    = {65N30 (47H05 47H10 47J25 49M15 65J15 65N12 65N50)},
  mrnumber   = {4136544},
  mrreviewer = {Huai\ Zhang},
}

@Article{cw2017,
  author  = {S. Congreve and T.P. Wihler},
  journal = {J. Comp. Appl. Math.},
  title   = {Iterative {G}alerkin discretizations for strongly monotone problems},
  year    = {2017},
  issn    = {0377-0427},
  pages   = {457-472},
  volume  = {311},
  doi     = {10.1016/j.cam.2016.08.014},
}

@Article{gmz2012,
  author  = {E.M. Garau and P. Morin and C. Zuppa},
  journal = {Numer. Math: Theory, Meth. Appl.},
  title   = {Quasi-Optimal Convergence Rate of an {AFEM} for quasi-linear 
             Problems of Monotone Type},
  year    = {2012},
  number  = {2},
  pages   = {131--156},
  volume  = {5},
  DOI = {10.4208/nmtma.2012.m1023}
}

@Article{gmz2011,
  author  = {E.M. Garau and P. Morin and C. Zuppa},
  journal = {Appl. Numer. Math.},
  title   = {Convergence of an adaptive {K}ačanov {FEM} for quasi-linear problems},
  year    = {2011},
  number  = {4},
  pages   = {512-529},
  volume  = {61},
  doi     = {10.1016/j.apnum.2010.12.001},
}

@InCollection{pl1954,
  author     = {Lax, P.D. and Milgram, A.N.},
  booktitle  = {Contributions to the theory of partial differential equations},
  publisher  = {Princeton Univ. Press, Princeton, NJ},
  title      = {Parabolic equations},
  year       = {1954},
  pages      = {167--190},
  series     = {Ann. of Math. Stud.},
  volume     = {no. 33},
  doi        = {10.1515/9781400882182-010},
  mrclass    = {35.0X},
  mrnumber   = {67317},
  mrreviewer = {L.\ G\aa rding},
}

@Article{Karkulik2013a,
  author   = {Karkulik, M. and Pavlicek, D. and Praetorius, D.},
  journal  = {Constr. Approx.},
  title    = {On 2{D} newest vertex bisection: optimality of mesh-closure and {$H^1$}-stability of {$L_2$}-projection},
  year     = {2013},
  number   = {2},
  pages    = {213--234},
  volume   = {38},
  doi      = {10.1007/s00365-013-9192-4},
  fjournal = {Constructive Approximation. An International Journal for Approximations and Expansions},
  mrclass  = {65N50 (65N30 65Y20)},
  mrnumber = {3097045},
}

@Article{cfpp2014,
  author  = {Carstensen, C. and Feischl, M. and Page, M. and Praetorius, D.},
  journal = {Comput. Math. Appl.},
  title   = {Axioms of Adaptivity},
  year    = {2014},
  number  = {6},
  pages   = {1195--1253},
  volume  = {67},
  doi     = {10.1016/j.camwa.2013.12.003},
}

@Article{ckns2008,
  author     = {Cascón, J.M. and Kreuzer, C. and Nochetto, R.H. and Siebert, K.G.},
  journal    = {SIAM J. Numer. Anal.},
  title      = {Quasi-optimal convergence rate for an adaptive finite element method},
  year       = {2008},
  number     = {5},
  pages      = {2524--2550},
  volume     = {46},
  doi        = {10.1137/07069047x},
  fjournal   = {SIAM Journal on Numerical Analysis},
  mrclass    = {65N30 (41A25)},
  mrnumber   = {2421046},
  mrreviewer = {Hans-Peter Helfrich},
}

@Article{stevenson2007,
  author     = {Stevenson, R.P.},
  journal    = {Found. Comput. Math.},
  title      = {Optimality of a standard adaptive finite element method},
  year       = {2007},
  number     = {2},
  pages      = {245--269},
  volume     = {7},
  fjournal   = {Foundations of Computational Mathematics. The Journal of the Society for the Foundations of Computational Mathematics},
  mrclass    = {65N30},
  mrnumber   = {2324418},
  mrreviewer = {Erwin Stein},
  DOI = {10.1007/s10208-005-0183-0}

}

@Article{ffp2014,
  author  = {Feischl, M. and F\"{u}hrer, T. and Praetorius, D.},
  journal = {SIAM J. Numer. Anal.},
  title   = {Adaptive {FEM} with optimal convergence rates for a certain class of nonsymmetric and possibly nonlinear problems},
  year    = {2014},
  number  = {2},
  pages   = {601--625},
  volume  = {52},
  doi     = {10.1137/120897225},
}

@Article{bhp2017,
  author  = {Bespalov, A. and Haberl, A. and Praetorius, D.},
  journal = {Comput. Methods Appl. Mech. Engrg.},
  title   = {Adaptive {FEM} with coarse initial mesh guarantees optimal convergence rates for compactly perturbed elliptic problems},
  year    = {2017},
  pages   = {318--340},
  volume  = {317},
  doi     = {10.1016/j.cma.2016.12.014},
}

@Article{ghps2021,
  author  = {Gantner, G. and Haberl, A. and Praetorius, D. and Schimanko, S.},
  journal = {Math. Comp.},
  title   = {Rate optimality of adaptive finite element methods with respect to overall computational costs},
  year    = {2021},
  number  = {331},
  pages   = {2011--2040},
  volume  = {90},
  doi     = {10.1090/mcom/3654},
}

@Article{fp2018,
  author     = {F\"{u}hrer, T. and Praetorius, D.},
  journal    = {Comput. Math. Appl.},
  title      = {A linear {U}zawa-type {FEM}-{BEM} solver for nonlinear transmission problems},
  year       = {2018},
  issn       = {0898-1221,1873-7668},
  number     = {8},
  pages      = {2678--2697},
  volume     = {75},
  doi        = {10.1016/j.camwa.2017.12.035},
  fjournal   = {Computers \& Mathematics with Applications. An International Journal},
  mrclass    = {65N22 (35J25 35J62 65N30 65N38)},
  mrnumber   = {3787479},
  mrreviewer = {Jennifer\ Pestana},
}

@Article{feischl2022,
  author     = {Feischl, M.},
  journal    = {Math. Comp.},
  title      = {Inf-sup stability implies quasi-orthogonality},
  year       = {2022},
  issn       = {0025-5718,1088-6842},
  number     = {337},
  pages      = {2059--2094},
  volume     = {91},
  doi        = {10.1090/mcom/3748},
  fjournal   = {Mathematics of Computation},
  mrclass    = {65N30 (15A23 65N50)},
  mrnumber   = {4451456},
  mrreviewer = {Dmitrii\ Legatiuk},
}

@Article{pp2020,
  author     = {Pfeiler, C.M. and Praetorius, D.},
  journal    = {Math. Comp.},
  title      = {D\"{o}rfler marking with minimal cardinality is a linear complexity problem},
  year       = {2020},
  issn       = {0025-5718,1088-6842},
  number     = {326},
  pages      = {2735--2752},
  volume     = {89},
  doi        = {10.1090/mcom/3553},
  fjournal   = {Mathematics of Computation},
  mrclass    = {65N30 (65N50 68Q25)},
  mrnumber   = {4136545},
  mrreviewer = {Jinbiao\ Wu},
}

@Article{aisfem,
  author   = {Brunner, Maximilian and Innerberger, Michael and Miraçi, Ani and Praetorius, Dirk and 
              Streitberger, Julian and Heid, Pascal},
  journal  = {IMA J.\ Numer.\ Anal.},
  title    = {Adaptive {FEM} with quasi-optimal overall cost for nonsymmetric linear elliptic PDEs},
  year     = {2024},
  number   = {3},
  pages    = {1560-1596},
  volume   = {44},
  doi      = {10.1093/imanum/drad039},
  fjournal = {IMA Journal of Numerical Analysis},
}

@Article{bhimps2023b,
  author   = {Brunner, Maximilian and Innerberger, Michael and Miraçi, Ani and Praetorius, Dirk and 
              Streitberger, Julian and Heid, Pascal},
  journal  = {IMA J. Numer. Anal.},
  title    = {Corrigendum to: Adaptive {FEM} with quasi-optimal overall cost for nonsymmetric linear elliptic {PDEs}},
  year     = {2024},
  number   = {3},
  pages    = {1903–1909},
  volume   = {44},
  doi      = {10.1093/imanum/drad103},
  fjournal = {IMA Journal of Numerical Analysis},
}

@Article{ghps2018,
  author  = {Gantner, G. and Haberl, A. and Praetorius, D. and Stiftner, B.},
  journal = {IMA J. Numer. Anal.},
  title   = {Rate optimal adaptive {FEM} with inexact solver for nonlinear operators},
  year    = {2018},
  number  = {4},
  pages   = {1797--1831},
  volume  = {38},
  doi     = {10.1093/imanum/drx050}
}

@Article{AFFKP13,
  author     = {Aurada, M. and Feischl, M. and F\"{u}hrer, T. and
                Karkulik, M. and Praetorius, D.},
  journal    = {Appl. Numer. Math.},
  title      = {Energy norm based error estimators for adaptive {BEM} for
                hypersingular integral equations},
  year       = {2015},
  pages      = {15--35},
  volume     = {95},
  doi        = {10.1016/j.apnum.2013.12.004},
  fjournal   = {Applied Numerical Mathematics. An IMACS Journal},
  mrclass    = {65N38 (65N15)},
  mrnumber   = {3349683},
  mrreviewer = {Paul Andrew Martin},
}

@Book{Zeidler1990,
  author     = {Zeidler, E.},
  publisher  = {Springer, New York},
  title      = {Nonlinear functional analysis and its applications. {P}art {II}/{B} - Nonlinear monotone operators},
  year       = {1990},
  mrnumber   = {1033497},
  mrreviewer = {Jean Mawhin},
  pages      = {xviii+467},
  DOI = {10.1007/978-1-4612-0981-2}
}

@Article{Zarantonello1960,
  author  = {Zarantonello, E.H.},
  journal = {Math. Research Center Report},
  title   = {Solving functional equations by contractive averaging},
  year    = {1960},
  volume  = {160},
}

@Article{MooAFEM,
  author  = {M. Innerberger and D. Praetorius},
  journal = {Appl. Math. Comput.},
  title   = {MooAFEM: An object oriented Matlab code for higher-order adaptive {FEM} for (nonlinear) elliptic {PDE}s},
  year    = {2023},
  pages   = {127731},
  volume  = {442},
  doi     = {10.1016/j.amc.2022.127731},
}

@Article{MR2035007,
  author     = {Cohen, A. and Dahmen, W. and DeVore, R.},
  journal    = {SIAM J. Numer. Anal.},
  title      = {Adaptive wavelet schemes for nonlinear variational problems},
  year       = {2003},
  number     = {5},
  pages      = {1785--1823},
  volume     = {41},
  doi        = {10.1137/s0036142902412269},
  fjournal   = {SIAM Journal on Numerical Analysis},
  mrclass    = {65J15 (35A15 35A35 41A65 42C40)},
  mrnumber   = {2035007},
  mrreviewer = {Peter G. Binev},
}

@Article{MR1803124,
  author   = {Cohen, A. and Dahmen, W. and DeVore, R.},
  journal  = {Math. Comp.},
  title    = {Adaptive wavelet methods for elliptic operator equations: convergence rates},
  year     = {2001},
  number   = {233},
  pages    = {27--75},
  volume   = {70},
  fjournal = {Mathematics of Computation},
  mrclass  = {65N35 (65N22 65T60)},
  mrnumber = {1803124},
  DOI = {10.1090/S0025-5718-00-01252-7}
}

@Article{imps2022,
  author   = {Innerberger, Michael and Mira\c{c}i, Ani and Praetorius, Dirk
              and Streitberger, Julian},
  journal  = {ESAIM Math. Model. Numer. Anal.},
  title    = {{$hp$}-robust multigrid solver on locally refined meshes for
              {FEM} discretizations of symmetric elliptic {PDE}s},
  year     = {2024},
  number   = {1},
  pages    = {247--272},
  volume   = {58},
  fjournal = {ESAIM. Mathematical Modelling and Numerical Analysis},
  mrclass  = {65N55 (65N12 65N30 65Y20)},
  mrnumber = {4705852},
  DOI = {10.1051/m2an/2023104}
}

@Article{hw2020b,
  author     = {Heid, P. and Wihler, T.P.},
  journal    = {Calcolo},
  title      = {On the convergence of adaptive iterative linearized {G}alerkin methods},
  year       = {2020},
  number     = {3},
  volume     = {57},
  doi        = {10.1007/s10092-020-00368-4},
  fjournal   = {Calcolo. A Quarterly on Numerical Analysis and Theory of Computation},
  mrclass    = {65N30 (35J62 47H05 47H10 47J25 49M15 65N12 65N50)},
  mrnumber   = {4131951},
  mrreviewer = {Riccardo Sacco},
}

@Article{cnx2012,
  author     = {Chen, L. and Nochetto, R.H. and Xu, J.},
  journal    = {Numer. Math.},
  title      = {Optimal multilevel methods for graded bisection grids},
  year       = {2012},
  number     = {1},
  pages      = {1--34},
  volume     = {120},
  fjournal   = {Numer. Math.},
  mrclass    = {65N30 (65N50 65N55)},
  mrnumber   = {2885595},
  mrreviewer = {Constantin B\u{a}cu\c{t}\u{a}},
  DOI = {10.1007/s00211-011-0401-4}
}

@Article{wz2017,
  author     = {Wu, J. and Zheng, H.},
  journal    = {Appl. Numer. Math.},
  title      = {Uniform Convergence of Multigrid Methods for Adaptive Meshes},
  year       = {2017},
  pages      = {109–123},
  volume     = {113},
  address    = {NLD},
  doi        = {10.1016/j.apnum.2016.11.005},
  issue_date = {2017},
  numpages   = {15},
  publisher  = {Elsevier Science Publishers B. V.},
}

@Article{hpw2021,
  author   = {Heid, P. and Praetorius, D. and Wihler, T.P.},
  journal  = {Comput. Methods Appl. Math.},
  title    = {Energy contraction and optimal convergence of adaptive iterative linearized finite element methods},
  year     = {2021},
  number   = {2},
  pages    = {407--422},
  volume   = {21},
  doi      = {10.1515/cmam-2021-0025},
  fjournal = {Computational Methods in Applied Mathematics},
  mrclass  = {65N30 (35J62 47H05 47J25 65N12 65Y20)},
  mrnumber = {4235817},
}

@Article{bdd2004,
  author     = {Binev, P. and Dahmen, W. and DeVore, R.},
  journal    = {Numer. Math.},
  title      = {Adaptive finite element methods with convergence rates},
  year       = {2004},
  number     = {2},
  pages      = {219--268},
  volume     = {97},
  doi        = {10.1007/s00211-003-0492-7},
  fjournal   = {Numerische Mathematik},
  mrclass    = {65N50 (65N12 65N30 65Y20 68W25 68W40)},
  mrnumber   = {2050077},
  mrreviewer = {Thomas Apel},
}

@Article{cn2012,
  author   = {Casc\'{o}n, J.M. and Nochetto, R.H.},
  journal  = {IMA J. Numer. Anal.},
  title    = {Quasioptimal cardinality of {AFEM} driven by nonresidual estimators},
  year     = {2012},
  number   = {1},
  pages    = {1--29},
  volume   = {32},
  fjournal = {IMA Journal of Numerical Analysis},
  mrclass  = {65N30 (65N12 65N50)},
  mrnumber = {2875241},
  DOI = {10.1093/imanum/drr014}
}

@Article{stevenson2008,
  author  = {Stevenson, R.P.},
  journal = {Math. Comp.},
  title   = {The completion of locally refined simplicial partitions created by bisection},
  year    = {2008},
  number  = {261},
  pages   = {227--241},
  volume  = {77},
  doi     = {10.1090/s0025-5718-07-01959-x},
}

@Article{hpsv2021,
  author  = {A. Haberl and D. Praetorius and S. Schimanko and M. Vohral{\'{\i}}k},
  journal = {Numer. Math.},
  title   = {Convergence and quasi-optimal cost of adaptive algorithms for nonlinear operators including iterative linearization and algebraic solver},
  year    = {2021},
  number  = {3},
  pages   = {679--725},
  volume  = {147},
  doi     = {10.1007/s00211-021-01176-w},
}

%%%%%%%%%%%%%%%%%%%%%%%%%%%%%%%%%%%%%%%%%%%%%%%%%%%%%%%%%%%%%%%%%%%%%%%
%%%%%%%%%%%%%%%%%%%%%%%%%%%%%%%%%%%%%%%%%%%%%%%%%%%%%%%%%%%%%%%%%%%%%%%
\appendix
\section{Proofs of Lemma~\ref{lemma:summability:criterion},
  Lemma~\ref{lemma:summability}, and Lemma~\ref{lem:inexact_contraction}}
\label{appendix}
%%%%%%%%%%%%%%%%%%%%%%%%%%%%%%%%%%%%%%%%%%%%%%%%%%%%%%%%%%%%%%%%%%%%%%%
%%%%%%%%%%%%%%%%%%%%%%%%%%%%%%%%%%%%%%%%%%%%%%%%%%%%%%%%%%%%%%%%%%%%%%%

\begin{proof}[\textbf{Proof of Lemma~\ref{lemma:summability:criterion}}]
	The proof is split into four steps.

	\textbf{Step~1.}
	We consider the perturbed contraction of \((a_\ell)_{\ell \in \mathbb{N}_0}\) from~\eqref{eq:summability:criterion}.
	By induction on $n$, we see with the empty sum understood (as usual) as zero that
	\begin{equation*}
		a_{\ell+n} \le q^n a_\ell + \sum_{j=1}^n q^{n-j} b_{\ell+j-1} \quad \text{for all } \ell, n \in \mathbb{N}_0.
	\end{equation*}
	From this and the geometric series, we infer that
	\begin{equation}\label{eq:a_ell:quasi-mon}
		a_{\ell+n}
		\le
		q^n a_\ell + C_1 \Big( \sum_{j=1}^n q^{n-j} \Big) a_\ell
		\le
		\Big(q^n + \frac{C_1}{1-q} \Big) \, a_\ell
		\eqqcolon
		C_3 \, a_\ell
		\quad \text{for all } \ell, n \in \mathbb{N}_0.
	\end{equation}

	\textbf{Step~2.}
	Next, we note that the perturbed contraction of \((a_\ell)_{\ell \in \mathbb{N}_0}\) from
	\eqref{eq:summability:criterion} and the
	Young inequality with sufficiently small \(\varepsilon > 0\) ensure
	\begin{equation*}
		0 < \kappa
		\coloneqq
		(1+\varepsilon) \, q^2 < 1
		\quad \text{and} \quad
		a_{\ell+1}^2
		\stackrel{\eqref{eq:summability:criterion}}\le
		\kappa \, a_{\ell}^2
		+ (1+\varepsilon^{-1}) \, b_\ell^2 \quad \text{for all } \ell \in \mathbb{N}_0.
	\end{equation*}
	This and the summability of
	\((b_\ell)_{\ell \in \mathbb{N}_0}\) from \eqref{eq:summability:criterion} guarantee
	\begin{equation*}
		\sum_{\ell' = \ell+1}^{\ell+N} a_{\ell'}^2
		=
		\sum_{\ell' = \ell}^{\ell+N-1} a_{\ell'+1}^2
		\stackrel{\mathclap{\eqref{eq:summability:criterion}}}\le
		\kappa \sum_{\ell' = \ell}^{\ell+N-1} a_{\ell'}^2 + (1+\varepsilon^{-1}) C_2 \, N^{1-\delta} \,
		a_\ell^2.
	\end{equation*}
	Rearranging the estimate, we arrive at
	\begin{equation}\label{eq1:feischl}
		\sum_{\ell' = \ell}^{\ell+N} a_{\ell'}^2
		\le
		\Bigl[ 1 + \frac{\kappa + (1+\varepsilon^{-1}) C_2 \, N^{1-\delta}}{1-\kappa} \Bigr]
		\, a_\ell^2 \eqqcolon D_N
		\, a_\ell^2
		\quad \text{for all } \ell, N \in \mathbb{N}_0,
	\end{equation}
	where we note that $1 \le D_N \simeq N^{1-\delta}$ as $N \to \infty$. In the following, we prove that this already guarantees that~\eqref{eq1:feischl} holds with an $N$-independent constant (instead of the constant $D_N$ growing with $N$); see also Lemma~\ref{lemma:summability} below.

	\textbf{Step~3.}
	We show by mathematical induction on $n$ that~\eqref{eq1:feischl} implies
	\begin{equation}\label{eq2:feischl}
		a_{\ell+n}^2
		\le
		\biggl(\prod_{j=1}^n (1-D_j^{-1})\biggr) \, \sum_{{\ell'}=\ell}^{\ell+n} a_{\ell'}^2
		\quad \text{for all } \ell, n \in \mathbb{N}_0.
	\end{equation}
	Note that~\eqref{eq2:feischl} holds for all $\ell \in \mathbb{N}_0$ and $n = 0$ (with the empty product interpreted as $1$). Hence, we may suppose that~\eqref{eq2:feischl} holds for all $\ell \in \mathbb{N}_0$ and up to $n \in \mathbb{N}_0$. Then,
	\begin{align*}
		 & a_{\ell+(n+1)}^2
		=
		a_{(\ell+1)+n}^2
		\stackrel{\eqref{eq2:feischl}}\le
		\biggl(\prod_{j=1}^n (1-D_j^{-1})\biggr) \, \sum_{{\ell'}=\ell+1}^{(\ell+1)+n} a_{\ell'}^2
		= \biggl(\prod_{j=1}^n (1-D_j^{-1})\biggr) \, \bigg( \sum_{{\ell'}=\ell}^{\ell+(n+1)} a_{\ell'}^2 - a_\ell^2 \bigg)
		\\& \qquad
		\stackrel{\mathclap{\eqref{eq1:feischl}}}\le \
		\biggl(\prod_{j=1}^n (1-D_j^{-1})\biggr) \, \bigg( \sum_{{\ell'}=\ell}^{\ell+(n+1)} a_{\ell'}^2 - D_{n+1}^{-1} \sum_{{\ell'}=\ell}^{\ell+(n+1)} a_{\ell'}^2\bigg)
		= \biggl(\prod_{j=1}^{n+1} (1-D_j^{-1})\biggr) \, \sum_{{\ell'}=\ell}^{\ell+(n+1)} a_{\ell'}^2.
	\end{align*}
	This concludes the proof of~\eqref{eq2:feischl}.

	\textbf{Step~4.}
	From~\eqref{eq1:feischl}--\eqref{eq2:feischl}, we infer that
	\begin{equation}\label{eq3:feischl}
		a_{\ell+n}^2
		\le \biggl(\prod_{j=1}^{n} (1-D_j^{-1})\biggr) D_n \, a_\ell^2
		\quad \text{for all } \ell, n \in \mathbb{N}.
	\end{equation}
	Note that
	\begin{equation*}
		M_n \coloneqq \log \bigg[ \biggl(\prod_{j=1}^n (1-D_j^{-1})\biggr) D_n \biggr]
		= \sum_{j=1}^n \log(1-D_j^{-1}) + \log D_n.
	\end{equation*}
	With
	$1 - x \le \exp(-x)$ for all $0 < x < 1$, it follows for $x = D_j^{-1}$ that
	\begin{equation*}
		M_n \le \log D_n - \sum_{j=1}^n D_j^{-1}
		\simeq (1-\delta) \, \log n - \sum_{j=1}^n \frac{1}{j^{1-\delta}}
		\xrightarrow{n \to \infty} -\infty,
	\end{equation*}
	since $\log n \le \sum_{j=1}^n (1/j)$.
	Fix $n_0 \in \mathbb{N}$ such that $M_{n_0} < 0$. It follows from~\eqref{eq3:feischl} that
	\begin{equation}\label{eq4:feischl}
		a_{\ell+in_0}^2
		\le q_0^i \, a_\ell^2
		\quad \text{for all } \ell, i \in \mathbb{N}_0,
		\quad \text{where } 0 < q_0 \coloneqq \exp(M_{n_0}) < 1.
	\end{equation}
	Let $\ell \in \mathbb{N}_0$. For general $n \in \mathbb{N}_0$, choose $i,j \in \mathbb{N}$ with $j < n_0$ such that $n = i n_0 + j$. With~\eqref{eq4:feischl} and quasi-monotonicity~\eqref{eq:a_ell:quasi-mon} of $a_\ell$, we derive
	\begin{equation*}
		a_{\ell+n}^2
		=
		a_{(\ell + j) + in_0}^2
		\stackrel{\eqref{eq4:feischl}}\le
		q_0^i \, a_{\ell + j}^2
		\stackrel{\eqref{eq:a_ell:quasi-mon}}\le
		C_3^2 \, q_0^i \, a_\ell^2
		=
		C_3^2 \, q_0^{-j/n_0} q_0^{n/n_0} a_\ell^2
		\le
		(C_3^2/q_0) \, (q_0^{1/n_0})^n a_\ell^2.
	\end{equation*}
	This completes the proof of~\eqref{eq:Rlinear:convergence} with $C_{\textup{lin}} \coloneqq C_3^2/q_0 > 0$ and $0 < q_{\textup{lin}} \coloneqq q_0^{1/n_0} < 1$.
\end{proof}

\begin{proof}[\textbf{Proof of Lemma~\ref{lemma:summability}}]
	First, observe that $(a_\ell)_{\ell \in \mathbb{N}_0}$ is R-linearly convergent in
	the sense of~(ii) if and only if  $(a_\ell^m)_{\ell \in \mathbb{N}_0}$ is
	R-linearly convergent in the sense of~(ii) with $C_{\textup{lin}}$ replaced by
	$C_{\textup{lin}}^m$ and $q_{\textup{lin}}$ replaced by $q_{\textup{lin}}^m$. Therefore, we
	may restrict to $m = 1$.

	The implication $\textup{(ii)} \Longrightarrow \textup{(i)}$ follows from the geometric
	series, i.e.,
	\begin{equation*}
		\sum_{\ell' = \ell+1}^\infty a_{\ell'}
		\stackrel{\textup{(ii)}}\le
		C a_\ell \sum_{\ell' = \ell+1}^\infty q^{\ell'-\ell} \leq \frac{Cq}{1-q}
		\, a_\ell
		\quad \text{for all } \ell \in \mathbb{N}_0.
	\end{equation*}
	Conversely, (i) yields that
	\begin{equation*}
		(C_1^{-1} + 1) \sum_{\ell' = \ell+1}^\infty a_{\ell'}
		\stackrel{\textup{(i)}}\le
		a_\ell + \sum_{\ell' = \ell+1}^\infty a_{\ell'}
		=
		\sum_{\ell' = \ell}^\infty a_{\ell'}
		\quad \text{for all } \ell \in \mathbb{N}_0.
	\end{equation*}
	Inductively, this leads to
	\begin{equation*}
		a_{\ell+n}
		\le
		\sum_{\ell' = \ell+n}^\infty a_{\ell'}
		\stackrel{\textup{(i)}}\le
		\frac{1}{(C_1^{-1}+1)^n} \, \sum_{\ell' = \ell}^\infty a_{\ell'}
		\stackrel{\textup{(i)}}\le
		\frac{1 + C_1}{ (C_1^{-1}+1)^n} \, a_\ell
		\quad \text{for all } \ell, n \in \mathbb{N}_0.
	\end{equation*}
	This proves~(ii) with $C_{\textup{lin}} \coloneqq 1 + C_1$ and $q_{\textup{lin}} \coloneqq
		(C_1^{-1}+1)^{-1}$.
\end{proof}

\begin{proof}[\textbf{Proof of Lemma~\ref{lem:inexact_contraction}}]
	Let $(\ell,k,\underline{j}) \in \mathcal{Q}$ with $k \ge 1$.
	Contraction of the Zarantonello iteration~\eqref{eq:double:zarantonello}
	proves
	\begin{equation*}
		|\mkern-1.5mu|\mkern-1.5mu| u_\ell^\star - u_\ell^{k,\underline{j}} |\mkern-1.5mu|\mkern-1.5mu|
		\le
		|\mkern-1.5mu|\mkern-1.5mu| u_\ell^\star - u_\ell^{k,\star} |\mkern-1.5mu|\mkern-1.5mu|
		+ |\mkern-1.5mu|\mkern-1.5mu| u_\ell^{k,\star} - u_\ell^{k,\underline{j}} |\mkern-1.5mu|\mkern-1.5mu|
		\stackrel{\eqref{eq:double:zarantonello}}\le
		q_{\textup{sym}}^\star \, |\mkern-1.5mu|\mkern-1.5mu| u_\ell^\star - u_\ell^{k-1,\underline{j}} |\mkern-1.5mu|\mkern-1.5mu|
		+ |\mkern-1.5mu|\mkern-1.5mu| u_\ell^{k,\star} - u_\ell^{k,\underline{j}} |\mkern-1.5mu|\mkern-1.5mu|.
	\end{equation*}
	From the termination criterion of the algebraic
	solver~\eqref{eq:double:termination:alg}, we see that
	\begin{equation*}
		|\mkern-1.5mu|\mkern-1.5mu| u_\ell^{k,\star} - u_\ell^{k,\underline{j}} |\mkern-1.5mu|\mkern-1.5mu|
		\le
		\frac{q_{\textup{alg}}}{1-q_{\textup{alg}}} \, |\mkern-1.5mu|\mkern-1.5mu| u_\ell^{k,\underline{j}} - u_\ell^{k,\underline{j}-1} |\mkern-1.5mu|\mkern-1.5mu|
		\stackrel{\eqref{eq:double:termination:alg}}\le
		\frac{q_{\textup{alg}}}{1-q_{\textup{alg}}} \, \lambda_{\textup{alg}} \big[ \lambda_{\textup{sym}}
			\eta_\ell(u_\ell^{k,\underline{j}}) + |\mkern-1.5mu|\mkern-1.5mu| u_\ell^{k,\underline{j}} - u_\ell^{k-1,\underline{j}} |\mkern-1.5mu|\mkern-1.5mu|
			\big].
	\end{equation*}
	With the termination criterion of the inexact Zarantonello
	iteration~\eqref{eq:double:termination:sym}, it follows that
	\begin{equation*}
		|\mkern-1.5mu|\mkern-1.5mu| u_\ell^{k,\star} - u_\ell^{k,\underline{j}} |\mkern-1.5mu|\mkern-1.5mu|
		\stackrel{\eqref{eq:double:termination:sym}}\le
		\frac{2 \, q_{\textup{alg}}}{1-q_{\textup{alg}}} \, \lambda_{\textup{alg}}
		\begin{cases}
			\lambda_{\textup{sym}} \eta_\ell(u_\ell^{k,\underline{j}})                                                    & \text{for } k = \underline{k}[\ell], \\
			|\mkern-1.5mu|\mkern-1.5mu| u_\ell^{k,\underline{j}} - u_\ell^{k-1,\underline{j}} |\mkern-1.5mu|\mkern-1.5mu| & \text{for } 1 \le k <
			\underline{k}[\ell].
		\end{cases}
	\end{equation*}
	For $k = \underline{k}[\ell]$, the preceding estimates
	prove~\eqref{eq2:zarantonello:inexact}.
	For $k < \underline{k}[\ell]$, it follows that
	\begin{equation*}
		|\mkern-1.5mu|\mkern-1.5mu| u_\ell^\star - u_\ell^{k,\underline{j}} |\mkern-1.5mu|\mkern-1.5mu|
		\le
		q_{\textup{sym}}^\star \, |\mkern-1.5mu|\mkern-1.5mu| u_\ell^\star - u_\ell^{k-1,\underline{j}} |\mkern-1.5mu|\mkern-1.5mu|
		+ \frac{2 \, q_{\textup{alg}}}{1-q_{\textup{alg}}} \, \lambda_{\textup{alg}} \,
		\big[ |\mkern-1.5mu|\mkern-1.5mu| u_\ell^\star - u_\ell^{k,\underline{j}} |\mkern-1.5mu|\mkern-1.5mu| + |\mkern-1.5mu|\mkern-1.5mu| u_\ell^\star -
			u_\ell^{k-1,\underline{j}} |\mkern-1.5mu|\mkern-1.5mu| \big].
	\end{equation*}
	Provided that $\frac{2 \, q_{\textup{alg}}}{1-q_{\textup{alg}}} \, \lambda_{\textup{alg}} < 1$, this proves
	\begin{equation*}
		|\mkern-1.5mu|\mkern-1.5mu| u_\ell^\star - u_\ell^{k,\underline{j}} |\mkern-1.5mu|\mkern-1.5mu|
		\le
		\frac{q_{\textup{sym}}^\star + \frac{2 \, q_{\textup{alg}}}{1-q_{\textup{alg}}} \, \lambda_{\textup{alg}}}{1 - \frac{2 \,
		q_{\textup{alg}}}{1-q_{\textup{alg}}} \, \lambda_{\textup{alg}}} \, |\mkern-1.5mu|\mkern-1.5mu| u_\ell^\star -
		u_\ell^{k-1,\underline{j}} |\mkern-1.5mu|\mkern-1.5mu|
		\stackrel{\eqref{eq:double:assumption:lambda}}=
		q_{\textup{sym}} \, |\mkern-1.5mu|\mkern-1.5mu| u_\ell^\star - u_\ell^{k-1,\underline{j}} |\mkern-1.5mu|\mkern-1.5mu|,
	\end{equation*}
	which is \eqref{eq1:zarantonello:inexact}.
	This concludes the proof.
\end{proof}

\end{document}